\documentclass[11pt]{article}
\usepackage[T1]{fontenc}
\usepackage[margin=1in]{geometry}
\usepackage[utf8]{inputenc}

\usepackage{microtype}
\usepackage{graphicx}
\usepackage{subfig}
\usepackage{caption}
\usepackage{booktabs} %

\usepackage{amsmath,amsfonts,amssymb}
\usepackage{hyperref}
\hypersetup{
    colorlinks=true,
    linkcolor=blue,
    citecolor=red,
    filecolor=magenta,      
    urlcolor=cyan,
   }      
\usepackage[capitalise]{cleveref}
\usepackage{epsfig}
\usepackage[algo2e,linesnumbered,ruled,vlined,resetcount]{algorithm2e}
\RestyleAlgo{ruled}

\SetCommentSty{mycommfont}
\usepackage{array}
\usepackage{epstopdf}
\usepackage{cite}
\usepackage{footmisc}
\usepackage{mathtools}
\usepackage{bm}
\usepackage[utf8]{inputenc}
\usepackage{dsfont}
\usepackage{subfig}
\usepackage{caption}

\title{Online Actuator Selection and Controller Design for Linear Quadratic Regulation with Unknown System Model}
\author{%
Lintao Ye%
\thanks{School of Artificial Intelligence and Automation, Huazhong University of Science and Technology, Wuhan, China; \texttt{\{yelintao93,chiming,zwliu\}@hust.edu.cn}.}
\and
Ming Chi% 
\footnotemark[1]
\and
Zhi-Wei Liu%
\footnotemark[1]
\and
Vijay Gupta%
\thanks{The Elmore Family School of Electrical and Computer Engineering, Purdue University, IN, USA. Email: \texttt{gupta869@purdue.edu}.}
}

\RequirePackage{mathtools}
\RequirePackage{dsfont}
\RequirePackage{delimset}

\usepackage{amsthm}
\usepackage{amssymb}
\usepackage{thmtools}
\usepackage{footnote}

\DeclareMathOperator*{\argmin}{arg\,min}

\newtheorem{theorem}{Theorem}

\newtheorem{lemma}{Lemma}
\newtheorem{remark}{Remark}
\newtheorem{definition}{Definition}
\newtheorem{proposition}{Proposition}

\newtheorem{assumption}{Assumption}

\newcommand{\C}{\mathcal{C}}
\newcommand{\A}{\mathcal{A}}
\newcommand{\Q}{\mathcal{Q}}

\newcommand{\B}{\mathcal{B}}
\newcommand{\CN}{\mathcal{N}}
\newcommand{\R}{\mathbb{R}}
\newcommand{\BS}{\mathbb{S}}
\newcommand{\CS}{\mathcal{S}}
\newcommand{\G}{\mathcal{G}}
\newcommand{\K}{\mathcal{K}}
\newcommand{\Z}{\mathbb{Z}}
\newcommand{\E}{\mathbb{E}}
\newcommand{\F}{\mathcal{F}}
\newcommand{\CE}{\mathcal{E}}
\newcommand{\Prob}{\mathbb{P}}

\newcommand{\N}{\mathcal{N}}

\newcommand{\Tr}{\text{Tr}}
\newcommand{\rank}{\text{rank}}

\usepackage{enumitem}
\usepackage[numbers]{natbib}
\usepackage{crossreftools}
\usepackage{algorithm}
\usepackage[noend]{algpseudocode}
\algnewcommand{\LineComment}[1]{\State \(\triangleright\) #1}

\algnewcommand{\IfThenElse}[3]{%
  \State \algorithmicif\ #1\ \algorithmicthen\ #2\ \algorithmicelse\ #3}

\begin{document}

\maketitle

\begin{abstract}
We study the simultaneous actuator selection and controller design problem for linear quadratic regulation {\color{black} with Gaussian noise over a finite horizon of length $T$} and unknown system model. {\color{black} We consider both episodic and non-episodic settings of the problem and propose online algorithms that specify} both the sets of actuators to be utilized under a cardinality constraint and the controls corresponding to the sets of selected actuators. In the episodic setting, the interaction with the system breaks into {\color{black} $N$ episodes, each of which restarts from a given initial condition and has length $T$}. In the non-episodic setting, the interaction goes on continuously. Our online algorithms leverage a multiarmed bandit algorithm to select the sets of actuators and a certainty equivalence approach to design the corresponding controls. {\color{black} We show that our online algorithms yield $\sqrt{N}$-regret for the episodic setting and $T^{2/3}$-regret for the non-episodic setting}. We extend our algorithm design and analysis to {\color{black} show scalability with respect to} %efficiently handle instances of the problem when 
both the total number of candidate actuators and the cardinality constraint. % scale large. 
We numerically validate our theoretical results.
\end{abstract}

\section{Introduction}\label{sec:introduction}
In large-scale control system design, the number of actuators (or sensors) that can be installed is often limited by budget or complexity constraints. The problem of selecting a subset of all the candidate actuators (or sensors), in order to optimize a system objective while satisfying a budget constraint is a classic problem referred to as actuator (or sensor) selection \cite{van2001review,gupta2006stochastic,olshevsky2014minimal,summers2016submodularity,ye2020resilient,siami2020separation,tzoumas2020lqg,ye2020complexity}. However, most of the existing works on this problem assume the knowledge of the system model when designing the actuator (or sensor) selection algorithms. In this work, we are interested in the situation when the system model is unknown \cite{hou2013model}. In such a case, the existing algorithms for the actuator selection problem do not apply.

We study the simultaneous actuator selection and controller design problem for Linear Quadratic Regulation (LQR) \cite{anderson2007optimal}. The goal is to select a sequence of sets of actuators each with a cardinality constraint, while minimizing the accumulative quadratic cost over a time horizon. We assume that the system model is unknown, and conside solving the problem by interacting with the system in an online manner. We study two settings of the problem: episodic and non-episodic settings. In the episodic setting, the interaction with the system breaks into subsequences, each of which starts from a given initial condition and ends at a terminal time step. In the non-episodic setting, the interaction with the system goes on continuously. Both the episodic and non-episodic settings are widely studied in general reinforcement learning problems, and capture different scenarios in practice \cite{sutton2018reinforcement,jin2020provably}. We provide online algorithms to solve the problem, and leverage the notion of regret \cite{bubeck2011introduction,arora2012online,hazan2016introduction} to characterize their performance. %Our online algorithms utilize the exploration-and-exploitation scheme for reinforcement learning problems \cite{sutton2018reinforcement}. Specifically, the algorithms first explore to obtain estimates of the system matrices based on data samples from the system trajectory, and then exploit the obtained estimates to select the sets of actuators and design the corresponding controls. %For the episodic (resp., non-episodic) setting, the online algorithm selects a set of actuators for each episode (i.e., time step) in the problem. 

%A major challenge in our problem is that in order to obtain a solution to the problem, the online actuator selection algorithms need to specify both the set of selected actuators and the corresponding control policy. 

%Since the system matrices are not known, it is not possible to directly use the well-known formula for the optimal LQR control policy \cite{bertsekas2017dynamic}, given a set of selected actuators. In order to tackle this challenge, the online actuator selection algorithms that we propose and the corresponding analysis combines ideas from (model-based) learning algorithms for LQR \cite{dean2020sample,cassel2020logarithmic} and online algorithms for the multiarmed bandit problem \cite{auer2002nonstochastic}. Specifically, the online actuator selection algorithms that we propose first obtain estimates of the system matrices from the system trajectory. Then, the online algorithms use a multiarmed bandit algorithm as a subroutine to select the set of actuators, and determine the corresponding control policy based on the estimates of the system matrices. By finding an appropriate frequency of these episodes, we show that the online actuator selection algorithm achieves a regret of $\tilde{O}(\sqrt{T})$, where $T$ is the number of episodes in the actuator selection problem and $\tilde{O}(\cdot)$ hides logarithmic factors in $T$.

\subsection*{Related Work}
Actuator (or sensor) selection in control systems has been studied in the literature extensively. Since the problem is NP-hard \cite{ye2020complexity}, much work in the literature provides approximation algorithms to solve the problem with performance guarantees \cite{olshevsky2014minimal,tzoumas2015minimal,guo2020actuator}. However, most of the previous work assumes known system model. Exceptions are \cite{golovin2010online,streeter2008online}, where the authors studied an online sensor selection problem for the estimation of a static random variable. %The goal is to select a subset of all the candidate sensors in order to minimize the estimation error of the static random variable using sensor measurements, where the estimation error associated with a set of selected sensors is revealed after the sensors are selected to provide measurements. %The authors considered an episodic setting to the one that we described above and proposed an online sensor selection algorithm with regret analysis. 
Another related work is \cite{fotiadis2021learning}, where the authors considered an unknown continuous-time linear time-invariant system without stochastic noise and studied the problem of selecting a subset of actuators under a cardinality constraint such that a controllability metric of the system is optimized. The authors proposed an online actuator selection algorithm and showed that the algorithm will select the optimal set of actuators after a finite number of time steps.

The LQR problem with unknown system matrices, also known as the optimal adaptive control problem, has been studied for decades \cite{kumar1983optimal,campi1998adaptive,chen1991identification,aastrom2013adaptive,faradonbeh2020optimism}. One standard approach is the certainty equivalence approach which first estimates the system matrices from system trajectories and then use the estimate of the system matrices to design the control as if the true system matrices are used. Thus, it is crucial to ensure the consistency of the estimate in order to achieve the optimal performance. Based on the consistent estimates returned by least squares as shown by \cite{lai1982least}, \cite{chen1987optimal} designed a certainty equivalent controller with an additive random perturbation. In \cite{kumar1983optimal} and \cite{campi1998adaptive}, a reward-biased estimate of the system matrices is utilized. {\color{black} In \cite{faradonbeh2020adaptive},  randomly perturbed least squares and Thompson sampling were used to obtain the estimates to design the certainty equivalent controller and a regret of nearly square-root growth rate was established.} However, the aforementioned works focused on the asymptotic performance of the certainty equivalent controller (as the number of data samples from the system trajectories used for estimating the system matrices goes to infinity). The  finite-sample analyses of the certainty equivalence approach have been studied for learning LQR \cite{dean2020sample,mania2019certainty,cassel2020logarithmic,faradonbeh2020input,ziemann2022regret}. {\color{black}In particular, it was shown in \cite{faradonbeh2020input,mania2019certainty} that the certainty equivalent controller with a certain additive random perturbation achieves a regret of $\tilde{O}(\sqrt{T})$, where $T$ is the number of time steps in the LQR problem and $\tilde{O}(\cdot)$ hides logarithmic factors in $T$.} Moreover, \cite{mete2022augmented} analyzed the regret of the certainty equivalence approach based on the reward-biased estimate.

\subsection*{Contributions}
We summarize the technical challenges and our contributions in this paper. First, we formulate the simultaneous actuator selection and controller design problem for LQR with unknown system model. This problem is challenging since it contains both discrete and continuous variables (i.e., the sets of actuators and the corresponding controls, respectively). The online algorithms that we propose to solve the problem contain two phases. The online algorithms first estimate the system matrices based on the data samples from a single system trajectory. Based on the estimated system matrices, the online algorithms then leverage a multiarmed bandit algorithm \cite{auer2002nonstochastic} to select the set of actuators, and leverage the certainty equivalence approach \cite{mania2019certainty} for the controller design. We carefully balance the length of the two phases to achieve a desired performance of the online algorithms. When characterizing the regret of the online algorithms, we identify a proper decomposition of the regret, and upper bound each term in the regret decomposition. %Moreover, we consider the general dynamic regret of online algorithms \cite{arora2012online}, which allows us to compare the performance of our algorithms to general benchmarks with a possibly time-varying actuator selection (i.e., different sets of actuators can be selected for different time steps). 

Second, since we consider the actuator selection problem for finite-horizon LQR, we extend the analysis and results for the certainty equivalence approach proposed for learning infinite-horizon LQR (without the actuator selection component) \cite{mania2019certainty,cassel2020logarithmic} to the finite-horizon setting. The analysis for the finite-horizon setting is more challenging, since the optimal controller for finite-horizon LQR is time-varying in general, while the optimal controller for infinite LQR is time-invariant \cite{bertsekas2017dynamic}.

Third, we provide a comprehensive study of the problem by considering both the online episodic and non-episodic settings. {\color{black} The non-episodic setting is more challenging than the episodic setting, since the system state cannot be reset to a given initial condition after each episode. However, we show that given a non-episodic instance of the problem, one can first construct a corresponding episodic instance and then apply the proposed online algorithms.} We show that our online algorithm for the episodic setting yields a regret of $\tilde{O}(\sqrt{T^2N})$, where $N$ is the number of episodes and $T$ is the number of time steps in each episode. For the non-episodic setting, the online algorithm yields a regret of $\tilde{O}(T^{2/3})$, where $T$ is the horizon length. %Moreover, we show that our online algorithms are efficient (in time and memory) for instances of the problem with either of the total number of candidate actuators or the cardinality constraint to be relatively small (or bounded by a constant). 

Finally, we extend our analysis to efficiently handle instances of the problem when both the total number of candidate actuators and the cardinality constraint scale large. Since the (offline) actuator selection problem for LQR with known system model is NP-hard \cite{ye2020complexity}, we leverage a weaker notion of regret, i.e., $c$-regret, introduced for online algorithms for combinatorial optimization problems \cite{streeter2008online,chen2020black,salazar2021differentially}, and characterize the performance of the online algorithm that we propose for the large-scale problem instances. We show that the $c$-regret of our online algorithm scales as $\tilde{O}(TN^{2/3})$ (resp., $\tilde{O}(T^{3/4})$) in the episodic (resp., non-episodic) setting, where $c\in(0,1)$ is parameterized by the problem parameters.

%%%%%%%%%%%%%%%%%%%%%%%%%%%%%%%%%%%%%%%%%%%%%%%%%%%%%%%%%%%%%%%%%%%%%%%%%%%%%%%%

\subsection*{Notation and terminology}
The sets of integers and real numbers are denoted as $\mathbb{Z}$ and $\mathbb{R}$, respectively. The set of integers (resp., real numbers) that are greater than or equal to $a\in\mathbb{R}$ is denoted as $\mathbb{Z}_{\ge a}$ (resp., $\R_{\ge a}$). For a real number $a$, let $\lceil a \rceil$ be the smallest integer that is greater than or equal to $a$. The space of $n$-dimensional real vectors is denoted by $\mathbb{R}^{n}$, and the space of $m\times n$ real matrices is denoted by $\mathbb{R}^{m\times n}$. For a matrix $P\in\R^{n\times n}$, let $P^{\top}$, $\Tr(P)$, and $\{\sigma_i(P):i\in\{1,\dots,n\}\}$ be its transpose, trace, and set of singular values, respectively. Without loss of generality, let the singular values of $P$ be ordered as $\sigma_1(P)\ge\cdots\ge\sigma_n(P)$. Let $\norm{\cdot}$ denote the $\ell_2$ norm, i.e., $\norm{P}=\sigma_1(P)$ for a matrix $P\in\R^{n\times n}$, and $\norm{x}=\sqrt{x^{\top}x}$ for a vector $x\in\R^n$. Let $\norm{P}_F=\sqrt{\Tr(PP^{\top})}$ be the Frobenius norm of $P\in\R^{n\times m}$. A positive semidefinite matrix $P$ is denoted by $P\succeq0$, and $P\succeq Q$ if and only if $P-Q\succeq0$. Let $\BS_+^n$ (resp., $\BS_{++}^n$) denote the set of $n\times n$ positive semidefinite (resp., positive definite) matrices. Let $I$ be an identity matrix whose dimension can be inferred from the context. Given any integer $n\ge1$, define $[n]=\{1,\dots,n\}$. The cardinality of a finite set $\mathcal{A}$ is denoted by $|\mathcal{A}|$. Let $\mathds{1}\{\cdot\}$ denote an indicator function.

\section{Problem Formulation and Preliminaries}\label{sec:problem formulation}
\subsection{Problem Formulation}\label{sec:LQR}
Consider a discrete-time linear time-invariant system
\begin{equation}
\label{eqn:LTI}
x_{t+1} = Ax_t + Bu_t + w_t,
\end{equation} 
where $A\in\R^{n\times n}$ is the system dynamics matrix, $x_t\in\R^n$ is the state vector, $B\in\R^{n\times m}$ is the input matrix, $u_t\in\R^m$ is the control, and $\{w_t\}_{t\ge0}$ are i.i.d. noise terms with zero mean and covariance $W$ for all $t\in\Z_{\ge0}$. Let $\G$ be the set that contains all the candidate actuators. Denote $B=\begin{bmatrix}B_1 & \cdots & B_{|\G|}\end{bmatrix}$, where $B_i\in\R^{n\times m_i}$ for all $i\in\G$ with $\sum_{i\in\G}m_i=m$. For any $i\in\G$, $B_i$ corresponds to a candidate actuator that can be potentially selected and installed. At each time step $t\in\Z_{\ge0}$, only a subset of actuators out of all the candidate actuators is selected to provide controls to system~\eqref{eqn:LTI}, due to, e.g., budget constraints. For any $t\in\Z_{\ge0}$, let $\CS_t\subseteq\G$ denote the set of actuators selected for time step $t$, let $B_{\CS_t}\triangleq\begin{bmatrix}B_{i_1} & \cdots & B_{i_{|\CS_t|}}\end{bmatrix}$ be the input matrix associated with the actuators in $\CS_t$, and let $u_{t,\CS_t}\triangleq\begin{bmatrix}u_{t,i_1}^{\top} & \cdots & u_{t,i_{|\CS_t|}}^{\top}\end{bmatrix}^{\top}$ be the control provided by the actuators in $\CS_t$, where $\CS_t=\{i_1,\dots,i_{|\CS_t|}\}$. Given a horizon length $T\in\Z_{\ge1}$ and an actuator selection $\CS\triangleq(\CS_0,\dots,\CS_{T-1})$, we consider the following quadratic cost:
\begin{equation}
\label{eqn:cost}
J(\CS,u_{\CS})\triangleq\Big(\sum_{t=0}^{T-1}x_t^{\top}Q x_t+u_{t,\CS_t}^{\top}R_{\CS_t}u_{t,\CS_t}\Big)+x_T^{\top}Q_f x_T,
\end{equation}
where $u_{\CS}\triangleq(u_{0,\CS_0},\dots,u_{T-1,\CS_{T-1}})$, $Q,Q_f\in\BS_{+}^n,R\in\BS_{++}^m$ are the cost matrices, and $R_{\CS_t}\in\BS_{++}^{m_{\CS_t}}$ (with $m_{\CS_t}=\sum_{i\in \CS_t}m_i$) is a submatrix of $R$ corresponding to the set $\CS_t$.\footnote{In other words, the matrix $R_{\CS_t}$ is obtained by deleting the rows and columns of $R$ indexed by the elements in the set $\G\setminus \CS_t$.}

Given system~\eqref{eqn:LTI} and $T\in\Z_{\ge1}$, our goal is to solve the following actuator selection problem for LQR:
\begin{equation}
\begin{split}
\label{eqn:LQR obj}
&\min_{\CS,u_{\CS}}\E\big[J(\CS,u_{\CS})\big],\\
&s.t.\ \CS_t\subseteq\G,|\CS_t|= H,\ \forall t\in\{0,\dots,T-1\},
\end{split}
\end{equation}
where $H\in\Z_{\ge1}$ is a cardinality constraint on the sets of selected actuators, and the expectation is taken with respective to $w_0,\dots,w_{T-1}$. {\color{black} Conditioning on an actuator selection $\CS=(\CS_0,\dots,\CS_{T-1})$, it is well-known that the corresponding optimal control, i.e., $\tilde{u}_{\CS}\in\arg\min_{u_{\CS}}\E[J(\CS,u_{\CS})]$, is given by a linear state-feedback controller \cite[Chapter~3]{bertsekas2017dynamic}}
\begin{equation} 
\label{eqn:optimal control law}
\tilde{u}_{t,\CS_t} = K_{t,\CS}x_t,\ \forall t\in\{0,1,\dots,T-1\},
\end{equation}
where the control gain matrix $K_{t,\CS}\in\R^{m_{\CS_t}\times n}$, with $m_{\CS_t}=\sum_{i\in\CS_t}m_i$, is given by
\begin{equation}
\label{eqn:control gain}
K_{t,\CS} = -(B^{\top}_{\CS}P_{t+1,\CS}B_{\CS} + R_{\CS})^{-1}B^{\top}_{\CS_t}P_{t+1,\CS}A,
\end{equation}
where $P_{t,\CS}\in\BS_+^n$ is given recursively by the following Discrete Algebraic Riccati Equation (DARE):
\begin{equation}
\label{eqn:recursion for P_k}
P_{t,\CS} = A^{\top}P_{t+1,\CS}A-A^{\top}P_{t+1,\CS}B_{\CS_t}(B^{\top}_{\CS_t}P_{t+1,\CS}B_{\CS_t}+R_{\CS_t})^{-1}B^{\top}_{\CS_t}P_{t+1,\CS}A+Q
\end{equation}
initialized with $P_{T,\CS}=Q_f$. Moreover, {\color{black} conditioning on an actuator selection $\CS$}, we know that \cite[Chapter~3]{bertsekas2017dynamic}
\begin{align}\nonumber
J(\CS)&\triangleq\min_{u_{\CS}}\E\big[J(\CS,u_{\CS})\big]=\E\big[J(\CS,\tilde{u}_{\CS})\big]\\
&=\E[x_0^{\top} P_{0,\CS}x_0]+\sum_{t=0}^{T-1}\Tr(P_{t+1,\CS}W).\label{eqn:opt LQR cost S_t}
\end{align}
Thus, supposing the system matrices are known, we see that solving Problem~\eqref{eqn:LQR obj} is equivalent to solving 
\begin{equation}
\label{eqn:LQR obj 2nd}
\begin{split}
&\min_{\CS} J(\CS)\\
&s.t.\ \CS_t\subseteq\G,|\CS_t|=H,\ \forall t\in\{0,\dots,T-1\}.
\end{split}
\end{equation}
When the system matrices $A$ and $B$ are unknown, Eqs.~\eqref{eqn:optimal control law}-\eqref{eqn:recursion for P_k} cannot be directly used to design the control $u_{\CS}$ {\color{black} conditioning on an actuator selection $\CS=(\CS_1,\dots,\CS_{T-1})$.} {\color{black}We now define and solve both the episodic and non-episodic settings of Problem~\eqref{eqn:LQR obj}. In the sequel, we use superscript $k$ to index an episode and subscript $t$ to index a time step.}

%Note that we consider finite-horizon LQR in Problem~(3). The reason is that we allow time-varying actuator selections, i.e., different sets of actuators can be selected for different time steps, so that obtaining an optimal (time-varying) actuator selection becomes intractable if the horizon length goes to infinity.  %Thus, we aim to propose online algorithms to solve Problem~\eqref{eqn:LQR obj}, which {\it simultaneously} learn the actuator selection strategy and the corresponding controller design by interacting with system~\eqref{eqn:LTI}.

%
\subsection{Online Algorithm for Episodic Setting}\label{sec:online AS}
In the episodic setting, system~\eqref{eqn:LTI} starts from an initial condition at the beginning of each episode, and we aim to obtain a solution to Problem~\eqref{eqn:LQR obj} by interacting with system~\eqref{eqn:LTI} for a set of episodes. Specifically, let $N\in\Z_{\ge1}$ be the total number of episodes and let {\color{black} $T$} %$T_k$ 
be the time horizon length of any episode $k\in[N]$. %We assume for simplicity that $T_k=T$ for all $k\in[N]$. 
Considering any $k\in[N]$, the dynamics of system~\eqref{eqn:LTI} in episode $k$ is given by
\begin{align}
x^{k}_{t+1} = Ax^{k}_t + B_{\CS^{k}_t} u_{t,\CS^{k}_t}^k + w^{k}_t,\label{eqn:LTI with B_i}
\end{align}
where $x^{k}_t$, $u_{t,\CS^k}^k$ and $w^{k}_t$ are the state, control and noise at time step $t$ in episode $k$, respectively, and $\CS^k=(\CS^k_0,\dots,\CS^k_{T-1})$ with $S^k_t$ to be the set of actuators selected for time step $t$, for all $t\in\{0,\dots,T-1\}$. We assume that $\{w_t^k\}_{t=0}^{T-1}$ are i.i.d with $\E[w_t^k]=0$ and $\E[w^{k}_tw^{k\top}_t]=W$ for all $t\in\{0,1,\dots,T-1\}$ and for all $k\in[N]$. {\color{black} We also assume for simplicity that $x_0^k=0$ for all $k\in[N]$.} %\footnote{Our analysis can be extended to the case when $x^{k}_0\sim\CN(0,\Sigma_0)$, as one may view $x^{k}_0$ as $w^{k}_{-1}$ in the analysis.} 
In this work, we focus on the scenario with $\CS_0^k=\cdots=\CS_{T-1}^k$ for all $k\in[K]$, i.e., the set of selected actuators in each episode is fixed during that episode. Slightly abusing the notation, we simply denote the set of selected actuators for episode $k$ as $\CS^k\subseteq\G$. 

Now, similarly to Eq.~\eqref{eqn:cost}, for any $k\in[N]$ we define the following quadratic cost of episode $k$ when the set of actuators $\CS^k\subseteq\G$ is selected to provide $u_{t,\CS^k}$ for all $t\in\{0,\dots,T-1\}$:
\begin{equation}
\label{eqn:episode t cost}
J_k(\CS^k,u_{\CS^k}^k)=\Big(\sum_{t=0}^{T-1}x_t^{k\top}Q^k x^{k}_t+u_{t,\CS^k}^{k\top}R_{\CS^k}^{k}u_{t,\CS^k}^k\Big)+x_T^{k\top}Q_f^{k} x_T^{k},
\end{equation}
where $u_{\CS^k}^k=(u_{0,\CS^k}^k,\dots,u_{T-1,\CS^k}^k)$, $Q^k,Q^k_f\in\BS_{+}^n,R\in\BS_{++}^m$ are the cost matrices, and $R_{\CS^k}^{k}\in\BS_{++}^{m_{\CS^k}}$ (with $m_{\CS^k}=\sum_{i\in \CS^k}m_i$) is a submatrix of $R^k$ corresponding to the set $\CS^k$. Note that we allow different cost matrices across the episodes. %Also note that the cost function of episode $k\in[N]$, i.e., $J_k(\cdot)$, depends on both the actuators selected in episode $k$ and the controls provided by the selected actuators in episode $k$. 
%Throughout this work, 
We assume that $Q^k$, $Q_f^k$ and $R^k$ are known for all $k\in[N]$. 

An online algorithm for the episodic setting of Problem~\eqref{eqn:LQR obj} then works as follows: At the beginning of each episode $k\in[N]$, the algorithm selects a set $S^k\subseteq\G$ (with $|S^k|= H$) of actuators and designs the control $u_{\CS^k}=(u_{0,\CS^k},\dots,u_{T-1,\CS^k})$ provided by the actuators in $\CS^k$. Note that when making the decisions at the beginning of each episode $k\in[N]$, the following information is available to the online algorithm: (a) the system state trajectories $x^1,\dots,x^{k-1}$, where $x^{k^{\prime}}\triangleq(x^{k^{\prime}}_0,\dots,x^{k^{\prime}}_{T-1})$ for all $k^{\prime}\in[k-1]$; and (b) previous decisions made by the algorithm, i.e., $\CS^1,\dots,\CS^{k-1}$ and $u_{\CS^1},\dots,u_{\CS^{k-1}}$. %In order to measure the performance of the online algorithm, we introduce the notion of regret which is a typical performance metric of online algorithms \cite{bubeck2011introduction}. Specifically, 
Since $Q^k,Q_f^k,R^k$ are assumed to be known, the costs $J_{k^{\prime}}(\CS^{k^{\prime}},u_{\CS^{k^{\prime}}}^{k^{\prime}})$ $\forall k^{\prime}\in[k-1]$ are also given at the beginning of any episode $k\in[N]$. Thus, the information setting discussed above corresponds to the bandit information setting in online optimization literature (see., e.g., \cite{hazan2016introduction}). In order to characterize the performance of such an online algorithm, denoted as $\A_e$, for Problem~\eqref{eqn:LQR obj} in the episodic setting, we aim to minimize the following regret of $\A_e$:
\begin{equation}
\label{eqn:regret of A_e}
R_{\A_e}\triangleq\E_{\A_e}\Big[\sum_{k=1}^N J_k(\CS^k,u_{\CS^k}^k)\Big]-\sum_{k=1}^N J_k(\CS^{k}_{\star}),
\end{equation} 
where $\E_{\A_e}[\cdot]$ denotes the expectation with respect to the randomness of the algorithm, $J_k(\CS^{k}_{\star})$ is defined as \eqref{eqn:opt LQR cost S_t} and $\CS^{k}_{\star}$ is an optimal solution to \eqref{eqn:LQR obj 2nd} (with cost matrices $Q^k,R^k,Q_f^k$ and an extra constraint $\CS_0=\cdots=\CS_{T-1}$), for all $k\in[N]$. %Note that for any $k\in[N]$, $J_k(\CS^{k}_{\star})$ is the minimum expected cost of episode $k$. 
\begin{remark}
\label{remark:dynamic regret}
Note that $R_{\A_e}$ compares the cost incurred by the online algorithm to the benchmark given by the minimum achievable cost of Problem~\eqref{eqn:LQR obj} in the episodic setting. 
Since $\CS_{\star}^k$ can potentially be different across the episodes in the benchmark in Eq.~\eqref{eqn:regret of A_e}, $R_{\A_e}$ is a dynamic regret \cite{auer2002nonstochastic,zinkevich2003online,arora2012online}. In fact, one can consider any sets $\CS^1_{\star},\dots,\CS^N_{\star}\subseteq\G$ with $|\CS^k_{\star}|= H$ for all $k\in[N]$ as the benchmark in Eq.~\eqref{eqn:regret of A_e}. For any $\CS_{\star}=(\CS_{\star}^1,\dots,\CS_{\star}^N)$, define
\begin{equation}
\label{eqn:number of switchings}
h((\CS^1_{\star},\dots,\CS^N_{\star}))=1+|\{1\le\ell<N-1:\CS^{\ell}_{\star}\neq\CS^{\ell+1}_{\star}\}|.
\end{equation}
Our regret bound for $R_{\A_e}$ %shown in this paper 
holds for a general benchmark $\CS_{\star}=(\CS_{\star}^1,\dots,\CS_{\star}^N)$, where $\CS_{\star}^k$ is any $\CS_{\star}^k\subseteq\G$ with $|\CS_{\star}^k|=H$. If we consider benchmark $\CS_{\star}$ with $h(\CS_{\star})=1$, i.e., $\CS_{\star}^1=\cdots=\CS_{\star}^N$, Eq.~\eqref{eqn:regret of A_e} reduces to a static regret \cite{auer2002nonstochastic,bubeck2011introduction}. 
\end{remark}

\subsection{Online Algorithm for Non-episodic Setting}\label{sec:online AS non-episodic}
In the non-episodic (i.e., continuous) setting, we %aim to solve Problem~\eqref{eqn:LQR obj} by continuously interacting 
interact with system~\eqref{eqn:LTI} over a time horizon of length $T\in\Z_{\ge1}$, where the system is not reset to the initial condition $x_0=0$ during the interaction. At the beginning of each time step $t\in\{0,\dots,T-1\}$, the algorithm %an online algorithm for the non-episodic setting of Problem~\eqref{eqn:LQR obj} 
selects a set $\CS_t\subseteq\G$ (with $|\CS_t|= H$) of actuators and designs the corresponding control $u_{t,\CS_t}$. When making the decisions at the beginning of each time step $t\in\{0,\dots,T-1\}$, the information available to the online algorithm includes: (a) $x_0,\dots,x_{t-1}$; and (b) $\CS_0,\dots,\CS_{t-1}$ and $u_{0,\CS_0},\dots,u_{t-1,\CS_{t-1}}$. %Let $\A_c$ denote an online algorithm Problem~\eqref{eqn:LQR obj} in the non-episodic setting. Denote $\CS_{0:t}=(\CS_0,\dots,\CS_t)$ and $u_{\CS_{0:t}}=(u_{0,\CS_0},\dots,u_{t,\CS_t})$ for all $t\in\{0,\dots,T-1\}$. 
Similar to our arguments above, the information setting corresponds to the bandit setting in the online optimization literature. Denote $\CS_{0:t}\triangleq(\CS_0,\dots,\CS_t)$ and $u_{\CS_{0:t}}\triangleq(u_{0,\CS_0},\dots,u_{t,\CS_t})$ for all $t\in\{0,\dots,T-1\}$. To characterize the performance of such an online algorithm, denoted as $\A_c$, we minimize the following regret of $\A_c$:
\begin{equation}
\label{eqn:regret of A_c}
R_{\A_c}\triangleq\E_{\A_c}\Big[\sum_{t=0}^{T-1}c_t(\CS_{0:t},u_{\CS_{0:t}})\Big]-J(\CS^{\star}),
\end{equation}
where $\E_{\A_c}[\cdot]$ denotes the expectation with respect to the randomness of the algorithm, and the benchmark $J(\CS^{\star})$ is defined as \eqref{eqn:opt LQR cost S_t} with $\CS^{\star}=(\CS_{0}^{\star},\dots,\CS_{T-1}^{\star})$ to be an optimal solution to \eqref{eqn:LQR obj 2nd}. Note that in Eq.~\eqref{eqn:regret of A_c} we denote,
\begin{equation}
\label{eqn:cost per time step}
c_t(\CS_{0:t},u_{\CS_{0:t}})=x_t^{\top}Qx_t + u_{t,\CS_t}^{\top}R_{\CS_t}u_{t,\CS_t},
\end{equation}
for all $t\in\{0,\dots,T-2\}$, and 
\begin{equation}
\label{eqn:final time step cost}
c_t(\CS_{0:t},u_{\CS_{0:t}})=x_t^{\top}Qx_t + u_{t,\CS_t}^{\top}R_{\CS_t}u_{t,\CS_t}+x_{t+1}^{\top}Q_fx_{t+1},
\end{equation}
for $t=T-1$. Similarly, one can consider a general benchmark $\CS^{\star}$ in Eq.~\eqref{eqn:regret of A_c} as we described in Remark~\ref{remark:dynamic regret}.

\section{Algorithm Design for Episodic Setting}\label{sec:algorithm design for episodic setting}
We now design an online algorithm for the episodic setting of Problem~\eqref{eqn:LQR obj}. %As we discussed in Section~\ref{sec:online AS}, at the beginning of each episode $k\in[N]$, an online algorithm $\A_e$ needs to select a set $\CS^k\subseteq\G$ (with $|\CS^k|= H$) of actuators and design the corresponding control $u_{\CS^k}=(u_{0,\CS^k},\dots,u_{T-1,\CS^k})$. 
In our algorithm design, we leverage an algorithm for the multiarmed bandit problem (i.e., the {\bf Exp3.S} algorithm) \cite{auer2002nonstochastic} to select a set $\CS^k\subseteq\G$ (with $\CS^k=H$) for all $k$. Given the set $\CS^k$ of selected actuators, we then leverage a certainty equivalence approach \cite{mania2019certainty,cassel2020logarithmic} to design the corresponding control $u_{\CS^k}$. %In the following, we first describe the {\bf Exp3.S} algorithm and the certainty equivalence approach, and then introduce the overall algorithm design.

\subsection{Exp3.S Algorithm for Multiarmed Bandit}\label{sec:bandit problem}
%We briefly review the setting of the (nonstochastic) multiarmed bandit problem \cite{auer2002nonstochastic}.
The Multi-Armed Bandit (MAB) problem is specified by a number of episodes $N_s$, a finite set $\Q$ of possible actions (i.e., arms), and costs of actions $y^1,\dots,y^{N_s}$ with $y^k=(y_1^k,\dots,y_{|\Q|}^k)$ for all $k\in[N_s]$, where $\Q=\{1,\dots,|\Q|\}$ and $y_i^k\in[y_a,y_b]$ (with $y_a,y_b\in\R$) denotes the cost of choosing action $i$ in episode $k$, for all $k\in[N_s]$ and for all $i\in\Q$. At the beginning of each episode $k\in[N_s]$, one can choose an action from the set $\Q$. Choosing $i_k\in\Q$ for episode $k\in[N_s]$ incurs a cost $y_{i_k}^k$, which is revealed at the end of episode $k$. {\color{black} To  minimize the accumulative cost over the $N_s$ episodes, an online algorithm $\A_M$ chooses} action $i_k\in\Q$ for each episode $k\in[N_s]$, where the decision is made based on $i_{k^{\prime}}$ and $y_{i_{k^{\prime}}}^{k^{\prime}}$ for all $k^{\prime}\in\{1,\dots,k-1\}$. For any sequence of actions, i.e., $j^{N_s}\triangleq(j_1,\dots,j_{N_s})$, denote $h(j^{N_s})= 1+|\{1\le k<N_s:j_k\neq j_{k+1}\}|$. To design algorithm for the actuator selection problem (i.e., Problem~\eqref{eqn:LQR obj}), we leverage the {\bf Exp3.S} algorithm (i.e., Algorithm~\ref{alg:exp3.s}) \cite{auer2002nonstochastic}.

\begin{algorithm2e}
\caption{\bf Exp3.S}\label{alg:exp3.s}
\KwIn{Candidate set $\Q$, total number of episode $N_s$, parameters  $\alpha_1\in(0,1)$ and $\alpha_2>0$.}
Initialize $\varpi_i^1=1$, $\forall i\in[|\Q|]$.\\
\For{$k=1$ to $N_s$}{
  Set $q_i^k=(1-\alpha_1)\frac{\varpi_i^k}{\sum_{j=1}^{|\Q|}\varpi_j^k}+\frac{\alpha_1}{|\Q|},\ \forall i\in[|\Q|]$.\\
  Draw $i_k\in\Q$ according to the probabilities $q_1^k,\dots,q_{|\Q|}^k$, receive cost $y_{i_k}^k\in[y_a,y_b]$, and normalize $y_{i_k}^k=(-y_{i_k}^k+y_b)/(y_b-y_a)$.\\
  \For{$j=1,\dots,|\Q|$}{
    Set $\hat{y}_j^k=
    \begin{cases}
    y_j^k/q_j^k\ \text{if} j=i_k,\\
    0\ \text{otherwise},
    \end{cases}$
$\varpi_j^{k+1}=\varpi_j^k\text{exp}\Big(\frac{\alpha_1\hat{y}_j^k}{|\Q|}\Big)+\frac{e\alpha_2}{|\Q|}\sum_{i=1}^{|\Q|}\varpi_i^k$.
  }
}
\end{algorithm2e}

%Suppose that the costs of actions are generated by an {\it oblivious adversary}, i.e., the cost of an action in episode $k\in[N]$ does not depend on the previous actions $i_{1},\dots,i_{k-1}$ chosen by the algorithm and no statistical assumptions are made on the costs of the actions. 

%For any sequence of actions, i.e., $j^{N_s}\triangleq(j_1,\dots,j_{N_s})$, denote $h(j^{N_s})= 1+|\{1\le k<N_s:j_k\neq j_{k+1}\}|$. {\color{black} The following result is from~\cite{auer2002nonstochastic}.} %It has been shown in \cite{auer2002nonstochastic} that Algorithm~\ref{alg:exp3.s} yields the following regret when competing with any sequence of actions.
\begin{lemma}\label{thm:regret for exp3.s}\cite[Corollary~8.2]{auer2002nonstochastic} 
Consider any sequence $j^{N_s}=(j_1,\dots,j_{N_s})$. In Algorithm~\ref{alg:exp3.s}, let $\alpha_2=1/{N_s}$ and
\begin{equation*}
\alpha_1=\min\Big\{1,\sqrt{\frac{|\Q|(h(j^{N_s})\ln(|\Q|N_s)+e)}{(e-1)N_s}}\Big\}.
\end{equation*} 
Let $\E_M[\cdot]$ denote the expectation with respective to the randomness in the algorithm. Then, we have %the regret of Algorithm~\ref{alg:exp3.s} satisfies
\begin{align}
R_M(j^{N_s})\triangleq\E_{M}\bigg[\sum_{k=1}^{N_s}y_{i_k}^k\bigg] - \sum_{k=1}^{N_s}y_{j_k}^k\le 2(y_b-y_a)\sqrt{e-1}\sqrt{|\Q|{N_s}(h(j^{N_s})\ln(|\Q|N_s)+e)}.\label{eqn:regret for Exp3}
\end{align}
\end{lemma}
%Note that $R_M(j^{N_s})$ captures both the dynamic regret and the static regret described in Remark~\ref{remark:dynamic regret}.

\begin{remark}
\label{remark:condition for Exp3}
As argued in \cite{auer2002nonstochastic,arora2012online}, the regret bound in \eqref{eqn:regret for Exp3} holds under the assumption that for any $k\in[N_s]$, $y_{i_k}^k$ does not depend on the previous actions $i_1,,\dots,i_{k-1}$ chosen by the {\bf Exp3.S} algorithm. Other than this assumption, $y_i^k$ can be any real number in $[y_a,y_b]$, and no statistical assumption is made on $y_i^k$. %One may then view that $y_{i_k}(k)$ is generated by an oblivious adversary \cite{auer2002nonstochastic,arora2012online}. 
{\color{black} Also note that the random choices $i_1,\dots,i_{N_s}$ in line~3 in Algorithm~\ref{alg:exp3.s} ensure that in each episode $k\in[N_s]$, with some probability, the algorithm explores a new action or commits to the action that gives the lowest cost so far. Lemma~\ref{thm:regret for exp3.s} shows that such choices of $i_1,\dots,i_{N_s}$ yield sublinear regret in $N_s$ against an arbitrary benchmark $j^{N_s}$. }
\end{remark}

%We will call the {\bf Exp3.S} algorithm as a subroutine in our online algorithm for the actuator selection problem described in Section~\ref{sec:online AS}. Specifically, we let the set of possible actions $\Q$ corresponding to the {\bf Exp3.S} algorithm contain ${|\G|\choose H}$ actions, where recall that $\G$ is the set of candidate actuators and $H$ is the cardinality constraint on the set of selected actuators in each episode $t\in[T]$. Each action in $\Q$ now corresponds to a set $\CS\subseteq\G$ with $|\CS|=H$. Suppose that the set $\CS_t\subseteq\G$ with $|\CS_t|=H$ is selected by {\bf Exp3.S} in episode $t\in[T]$, and that a control policy $u^{(t)}_{\CS_t}=(u_{0,\CS_t}^{(t)},u_{1,\CS_t}^{(t)},\dots,u_{N-1,\CS_t}^{(t)})$ is chosen for the actuators in $\CS_t$. We then feedback the cost $J_t(\CS_t,u_{\CS_t}^{(t)})$ defined in Eq.~\eqref{eqn:episode t cost} as the cost that {\bf Exp3.S} would incur by choosing (the action corresponding to) $\CS_t$. %Note that since the set $\Q$ of possible actions contains ${|\B|\choose H}$ elements, the process that we described above may not scale well if both of $|\B|$ and $H$ grow large. We will discuss more about this issue later when we analyze the regret of our online algorithm. Nonetheless, supposing that $H$ is upper bounded by some fixed constant, then the number of elements in $\Q$ will be polynomial in $\B$ 

\subsection{Certainty Equivalence Approach}\label{sec:ce LQR}
In this subsection, we assume that a set $\CS\subseteq\G$ (with $|\CS|= H$) of actuators is selected and fixed for episode $k\in[N]$. We now describe our design of the corresponding control $u_{\CS}^k=(u_{0,\CS}^k,\dots,u_{T-1,\CS})$, based on the certainty equivalence approach \cite{mania2019certainty,cassel2020logarithmic}. First, {\color{black}conditioning on the set $\CS$ of selected actuators for episode $k\in[N]$, Eq.~\eqref{eqn:optimal control law} states that the corresponding optimal control is the linear state-feedback control given by} $\tilde{u}_{t,\CS}^k=K_{t,\CS}^kx_t^k$ for all $t\in\{0,\dots,T-1\}$, where the control gain matrix $K_{t,\CS}^k$ is obtained from Eq.~\eqref{eqn:control gain} (using the cost matrices $Q_k$ and $R_k$ in episode $k\in[N]$). Since the system matrices $A$ and $B$ are unknown, the certainty equivalence approach leverages estimates of the system matrices, denoted as $\hat{A}$ and $\hat{B}$,\footnote{The estimates $\hat{A}$ and $\hat{B}$ are obtained by some system identification method, using data samples from the system trajectory; we will elaborate more on the system identification part later.} in order to compute the control gain matrix \cite{mania2019certainty,cassel2020logarithmic}. For any $t\in\{0,1,\dots,T-1\}$, the certainty equivalent controller for the $k$th episode of Problem~\eqref{eqn:LQR obj} is given by 
\begin{align}
\label{eqn:ce u}
u_{t,\CS}^k&=\hat{K}_{t,\CS}^kx_t^{k},\\
\label{eqn:control gain S_t est}
\hat{K}_{t,\CS}^k &= -\big(\hat{B}_{\CS}^{\top}\hat{P}_{t+1,\CS}\hat{B}_{\CS}+ R_{\CS}^k\big)^{-1}\hat{B}_{\CS}^{\top}\hat{P}_{k+1,\CS}\hat{A}\\
\label{eqn:recursion for P_k(S_t) est}
\hat{P}_{t,\CS}^k &= Q^k + \hat{A}^{\top}\hat{P}_{t+1,\CS}^k\hat{A}-\hat{A}^{\top}\hat{P}_{t+1,\CS}^k\hat{B}_{\CS}\big(\hat{B}_{\CS}^{\top}\hat{P}_{t+1,\CS}^k\hat{B}_{\CS}+R_{\CS}^k\big)^{-1}\hat{B}_{\CS}^{\top}\hat{P}_{t+1,\CS}^k\hat{A},
\end{align}
% \begin{equation}
% \label{eqn:ce u}
% u_{t,\CS}^k=\hat{K}_{t,\CS}^kx_t^{k},\ \forall t\in\{0,1,\dots,T-1\},
% \end{equation}
% where
% \begin{equation}
% \label{eqn:control gain S_t est}
% \hat{K}_{t,\CS}^k = -\big(\hat{B}_{\CS}^{\top}\hat{P}_{t+1,\CS}\hat{B}_{\CS}+ R_{\CS}^k\big)^{-1}\hat{B}_{\CS}^{\top}\hat{P}_{k+1,\CS}\hat{A},
% \end{equation}
% and $\hat{P}_{t,\CS}^k\in\BS_+^n$ satisfies the following recursion:
% \begin{multline}
% \label{eqn:recursion for P_k(S_t) est}
% \hat{P}_{t,\CS}^k = Q^k + \hat{A}^{\top}\hat{P}_{t+1,\CS}^k\hat{A}-\hat{A}^{\top}\hat{P}_{t+1,\CS}^k\hat{B}_{\CS}\\\times\big(\hat{B}_{\CS}^{\top}\hat{P}_{t+1,\CS}^k\hat{B}_{\CS}+R_{\CS}^k\big)^{-1}\hat{B}_{\CS}^{\top}\hat{P}_{t+1,\CS}^k\hat{A},
% \end{multline}
where $\hat{P}_{t,\CS}^k\in\BS_+^n$ and~(\ref{eqn:recursion for P_k(S_t) est}) is initialized with $\hat{P}_{T,\CS}^k=Q_f^k$. %In words, the certainty equivalence controller is also a linear state-feedback controller where the control gain matrix is obtained by substituting $A$ and $B$ in Eqs.~\eqref{eqn:control gain}-\eqref{eqn:recursion for P_k} with the estimates $\hat{A}$ and $\hat{B}$. 

Next, we characterize the performance of the resulting certainty equivalent controller, which naturally depends on the estimation errors $\norm{\hat{A}-A}$ and $\norm{\hat{B}-B}$. Similarly to \eqref{eqn:opt LQR cost S_t}, we denote the expected cost corresponding to $\CS$ and $u_{t,\CS}^k=\hat{K}_{t,\CS}^kx^k_t$ for the $k$th episode as
\begin{equation}
\label{eqn:episode t ce cost}
\hat{J}_k(\CS)=\E[J_k(\CS,u_{\CS}^k)],
\end{equation}
where $J_k(\CS,u_{\CS}^k)$ is defined in Eq.~\eqref{eqn:episode t cost}. One can show that the following expression for $\hat{J}_k(\CS)$ holds \cite[Chapter~3]{bertsekas2017dynamic}: 
\begin{equation}
\label{eqn:exp for J_t hat}
\hat{J}_k(\CS)=\E\big[x_0^{k\top}\tilde{P}_{0,\CS}^kx_0^k\big]+\sum_{t=0}^{T-1}\Tr(\tilde{P}_{t+1,\CS}^kW),
\end{equation}
where $\tilde{P}_{t,\CS}^k$ satisfies the following recursion with $\tilde{P}_{T,\CS}^k=Q_f^k$
\begin{equation}
\label{eqn:recursion for P_k(S_t) tilde}
\tilde{P}_{t,\CS}^k = Q^k + \hat{K}_{t,\CS}^{k\top}R^k_{\CS}\hat{K}_{t,\CS}^k +(A+B_{\CS}\hat{K}_{k,\CS})^{k\top}P_{t+1,\CS}^k(A+B_{\CS}\hat{K}_{t,\CS}^k).
\end{equation}

{\color{black} We now %In order to characterize the performance (i.e., suboptimality) of the certainty equivalent controller given in \eqref{eqn:ce u}, we will provide an 
upper bound $\hat{J}_k(\CS)-J_k(\CS)$} in terms of the estimation error in $\hat{A}$ and $\hat{B}$, where $J_k(\CS)$ is defined in Eq.~\eqref{eqn:opt LQR cost S_t}. Note that both the optimal controller (with gain matrix $K_{t,\CS}^k$ given by Eq.~\eqref{eqn:optimal control law}) and the certainty equivalent controller (with gain matrix $\hat{K}_{t,\CS}^k$ given by Eq.~\eqref{eqn:control gain S_t est}) are time-varying for the finite-horizon setting. %, which are obtained recursively from the DAREs in Eqs.~\eqref{eqn:recursion for P_k} and~\eqref{eqn:recursion for P_k(S_t) est}, respectively. 
In contrast, both the optimal controller \cite{bertsekas2017dynamic} and the certainty equivalent controller proposed in \cite{mania2019certainty,cassel2020logarithmic} are time-invariant for the infinite-horizon setting, which are obtained from steady-state solutions to DAREs. Hence, %to handle the recursions in Eqs.\eqref{eqn:control gain S_t est}-\eqref{eqn:recursion for P_k(S_t) est}, 
our analysis for the certainty equivalence approach for learning finite-horizon LQR will be more challenging than that in \cite{mania2019certainty,cassel2020logarithmic} for learning infinite-horizon LQR. To proceed, supposing the estimation error satisfies that $\norm{A-\hat{A}}\le\varepsilon$ and $\norm{B-\hat{B}}\le\varepsilon$ with $\varepsilon\in\R_{>0}$, we provide upper bounds on $\norm{K_{t,\CS}^k-\hat{K}_{t,\CS}^k}$ and $\norm{P_{t,\CS}^k-\hat{P}_{t,\CS}^k}$, where  $P_{t,\CS}^k$ (resp., $\hat{P}_{t,\CS}^k$) is given by Eq.~\eqref{eqn:recursion for P_k} (resp., Eq.~\eqref{eqn:recursion for P_k(S_t) est}). We need the following mild assumption. % on the cost matrices $Q^k$, $Q_f^k$ and $R^k$.
\begin{assumption}
\label{ass:cost matrices}
We assume that $\sigma_n(Q^k)\ge1$ and $\sigma_m(R^k)\ge1$ for all $k\in[N]$.
\end{assumption}
%Note that assuming $\sigma_n(Q^k)\ge1$ and $\sigma_m(R^k)\ge1$ is not more restrictive than assuming $Q^k\in\BS_{++}^n$ and $R^k\in\BS_{++}^m$. This is because multiplying both sides of Eq.~\eqref{eqn:episode t cost} by a positive constant does not change $K_{t,\CS}^k$ defined Eq.~\eqref{eqn:control gain}. 
In order to simplify the notations in the sequel, we denote
\begin{align}
\Gamma_{\CS}&=\max_{k\in[N],t\in[T]}\Gamma_{t,\CS}^k,\label{eqn:def of Gamma S_t}\\
\tilde{\Gamma}_{\CS}&=1+\Gamma_{\CS},\label{eqn:def of Gamma S_t tilde}
\end{align}
where $\Gamma_{t,\CS}^k=\max\big\{\norm{A},\norm{B},\norm{P_{t,\CS}^k},\norm{K_{t-1,\CS}^k}\big\}$.
Moreover, we denote
\begin{equation}
\label{eqn:max sigma_1 Q and R}
\begin{split}
\sigma_Q&=\max\big\{\max_{k\in[N]}\sigma_1(Q^k),\max_{k\in[N]}\sigma_1(Q_f^k)\big\},\\
\sigma_R&=\max_{k\in[N]}\sigma_1(R^k).
\end{split}
\end{equation}
We have the following result; the proof can be found in Appendix~\ref{app:ce approach}.
\begin{lemma}
\label{lemma:mismatch of K_hat and P_hat}
Consider any $\CS\subseteq\G$, any $k\in[N]$ and any $t\in[T]$. Let $\varepsilon\in\R_{\ge0}$ and $D\in\R_{\ge0}$ with $\varepsilon\le 1$ and $D\ge1$. Suppose that $\norm{A-\hat{A}}\le\varepsilon$, $\norm{B_{\CS}-\hat{B}_{\CS}}\le\varepsilon$, and $\norm{P_{t,\CS}^{k}-\hat{P}_{t,\CS}^{k}}\le D\varepsilon$, and that Assumption~\ref{ass:cost matrices} holds. Then,
\begin{align}
&\norm{K_{t-1,\CS}^k-\hat{K}_{t-1,\CS}^k}\le3\tilde{\Gamma}_{\CS}^3 D\varepsilon,\label{eqn:est error of K_hat}\\
&\norm{P_{t-1,\CS}^k-\hat{P}_{t-1,\CS}^k}\le44\tilde{\Gamma}_{\CS}^9\sigma_RD\varepsilon.\label{eqn:est error of P_hat}
\end{align}
\end{lemma}
%Recalling that $P_N^{(t)}=\hat{P}_N^{(t)}=Q_f^{(t)}$, one can simply apply \eqref{eqn:est error of P_hat} recursively and obtain upper bounds on $\norm{P_k^{(t)}(\CS)-\hat{P}_k^{(t)}(\CS)}$ and $\norm{K_k^{(t)}(\CS)-\hat{K}_k^{(t)}(\CS)}$  for all $k\in\{0,1,\dots,N-1\}$, where $N\in\Z_{\ge1}$ is the length of the horizon of the LQR problem in each episode $t\in[T]$. However, the resulting upper bounds will grow exponentially in $k^{\prime}$. 
We make the following assumption on the controllability of the pair $(A,B)$ similar to \cite{cohen2019learning,mania2019certainty,dean2018regret}.
\begin{assumption}
\label{ass:controllability}
For any $\CS\subseteq\G$ with $|\CS|=H$, we assume that the pair $(A,B_{\CS})$ in system~\eqref{eqn:LTI} satisfies that $\sigma_1(\C_{\ell,\CS})\ge\nu$, where $\ell\in[n-1]$, $\nu\in\R_{>0}$ and
%\begin{equation*}
$\C_{\ell,\CS}\triangleq\begin{bmatrix}B_{\CS} & AB_{\CS} & \cdots & A^{\ell-1}B_{\CS}\end{bmatrix}.$
%\end{equation*}
\end{assumption}
If Assumption~\ref{ass:controllability} is satisfied, we say that $(A,B_{\CS})$ is $(\ell,\nu)$-controllable \cite{mania2019certainty}. %As discussed in \cite{mania2019certainty}, $(\ell,\nu)$-controllability is a quantitative counterpart of the classic notion of controllability. To see this, considering any $\CS\subseteq\G$ (with $|\CS|=H$) and supposing that the pair $(A,B_{\CS})$ is $(\ell,\nu)$-controllable for some $\ell\in[n-1]$ and $\nu\in\R_{>0}$, we have that $(A,B_{\CS})$ is controllable (since $\rank(\C_{n,\CS})=n$). The converse is also true. 
Note that if $(A,B_{\CS})$ is controllable, $(A,B_{\CS})$ can be $(\ell,\nu)$-controllable for some $\ell\in[n-1]$ that is much smaller than $n$. %For example, supposing that $\rank(B_{\CS})=n$, then $(A,B_{\CS})$ is $(1,\nu)$-controllable. 
One can also check that a sufficient condition for Assumption~\ref{ass:controllability} to hold is that for any actuator $s\in\G$, the pair $(A,B_s)$ is $(\ell,\nu)$-controllable. %Moreover, note that Assumption~\ref{ass:controllability} can be relaxed to a stabilizable assumption when considering the learning infinite-horizon LQR problem \cite{mania2019certainty,lale2022reinforcement}. However, due to the extra challenge of the finite-horizon setting as we argued above, we leave the relaxation of Assumption~\ref{ass:controllability} to future work. 
Denoting
\begin{equation*}
\label{eqn:def of C_hat}
\hat{\C}_{\ell,\CS}=\begin{bmatrix}\hat{B}_{\CS} & \hat{A}\hat{B}_{\CS} & \cdots & \hat{A}^{\ell-1}\hat{B}_{\CS}\end{bmatrix}\quad\forall \CS\subseteq\G,
\end{equation*}
we have the following lower bound on $\sigma_n(\hat{\C}_{\ell,\CS})$.
\begin{lemma}\cite[Lemma~6]{mania2019certainty}
\label{lemma:lower bound on sigma_1 C_hat}
Consider any $\CS\subseteq\G$. Suppose that $\norm{A-\hat{A}}\le\varepsilon$ and $\norm{B_{\CS}-\hat{B}_{\CS}}\le\varepsilon$, where $\varepsilon\in\R_{\ge0}$. Under Assumption~\ref{ass:controllability}, $\sigma_n(\hat{\C}_{\ell,\CS})\ge\nu-\varepsilon\ell^{\frac{3}{2}}\beta^{\ell-1}(\norm{B_{\CS}}+1)$, where $\beta\triangleq\max\{1,\varepsilon+\norm{A}\}$.
\end{lemma}
{\color{black} Lemma~\ref{lemma:lower bound on sigma_1 C_hat} states that if $\varepsilon$ is small enough, then $\sigma_n(\hat{\C}_{\ell,\CS})>0$, i.e., $\rank(\hat{\C}_{\ell,\CS})=n$ and the pair $(\hat{A},\hat{B}_{\CS})$ is controllable. We have the following result proved in Appendix~\ref{app:ce approach}.}
\begin{lemma}
\label{lemma:bound on est error of P_hat 1}
Consider any $\CS\subseteq\G$ with $|\CS|=H$ and any $k\in[N]$. If Assumptions~\ref{ass:cost matrices}-\ref{ass:controllability} hold,  $\norm{A-\hat{A}}\le\varepsilon$ and $\norm{B_{\CS}-\hat{B}_{\CS}}\le\varepsilon$, where $\varepsilon\in\R_{\ge0}$, then, for any $t\in\{T-\gamma\ell:\gamma\in\Z_{\ge0},\gamma\ell\le T\}$, and with $\beta=\max\{1,\varepsilon+\norm{A}\}$, it holds that 
\begin{equation}
\label{eqn:bound on est error of P_hat}
\norm{P_{t,\CS}^k-\hat{P}_{t,\CS}^k}\le\mu_{t,\CS}^k\varepsilon,
\end{equation}
under the assumption that $\mu_{t,\CS}^k\varepsilon\le1$ with
\begin{equation}
\label{eqn:def of mu}
\mu_{t,\CS}^k\triangleq32\ell^{\frac{5}{2}}\beta^{2(\ell-1)}(1+\nu^{-1})(1+\norm{B_{\CS}})^2\norm{P_{t,\CS}^{k}}\max\{\sigma_Q,\sigma_R\}.
\end{equation}
%where .
\end{lemma}

Let us further denote
\begin{equation}
\label{eqn:mu_S}
\mu_{\CS}=32\ell^{\frac{5}{2}}\tilde{\beta}^{2(\ell-1)}(1+\nu^{-1})\tilde{\Gamma}_{\CS}^3\max\{\sigma_Q,\sigma_R\},
\end{equation}
where $\tilde{\beta}=1+\norm{A}$. Now, combining Lemmas~\ref{lemma:mismatch of K_hat and P_hat} and~\ref{lemma:bound on est error of P_hat 1} yields the following result, which upper bounds $\norm{K_{t,\CS}^k-\hat{K}_{t,\CS}^k}$ and $\norm{P_{t,\CS}^k-\hat{P}_{t,\CS}^k}$ for all $t$; the proof can be found in Appendix~\ref{app:ce approach}.
\begin{proposition}
\label{prop:bound on est error of P_hat}
Consider any $\CS\subseteq\G$ with $|\CS|=H$. Suppose that Assumptions~\ref{ass:cost matrices}-\ref{ass:controllability} hold, and that $\norm{A-\hat{A}}\le\varepsilon$, $\norm{B_{\CS}-\hat{B}_{\CS}}\le\varepsilon$, where $\varepsilon\in\R_{\ge0}$ and $\mu_{\CS}\varepsilon\le1$. Then, for any $t\in\{0,1,\dots,T\}$, it holds that
\begin{equation}
\label{eqn:bound on est error of P_hat general}
\norm{P_{t,\CS}^k-\hat{P}_{t,\CS}^k}\le(44\tilde{\Gamma}_{\CS}^9\sigma_R)^{\ell-1}\mu_{\CS}\varepsilon, 
\end{equation}
Moreover, for any $t\in\{0,1,\dots,T-1\}$, it holds that
\begin{equation}
\label{eqn:est error of K_hat general}
\norm{K_{t,\CS}^k-\hat{K}_{t,\CS}^k}\le3\tilde{\Gamma}_{\CS}^3 (44\tilde{\Gamma}_{\CS}^9\sigma_R)^{\ell-1}\mu_{\CS}\varepsilon.
\end{equation}
\end{proposition}

{\color{black} We are now in place to upper bound $\hat{J}_k(\CS)-J_k(\CS)$.} % in terms of the estimation error corresponding to $\hat{A}$ and $\hat{B}$, where $\hat{J}_k(\CS)$ (resp., $J_k(\CS)$) is defined in Eq.~\eqref{eqn:episode t ce cost} (resp., \eqref{eqn:opt LQR cost S_t}). 
We begin with the following result; the proof can be found in Appendix~\ref{app:ce approach}.
\begin{lemma}
\label{lemma:J_hat minus J}
Consider any $\CS\subseteq\G$ and any $k\in[N]$. Let $x_t^k$ be the state corresponding to the certainty equivalence control $u^k_{t,\CS}=\hat{K}_{t,\CS}^kx_t^k$, i.e., $x^k_{t+1}=(A+B_{\CS}\hat{K}_{t,\CS}^k)x^k_t+w^k_t$, where $w^k_t$ is the zero-mean white Gaussian noise process with covariance $W$ for all $k$. Let $\Delta K_{t,\CS}^k\triangleq \hat{K}_{t,\CS}^{k}-K_{t,\CS}^{k}$. Then,
\begin{equation}
\label{eqn:J_hat minus J}
\hat{J}_k(\CS)-J_k(\CS)=\sum_{t=0}^{T-1}\E\Big[x_t^{k\top}\Delta \hat{K}_{t,\CS}^{k\top}(R^k_{\CS}+B_{\CS}^{\top}P_{t,\CS}B_{\CS})\Delta \hat{K}_{t,\CS}^kx_t^k\Big].
\end{equation}
%where . 
\end{lemma}

To proceed, consider any $\CS\subseteq\G$ and any $k\in[N]$. For any $t_1,t_2\in\{0,1,\dots,T\}$ with $t_2\ge t_1$, we use $\Psi_{t_2,t_1}^{k}(\CS)$ to denote the state transition matrix corresponding to $A+B_{\CS}K_{t,\CS}^{k}$, i.e., 
\begin{equation}
\label{eqn:Psi}
\Psi_{t_2,t_1}^{k}(\CS)=(A+B_{\CS}K_{t_2-1,\CS}^{k})(A+B_{\CS}K_{t_2-2,\CS}^{k})\cdots(A+B_{\CS}K_{t_1,\CS}^{k}),
\end{equation}
and $\Psi_{t_2,t_1}^{k}(\CS)\triangleq I$ if $t_1=t_2$, where $K_{t,\CS}^{k}$ is given by Eq.~\eqref{eqn:control gain}. Similarly, we denote
\begin{equation}
\label{eqn:Psi_hat}
\hat{\Psi}_{t_2,t_1}^{k}(\CS)=(A+B_{\CS}\hat{K}_{t_2-1,\CS}^{k})\times\cdots\times(A+B_{\CS}\hat{K}_{t_1,
\CS}^{k}),
\end{equation}
and $\hat{\Psi}_{t_2,t_1}^{k}(\CS)\triangleq I$ if $t_1=t_2$, where $\hat{K}_{t,\CS}^{k}$ is given by Eq.~\eqref{eqn:control gain S_t est}. One can now prove the following result, which shows that the state transition matrix $\Psi_{t_2,t_1}^k(\CS)$ is exponentially stable; a proof of the result can be found in \cite{anderson1981detectability,zhang2021regret}.
\begin{lemma}
\label{lemma:bound on norm of Phi}
Consider any $S\subseteq\G$ with $|\CS|=H$ and any $k\in[N]$. Supposing Assumptions~\ref{ass:cost matrices}-\ref{ass:controllability} hold. Then, there exist finite constants $\zeta_{\CS}\in\R_{\ge1}$ and $0<\eta_{\CS}<1$ such that $\norm{\Psi_{t_2,t_1}^{k}(\CS)}\le\zeta_{\CS}\eta_{\CS}^{t_2-t_1}$ for all $t_1,t_2\in\{0,1,\dots,T\}$ with $t_2\ge t_1$.
\end{lemma}

The following result characterizes the stability of $\hat{\Psi}_{t_2,t_1}^k(\CS)$; the proof follows from Lemma~\ref{lemma:bound on norm of Phi} and can be found in Appendix~\ref{app:ce approach}.
\begin{lemma}
\label{lemma:bound on Psi_hat}
Consider any $\CS\subseteq\G$ with $|\CS|=H$ and any $t\in[T]$. Suppose Assumptions~\ref{ass:cost matrices}-\ref{ass:controllability} hold, and $\norm{K_{t,\CS}^{k}-\hat{K}_{t,\CS}^{k}}\le\varepsilon$ for all $t\in\{0,1,\dots,T-1\}$, where $\varepsilon\in\R_{>0}$. Then, for all $t_1,t_2\in\{0,1,\dots,T\}$ with $t_2\ge t_1$, $\norm{\hat{\Psi}_{t_2,t_1}^{k}(\CS)}\le\zeta_{\CS}(\frac{1+\eta_{\CS}}{2})^{t_2-t_1}$, under the assumption that $\varepsilon\le\frac{1-\eta_{\CS}}{2\norm{B_{\CS}}\zeta_{\CS}}$, where $\zeta_{\CS}\ge1$ and $0<\eta_{\CS}<1$ are given by Lemma~\ref{lemma:bound on norm of Phi}.
\end{lemma}

Combining Lemmas~\ref{lemma:J_hat minus J} and \ref{lemma:bound on Psi_hat}, and Proposition~\ref{prop:bound on est error of P_hat} yields the following result; the proof is included can be found in Appendix~\ref{app:ce approach}.
\begin{proposition}
\label{prop:bound on J_hat minus J}
Consider any $\CS\subseteq\G$ with $|\CS|=H$ and any $k\in[N]$. Suppose Assumptions~\ref{ass:cost matrices}-\ref{ass:controllability} hold, and $\norm{A-\hat{A}}\le\varepsilon$ and $\norm{B-\hat{B}}\le\varepsilon$, where $\varepsilon\in\R_{>0}$. Then, it holds that
\begin{equation}
\hat{J}_k(\CS)-J_k(\CS)\le \frac{4\min\{n,m_{\CS}\}T\zeta_{\CS}^2}{1-\eta_{\CS}^2}\sigma_1(W)\big(\sigma_R+\Gamma_{\CS}^3\big)\big(3\tilde{\Gamma}_{\CS}^3 (20\tilde{\Gamma}_{\CS}^9\sigma_R)^{\ell-1}\mu_{\CS}\big)^2\varepsilon^2,\label{eqn:upper bound on J_hat minus J}
\end{equation}
under the assumption that $\varepsilon\le\frac{1-\eta_{\CS}}{\zeta_{\CS}\mu_{\CS}}$, where $\zeta_{\CS}\ge1$ and $0<\eta_{\CS}<1$ are given by Lemma~\ref{lemma:bound on norm of Phi}, $J_k(\CS)$ and $\hat{J}_k(\CS)$ are defined in \eqref{eqn:opt LQR cost S_t} and \eqref{eqn:episode t ce cost}, respectively, and $m_{\CS}=\sum_{i\in\CS}m_i$.
\end{proposition}
{\color{black} Hence, supposing the estimation error of $\hat{A},\hat{B}$ can be made small enough, Proposition~\ref{prop:bound on J_hat minus J} bounds the gap between the (expected) costs incurred by the certainty equivalent controller and the optimal controller that knows the system model $A,B$.}

\subsection{Overall Algorithm Design}\label{sec:online as episodic setting}
We introduce the overall algorithm (Algorithm~\ref{alg:online for AS}) for the episodic setting of Problem~\eqref{eqn:LQR obj}, under the assumptions below.
{\color{black}
\begin{assumption}
\label{ass:noise process}
We assume that  (a) for any $k\in[N]$, $\{w_t^k\}_{t=0}^{T-1}$ are i.i.d Gaussian with $\E[w_t^k]=0$ and $\E[w_t^kw_t^{k\top}]=\sigma^2 I$, i.e., $w_t^k\overset{\text{i.i.d.}}{\sim}\N(0,\sigma^2I)$, where $\sigma\in\R_{\ge0}$ is known; (b) for any distinct $t_1,t_2\in\{0,1\dots,T-1\}$ and any distinct $k_1,k_2\in[N]$, the noise terms $w^{k_1}_{t_1}$ and $w^{k_2}_{t_2}$ are independent.
\end{assumption}}

\begin{assumption}
\label{ass:system parameters}
There exist $\G_1,\dots,\G_p$ with $\G_i\subseteq\G$ and $|\G_i|=H$ for all $i\in[p]$ such that $\G=\cup_{i\in[p]}\G_i$ and there is a known stabilizing $K_{\G_i}\in\R^{m_{\G_i}\times n}$ with $\norm{(A+B_{\G_i}K_{\G_i})^t}\le\zeta_0\eta_0^t$ and $\norm{K_{\G_i}}\le\zeta_0$, $\forall t\in\R_{\ge0}$ and  $\forall i\in[p]$, where $p=\lceil m/H \rceil$, $m_{\G_i}=\sum_{j\in\G_i}m_j$, $\zeta_0\in\R_{\ge1}$ and $\eta_0\in\R_{>0}$ with $0<\eta_0<1$.
\end{assumption}

\begin{savenotes}
\begin{algorithm2e}[h]
\SetNoFillComment
\caption{Episodic Setting}\label{alg:online for AS}
\KwIn{Parameters $\tau_1,\lambda,N,T,\bar{y}_b$, and $K_{\G_j}$ for all $j\in[p]$ from Assumption~\ref{ass:system parameters}.}
Initialize $N_{1}=1$.\\
\For{$j=1$ to $p$}{
    Set $N_{j+1}\gets N_{j}+\tau_1$.
}
%Set $N_{i+1,1}=N_{i,p+1}+\tau_2$.\\

    \tcc{System identification phase}
    \For{$j=1$ to $p$}{
        \For{$k=N_{j}$ to $N_{j+1}-1$}{
            Select $\CS^k=\G_j$.\\
            Play $u_{\G_j}^{k}$ with $u_{t,\G_j}^{k}\overset{\text{i.i.d.}}{\sim}\CN(K_{\G_j}x^{k}_t,2\sigma^2\eta_0^2I)$,
            $\qquad\qquad\quad\forall t\in\{0,\dots,T-1\}$.
        }
        {\color{black} Obtain $\hat{\Theta}_{\G_j}$ from \eqref{eqn:ls approach}}.
    }
    {\color{black} Obtain $\hat{A}$ by extracting the first $n$ columns from $\hat{\Theta}_{\G_1}$; obtain $\hat{B}$ by extracting the last $m_{\G_j}$ columns from $\hat{\Theta}_{\G_j}$ for all $j\in[p]$ and merging them into $\hat{B}$.}\\
    \tcc{Control phase}
    {\color{black} Initialize an {\bf Exp3.S} subroutine with $N_s=N-N_{p+1}+1$, $\Q=\{\CS\subseteq\G:|\CS|=H\}$, and $\alpha_1,\alpha_2$ according to Lemma~\ref{thm:regret for exp3.s}.\footnote{Note that $h(j^{N_s})$ in Lemma~\ref{thm:regret for exp3.s} is set to be $h(\CS_{\star})$ defined in Eq.~\eqref{eqn:number of switchings}.}}\\ 
    \For{$k=N_{p+1}$ to $N$}{
        {\color{black} Enter the $(k-N_{p+1}+1)$th iteration of the for loop in lines~2-6 in {\bf Exp3.S}; select $\CS^k\in\Q$ according to the probabilities $q_1^k,\dots,q_{|Q|}^k$.}\\
        \For{$t=0$ to $T-1$}{
            Obtain $\hat{K}_{t,\CS^k}^{k}$ using $\hat{A},\hat{B}_{\CS^k}$ via Eq.~\eqref{eqn:control gain S_t est}.\\
            Play $u_{t,\CS^k}^{k}=\hat{K}_{t,\CS^k}^{k}x_t^{k}$.
        }
        {\color{black} Receive the cost $y_{\CS^k}^k=J_k(\CS,u_{\CS}^k)$; follow lines~4-6 in {\bf Exp3.S} with $y_a=-\bar{y}_b$ and $y_b=0$.} %and finish the $(k-N_{p+1}+1)$th iteration of the for loop in lines~2-6 in {\bf Exp3.S}.}
    }
{\color{black}\KwOut{$\CS^k,u_{\CS^k}^k=(u_{0,\CS^k}^k,\dots,u_{T-1,\CS^k}^k),\forall k\in[N]$.}}
\end{algorithm2e}
\end{savenotes}

Assumption~\ref{ass:noise process}(a) ensures that the noise terms from different episodes are independent. Similarly to \cite{dean2018regret,mania2019certainty,cassel2020logarithmic}, assuming the noise covariance is $W=\sigma^2I$ is only made to ease the presentation; our analysis in the remaining of this paper can be extended to $w_t^k$ with general covariance matrix $W\in\BS_{++}^n$, where the analysis will then depend on $\sigma_1(W)$ and $\sigma_n(W)$. {\color{black} Note that more general noise models are considered in, e.g., \cite{faradonbeh2020input,lale2022reinforcement}, where $\{w_t^k\}_{t=0}^{T-1}$ can be non-stationary and non-Gaussian. We restrict ourselves to the i.i.d. Gaussian noise model of $\{w_t^k\}_{t=0}^{T-1}$ described in Assumption~\ref{ass:noise process}, and leave the extension to the more general noise models to future work.}

%Note that similarly to \cite{cohen2019learning,cassel2020logarithmic}, we assume that $w_t^k$ (for any $k\in[N]$) is a zero-mean white Gaussian noise process can also be relaxed by assuming that $w_t^k$ is a martingale difference sequence with respect to a filtration, and each term in the sequence is conditionally sub-Gaussian \cite{abbasi2011regret,lale2022reinforcement}. %To focus on the actuator selection problem studied in this paper, we use the more restrictive Gaussian assumption on the noise process in system~\eqref{eqn:LTI}. As we will see in the next section, our regret analysis for Algorithm~\ref{alg:online for AS} leverages certain concentration inequalities for i.i.d. Gaussian distributions \cite{cohen2019learning,cassel2020logarithmic}. Hence, our analysis can be extended to the sub-Gaussian setting by leveraging the corresponding concentration inequalities \cite{abbasi2011regret,lale2022reinforcement}, but we leave details of the extension to future work since it is not core to the central problem of joint actuator selection and controller design.

Similar assumptions to Assumption~\ref{ass:system parameters} can also be found in \cite{dean2018regret,mania2019certainty,cohen2019learning,cassel2020logarithmic}. Note that under Assumption~\ref{ass:controllability}, the pair $(A,B_{\G_i})$ is controllable for all $i\in[p]$, which guarantees the existence of the $K_{\G_i}$ described in Assumption~\ref{ass:system parameters}. {\color{black}Moreover, the stability of $A+B_{\G_i}K_{\G_i}$ ensures via the Gelfand formula \cite{horn2012matrix} that the finite constants $\zeta_0\ge1$ and $0<\eta_0<1$ exist, which may be computed by the LQR cost of the control $u_t=K_{\G_i}x_t$} \cite{cassel2020logarithmic}. Also note that Assumption~\ref{ass:system parameters} gives us a set of known stabilizing controllers that we can use in the system identification of Algorithm~\ref{alg:online for AS} (see our detailed descriptions below). In fact, using the techniques from \cite{faradonbeh2018finite,chen2020black,lale2022reinforcement}, one can introduce an extra warm-up phase before the system identification phase in Algorithm~\ref{alg:online for AS}, which learns a stabilizing $K_{\G_i}$ for all $i\in[p]$ from the system trajectory. In particular, as shown in \cite{chen2020black}, the extra warm-up phase will incur an extra additive factor $2^{\tilde{O}(n)}$ in the regret of Algorithm~\ref{alg:online for AS} defined in Eq.~\eqref{eqn:regret of A_e}.

Now, we explain the steps in Algorithm~\ref{alg:online for AS}. %Note Algorithm~\ref{alg:online for AS} consists of the exploration phase and the exploitation phase, and calls the Algorithms~\ref{alg:exp3.s} and \ref{alg:ls estimation} as subroutines. %the estimation phase and the control phase. The estimation (resp., control) phase is also referred to as exploration (resp., exploitation) in the reinforcement learning context \cite{sutton2018reinforcement}. %We now describe the details about Algorithm~\ref{alg:online for AS}. %In other words, Algorithm~\ref{alg:online for AS} explores to continuously decrease the estimation error of the system matrices, and exploits the improved estimates of the system matrices to design the control.

\iffalse
\begin{algorithm}[h]
\caption{Least Squares Estimation ({\bf LSE})}\label{alg:ls estimation}
\For{$j=1$ to $p$}{
Obtain $\hat{\Theta}_{\G_j}$ from~\eqref{eqn:ls approach}.
}
Obtain $\hat{A}$ (resp., $\hat{B}$) by extracting the corresponding columns from $\hat{\Theta}_{\G_1}$ (resp., $\hat{\Theta}_{\G_j}$ for all $j\in[p]$).
\end{algorithm}
\fi

{\bf System identification phase:} In lines~4-9, Algorithm~\ref{alg:online for AS} computes estimates of $A$ and $B$, denoted as $\hat{A}$ and $\hat{B}$, respectively. This is achieved by first iteratively selecting the sets $\G_1,\dots,\G_p$ of actuators and playing the corresponding stabilizing controller given by Assumption~\ref{ass:system parameters} for $\tau_1$ episodes. Formally, for any $j\in[p]$, Algorithm~\ref{alg:online for AS} selects $\CS^k=\G_j$ and plays the control $u_{t,\G_j}^{k}\overset{\text{i.i.d.}}{\sim}\CN(K_{\G_j}x^{k}_t,2\sigma^2\eta_0^2I)$ for all time steps $t\in\{0,\dots,T-1\}$ and all episodes $k\in\{N_{j},\dots,N_{j+1}-1\}$ with $N_{j+1}=N_j+\tau_1$, where  {\color{black} $\tau_1\in\Z_{\ge1}$ is an input to Algorithm~\ref{alg:online for AS} whose value will be specified later.} We assume that $u_{t,\G_j}^k$ is independent of the noise $w_{t^{\prime}}^{k^{\prime}}$ for all $t^{\prime}\in\{0,\dots,T-1\}$ and all $k^{\prime}\in[N]$. For any $j\in[p]$, the estimate $\hat{\Theta}_{\G_j}\in\R^{n\times(n+m_{\G_j})}$ (with $m_{\G_j}=\sum_{i\in\G_j}m_i$) is obtained by solving the following regularized least squares:\footnote{Note that a solution to \eqref{eqn:ls approach} can be obtained recursively as a new data sample from the system trajectory becomes available at each time step \cite{kay1993fundamentals}.}
\begin{equation}
\label{eqn:ls approach}
\hat{\Theta}_{\G_j}\in\argmin_{Y}\Big\{\lambda\norm{Y}_F^2+\displaystyle\sum_{k=N_{j}}^{N_{j+1}-1}\displaystyle\sum_{t=0}^{T-1}\norm{x_{t+1}^{k}-Y z_{t,\G_j}^{k}}^2\Big\},
\end{equation}
where $\lambda\in\R_{>0}$ and
\begin{equation}
\label{eqn:def of z_S}
z_{t,\CS}^k\triangleq\begin{bmatrix}
x_t^{k\top} & u_{t,\CS}^{k\top}
\end{bmatrix}^{\top},
\end{equation}
for all $k\in[N]$, all $t\in\{0,\dots,T-1\}$ and all $\CS\subseteq\G$. For any $j\in[p]$, $\hat{\Theta}_{\G_j}$ can be viewed as an estimate of $\Theta_{\G_j}\triangleq\begin{bmatrix}A & B_{\G_j}\end{bmatrix}$ \cite{abbasi2011regret,cassel2020logarithmic}. Thus, we can obtain estimates of $A$ and $B$, i.e., $\hat{A}$ and $\hat{B}$, respectively, according to line~9 in Algorithm~\ref{alg:online for AS}. %{\color{black} As we will show in Section~\ref{sec:regret analysis}, by setting the system identification phase (i.e., $\tau_1p$) to be sufficiently long, we can ensure that the estimation error of $\hat{A},\hat{B}$ is small enough such that the results proved in Section~\ref{sec:ce LQR} can be applied.}

{\bf Control phase:} For any episode $k\in\{N_{p+1},\dots,N\}$ in lines~11-16 of Algorithm~\ref{alg:online for AS}, the algorithm calls the {\bf Exp3.S} subroutine to select a set $\CS^k$ of actuators, and invokes the certainty equivalence approach described in Section~\ref{sec:ce LQR} to design $u_{t,\CS^k}^k=\hat{K}_{t,\CS^k}^kx_t^k$, $\forall t\in\{0,\dots,T-1\}$, {\color{black} where $\hat{K}_{t,\CS^k}^k$ is computed by Eq.~\eqref{eqn:control gain S_t est} using the estimates $\hat{A},\hat{B}$ obtained from the system identification phase.} {\color{black} Here, the {\bf Exp3.S} subroutine is applied to the MAB instance, where the total number of episodes is $N_s=N-N_{p+1}+1$, the set of all possible actions is $\Q=\{\CS\subseteq\G:|\CS|=H\}$, and the cost associated with each possible action $\CS\in\Q$ in episode $k$ is $y_{\CS}^k=J_k(\CS,u_{\CS}^k)$ defined in Eq.~\eqref{eqn:episode t cost}, where $u_{\CS}^k=(u_{0,\CS}^k,\dots,u_{T-1,\CS}^k)$ with $u_{t,\CS}^k=\hat{K}_{t,\CS}x_t^k$. Thus, each arm in the MAB instance corresponds to a set of actuators with cardinality $H$.} 

Finally, one can check that the running time of each episode in Algorithm~\ref{alg:online for AS} is $O((n+m)^3T+|\Q|T)$, where the factor $(m+n)^3$ is due to the computation of Eq.~\eqref{eqn:control gain S_t est}. Since  $|\Q|={|\G|\choose H}$, Algorithm~\ref{alg:online for AS} is efficient for instances of Problem~\eqref{eqn:LQR obj} with either $|\G|$ (i.e., the total number of candidate actuators) or $H$ (i.e., the cardinality constraint on the set of selected actuators) to be small (or bounded by a constant). Nonetheless, we will later extend our algorithm design to efficiently handle large-scale instances of Problem~\eqref{eqn:LQR obj} in Section~\ref{sec:large scale instances}. 

%In summary, Algorithm~\ref{alg:online for AS} first explores to obtain estimates of $A$ and $B$, and then exploits the estimates of $A$ and $B$ to select actuators and design the corresponding controller. Hence, it is crucial for Algorithm~\ref{alg:online for AS} to balance between exploration and exploitation. This is achieved by carefully designing the number of episodes for exploration and exploitation through the choices of $N_{j}$ in lines~3-5 of Algorithm~\ref{alg:online for AS}.

\begin{remark}
\label{remark:episodic setting oblivious}
Note from Eq.~\eqref{eqn:episode t cost} that the cost of the action chosen by the {\bf Exp3.S} subroutine for any episode $k\in\{N_{p+1},\dots,N\}$ of Algorithm~\ref{alg:online for AS}, i.e., $J_k(\CS^k,u_{\CS^k}^k)$, does not depend on the previous actions $\CS^{N_{p+1}},\dots,\CS^{k-1}$ chosen by the {\bf Exp3.S} subroutine. Thus, we   know from Remark~\ref{remark:condition for Exp3} that the result in Lemma~\ref{thm:regret for exp3.s} can be applied when we analyze the regret of Algorithm~\ref{alg:online for AS} in the next section.
\end{remark}

\section{Regret Analysis for Episodic Setting}\label{sec:regret analysis}
In this section, we aim to provide high probability upper bounds on the regret of Algorithm~\ref{alg:online for AS} defined in Eq.~\eqref{eqn:regret of A_e} for the episodic setting of Problem~\eqref{eqn:LQR obj}. To this end, we first analyze the estimation error of the least squares approach given by \eqref{eqn:ls approach}. For any $j\in[p]$, we denote
\begin{equation}
\label{eqn:def of V}
V_{\G_j}=\lambda I+\sum_{k=N_{j}}^{N_{j+1}-1}\sum_{t=0}^{T-1}z_{t,\G_j}^{k}z_{t,\G_j}^{k\top},
\end{equation}
where $\G_j$ is given by Assumption~\ref{ass:system parameters}, $\lambda\in\R_{>0}$, $N_{j},N_{j+1}$ are given in Algorithm~\ref{alg:online for AS}, and $z_{t,\G_j}^{k}$ is given in Eq.~\eqref{eqn:def of z_S}. We then have the following result; the proof is similar to that of \cite[Lemma~6]{cohen2019learning} and is omitted here for conciseness.
\begin{lemma}
\label{lemma:est error of Theta_i hat}
Consider any $\G_j$ from Assumption~\ref{ass:system parameters}, where $j\in[p]$. Let $\Delta_{\G_j}=\Theta_{\G_j}-\hat{\Theta}_{\G_j}$, where $\Theta_{\G_j}=\begin{bmatrix}A & B_{\G_j}\end{bmatrix}$ and $\hat{\Theta}_{\G_j}$ is given by~\eqref{eqn:ls approach}. Suppose Assumption~\ref{ass:noise process} holds. Then, for any $\delta\in\R$ with $0<\delta<1$, with probability at least $1-\delta$, it holds
\begin{equation*}
\label{eqn:est error lemma upper bound}
\Tr(\Delta_{\G_j}^{\top}V_{\G_j}\Delta_{\G_j})\le4\sigma^2n\log\Big(\frac{n\det(V_{\G_j})}{\delta\det(\lambda I)}\Big) + 2\lambda\norm{\Theta_{\G_j}}_F^2.
\end{equation*}
\end{lemma}
For notational simplicity in the sequel, we further denote
\begin{align}\nonumber
&\vartheta=\max\{\norm{A},\norm{B}\},\ \varepsilon_0=\min_{\CS\subseteq\G,|\CS|=H}\frac{1-\eta_{\CS}}{\zeta_{\CS}\mu_{\CS}},\\\nonumber
&\zeta=\max\big\{\max_{\CS\subseteq\G,|\CS|=H}\zeta_{\CS},\zeta_0\big\},\ \eta=\max\big\{\max_{\CS\subseteq\G,|\CS|=H}\eta_{\CS},\eta_0\big\},\\\nonumber
&\kappa=\max\Big\{\max_{\CS\subseteq\G,|\CS|=H}\Big(\Gamma_{\CS}+\frac{1-\eta_{\CS}}{2\norm{B_{\CS}}\zeta_{\CS}}\Big),\eta_0\Big\},\\
&\Gamma=\max_{\CS\subseteq\G,|\CS|=H}\Gamma_{\CS},\ \tilde{\Gamma}=\Gamma+1, \label{eqn:parameters}
\end{align}
where $\tilde{\Gamma}_{\CS}$ (resp., $\Gamma_{\CS}$) is defined in Eq.~\eqref{eqn:def of Gamma S_t tilde} (resp., \eqref{eqn:def of Gamma S_t}), $\zeta_{\CS}$ and $\eta_{\CS}$ are provided in Lemma~\ref{lemma:bound on norm of Phi},  and $\mu_{\CS}$ is defined in Eq.~\eqref{eqn:mu_S}.\footnote{Under Assumption~\ref{ass:controllability}, one can check via the definition of $\Gamma_{\CS}$ in Eq.~\eqref{eqn:def of Gamma S_t} that $\Gamma$ is independent of $T$ (see, e.g., \cite{anderson1981detectability,zhang2021regret}).} We then have the following result for the regret of Algorithm~\ref{alg:online for AS} defined in Eq.~\eqref{eqn:regret of A_e}, where the results holds for any general benchmark $\CS_{\star}=(\CS_{\star}^1,\dots,\CS_{\star}^N)$ described in Remark~\ref{remark:dynamic regret}. %The main idea of the proof is to first decompose the regret $R_{\A_e}$ of Algorithm~\ref{alg:online for AS} properly, and then upper bound the terms in the decomposition separately.
\begin{theorem}
\label{thm:regret of as alg}
Suppose that Assumptions~\ref{ass:cost matrices}-\ref{ass:system parameters} hold. Consider any $\delta\in\R_{>0}$ with $0<\delta<1$. Denote 
\begin{equation}
\label{eqn:tau_0}
\tau_0=\frac{160np\Big(\frac{\lambda\vartheta^2}{\sigma^2}+2(n+m)\log\big(\frac{8n}{\delta}(p+\frac{TNz_b}{\lambda})\big)\Big)}{T-1},
\end{equation}
where
\begin{equation}
\label{eqn:z_b}
z_b=\frac{20\zeta_0^2(1+\eta_0)^2\sigma^2}{(1-\eta_0)^2}\big(2(\vartheta^2+1)\eta_0^2m+n\big)\log\frac{8TN}{\delta}.
\end{equation}
In Algorithm~\ref{alg:online for AS}, let 
\begin{align}
&\tau_1=\Big\lceil\max\Big\{\sqrt{N},\frac{\tau_0}{\varepsilon_0^2}\Big\}\Big\rceil,\label{eqn:tau_1}\\
&\bar{y}_b= T(2\sigma_Q+\kappa^2\sigma_R)\frac{4\zeta^2\sigma^2}{(1-\eta)^2}5n\log\frac{8TN}{\delta}.\label{eqn:y_b bar}
\end{align}
Then, for any $N>\tau_1p$, with probability at least $1-\delta$,
\begin{align}
&R_{\A_e}= \tilde{O}(n(m+n)^2p\sqrt{T^2|\Q|h(\CS_{\star})N}),\label{eqn:R_Ae}
\end{align}
where $R_{\A_e}$ is defined in Eq.~\eqref{eqn:regret of A_e}, $h(\CS_{\star})$ is defined in Eq.~\eqref{eqn:number of switchings}, $\Q=\{\CS\subseteq\G:|\CS|=H\}$, and $\tilde{O}(\cdot)$ hides polynomial factors in $\log(|\Q|N),\log((m+n)TN/\delta),\sigma_R,\sigma_Q,\sigma,\kappa,\zeta,\tilde{\Gamma},\ell,\tilde{\beta},(1-\eta)^{-1},\nu^{-1}$, where $\tilde{\beta}=1+\norm{A}$,  $\nu\in\R_{>0},\ell\in[n-1]$ are given in Assumption~\ref{ass:controllability}, and $\sigma_R,\sigma_Q$ are defined in \eqref{eqn:max sigma_1 Q and R}.
\end{theorem}

\subsection{Proof of Theorem~\ref{thm:regret of as alg}}\label{sec:proof of thm 1}
Recalling lines~5-7 in Algorithm~\ref{alg:online for AS}, for any $j\in[p]$ and any $k\in\{N_j,\dots,N_{j+1}-1\}$, one can show that the state of system~\eqref{eqn:LTI with B_i} satisfies that
\begin{equation}
\label{eqn:dynamics of x rewrite}
x_{t+1}^{k}=(A+B_{\G_j}K_{\G_j})x_t^{k}+B_{\G_j}\tilde{w}_t^{k}+w_t^{k},
\end{equation}
for all $t\in\{0,\dots,T-1\}$, where $x_0^{k}=0$, $K_{\G_j}$ is given by Assumption~\ref{ass:system parameters}, and $\tilde{w}_t^{k}\overset{\text{i.i.d.}}{\sim}\CN(0,2\sigma^2\eta_0^2I)$. Also note that $w_t^{k}$ is independent of $w_{t}^{k}$ as we assumed before. For notational simplicity in this proof, denote
\begin{align}
&\tilde{\K}=\{k:N_{j}\le k\le N_{j+1}-1,j\in[p]\}\label{eqn:T_tilde},\\
&\K=[N]\setminus\tilde{\K}=\{N_{p+1},\dots,N\}.\label{eqn:T}
\end{align}
In words, the set $\tilde{\K}$ (resp., $\K$) contains the indices of episodes for the system identification (resp., control) phase in Algorithm~\ref{alg:online for AS}.

Note that $(\CS^1,\dots,\CS^N)$ denotes the sequence of the sets of actuators selected by Algorithm~\ref{alg:online for AS}. From Eq.~\eqref{eqn:regret of A_e}, one can decompose the regret %of Algorithm~\ref{alg:online for AS} 
as $R_{\A_e} = R_e^1 + R_e^2 + R_e^3 + R_e^4$ with
\begin{align*}
R_e^1 &= \E_{\A_e}\Big[\sum_{k\in\tilde{\K}}J_k(\CS^k,u_{\CS^k}^{k})\Big]-\sum_{k\in\tilde{\K}}J_k(\CS^{k}_{\star}),\\
R_e^2 &= \E_{\A_e}\Big[\sum_{k\in\K}J_k(\CS^k,u_{\CS^k}^{k})\Big]-\sum_{k\in\K}J_k(\CS^{k}_{\star},u_{\CS^{k}_{\star}}^{k}),\\
R_e^3 &= \sum_{k\in\K}\big(J_k(\CS^{k}_{\star},u_{\CS^{k}_{\star}}^{k})-\hat{J}_k(\CS^{k}_{\star})\big),\\
R_e^4 &= \sum_{k\in\K}\big(\hat{J}_k(\CS^{k}_{\star})-J_k(\CS^{k}_{\star})\big),
\end{align*}
where $J_k(\CS^k,u_{\CS^k}^{k})$ is defined in Eq.~\eqref{eqn:episode t cost} with $u_{\CS^k}^{k}$ given by Algorithm~\ref{alg:online for AS}, and $J_k(\CS^{k}_{\star})$ (resp., $\hat{J}_k(\CS^k_{\star})$) is given by \eqref{eqn:opt LQR cost S_t} (resp., \eqref{eqn:exp for J_t hat}). Note that $R_e^2$, $R_e^3$ and $R_e^4$ together correspond to the regret incurred by the exploitation phase in Algorithm~\ref{alg:online for AS}, and $R_e^1$ corresponds to the regret incurred by the system identification phase in Algorithm~\ref{alg:online for AS}. In particular, $R_e^2$ corresponds to the {\bf Exp3.S} subroutine, and $R_e^3,R_e^4$ correspond to the certainty equivalent control subroutine. 

In order to prove the (high probability) upper bound on $R_{\A_e}$, we will provide upper bounds on $R_e^1$, $R_e^2$, $R_e^3$, and $R_e^4$ separately in the sequel. First, considering any $0<\delta<1$, we define the following probabilistic events:
\begin{align*}
\CE_w&=\Big\{\norm{w_{t-1}^{k}}\le\sigma\sqrt{5n\log\frac{8TN}{\delta}},\forall k\in[N],\forall t\in[T]\Big\},\\
\CE_{\tilde{w}}&=\Big\{\norm{\tilde{w}_{t-1}^{k}}\le\eta_0\sigma\sqrt{10m\log\frac{8TN}{\delta}},\forall k\in\tilde{\K},\forall t\in[T]\Big\},\\
\CE_{\Theta}&=\Big\{\Tr(\Delta_{\G_j}^{\top}V_{\G_j}\Delta_{\G_j})\le4\sigma^2n\log\Big(\frac{8np\det(V_{\G_j})}{\delta\det(\lambda I)}\Big) \\
&\qquad\qquad\qquad\qquad\qquad+2\lambda\norm{\Theta_{\G_j}}_F^2,\forall j\in[p]\Big\},\\
\CE_z&=\Big\{\sum_{k=N_{j}}^{N_{j+1}-1}\sum_{t=0}^{T-1}z_{t,\G_j}^{k}z_{t,\G_j}^{k\top}\succeq\frac{(T-1)\tau_1\sigma^2}{80}I,\forall j\in[p]\Big\}.
\end{align*}
Letting 
\begin{equation}
\label{eqn:event E}
\CE = \CE_w\cap\CE_{\tilde{w}}\cap\CE_{\Theta}\cap\CE_z,   
\end{equation}
we have the following result which shows that $\CE$ holds with high probability; the proof can be found in Appendix~\ref{app:regret bound proofs}.
\begin{lemma}
\label{lemma:event E high prob}
For any $0<\delta<1$, the event $\CE$ defined in Eq.~\eqref{eqn:event E} satisfies $\Prob(\CE)\ge1-\delta/2$.
\end{lemma}
Hence, we will provide upper bounds on $R_e^1$, $R_e^2$, $R_e^3$ and $R_e^4$, under the event $\CE$ defined in Eq.~\eqref{eqn:event E}. The following result characterizes the estimation error of $\hat{\Theta}_{\G_j}$, for all $j\in[p]$; {\color{black}the proof can be found in Appendix~\ref{app:regret bound proofs}. Lemma~\ref{lemma:est error of Theta_Gj(i)} shows that setting the system identification phase (i.e., $\tau_1p$) to be sufficiently long (i.e., Eq.~\eqref{eqn:tau_1}) ensures that the estimation error of $\hat{A},\hat{B}$ is small enough such that the results proved in Section~\ref{sec:ce LQR} can be applied to bound $R_e^3,R_e^4$ corresponding to the certainty equivalence subroutine in Algorithm~\ref{alg:online for AS}.}
\begin{lemma}
\label{lemma:est error of Theta_Gj(i)}
Consider any $0<\delta<1$, and suppose that the event $\CE$ holds. For any $j\in[p]$, it holds that $\norm{\hat{\Theta}_{\G_j}-\Theta_{\G_j}}^2\le\min\{\frac{\varepsilon_0^2}{p},\frac{\tau_0}{\sqrt{N}}\}$, where $\Theta_{\G_j}=\begin{bmatrix}A & B_{\G_j}\end{bmatrix}$.
\end{lemma}
We then have the following bounds on $R_e^1,R_e^2,R_e^3,R_e^4$; all the proofs are included in Appendix~\ref{app:regret bound proofs}. {\color{black}In particular, to bound $R_e^2$, we use $\bar{y}_b$ to normalize (i.e., upper bound) the cost of each episode in the control phase of Algorithm~\ref{alg:online for AS} so that the {\bf Exp3.S} subroutine and Lemma~\ref{thm:regret for exp3.s} can be applied.}
\begin{lemma}
\label{lemma:upper bound on R_e^1}
Under the event $\CE$, it holds that 
\begin{equation}
\label{eqn:upper bound on R_1}
R_e^1\le\max\{\sigma_Q,\sigma_R\}\frac{\tau_1p(2\eta_0^2+3)T\zeta_0^2}{(1-\eta_0)^2}\times\big(20\vartheta^2\eta_0^2\sigma^2m+10\sigma^2n\big)\log\frac{8TN}{\delta},
\end{equation}
where $\eta_0,\zeta_0$ are given by Assumption~\ref{ass:system parameters}.
\end{lemma}
\begin{lemma}
\label{lemma:upper bound on R_e^2}
Under the event $\CE$, it holds that 
\begin{align}
R_e^2\le\bar{y}_b 2\sqrt{e-1}\sqrt{|\Q|N(h(\CS_{\star})\log(|\Q|N)+e)}.\label{eqn:upper bound on R_2}
\end{align}
\end{lemma}
\begin{lemma}
\label{lemma:upper bound on R_e^3}
Under the event $\CE$, the following holds with probability at least $1-\delta/2$:
\begin{align}\nonumber
R_e^3&\le 64\sqrt{TN}\frac{\sigma^2(\sigma_Q+\sigma_R\kappa^2)\vartheta\zeta^3}{(1-\eta^2)(1-\eta)}\sqrt{5n\log\frac{8TN}{\delta}}\\
&\quad+32(\sigma_Q+\sigma_R\kappa^2)\frac{\zeta^2}{1-\eta^2}\sigma^2\sqrt{TN(\log\frac{16TN}{\delta})^3}.\label{eqn:upper bound on R_3}
\end{align}
\end{lemma}
\begin{lemma}
\label{lemma:upper bound on R_e^4}
Under the event $\CE$, it holds that 
\begin{align}\nonumber
&R_e^4\le \frac{4\max_{\CS\subseteq\G,|\CS|=H}\{n,m_{\CS}\}T\zeta^2\tau_0\sqrt{N}}{(1-\eta^2)}\sigma(\sigma_R+\Gamma^3)\\
&\times\big(3\tilde{\Gamma}^6(20\tilde{\Gamma}\sigma_R)^{\ell-1}32\ell^{\frac{5}{2}}\tilde{\beta}^{2(\ell-1)}(1+\nu^{-1})\max\{\sigma_Q,\sigma_R\}\big)^2.\label{eqn:upper bound on R_4}
\end{align}
\end{lemma}
Since $\Prob(\CE)\ge1-\delta/2$ from Lemma~\ref{lemma:event E high prob}, we can further apply a union bound and obtain an upper bound on $R_{\A_e}$ that holds with probability at least $1-\delta$. Specifically, one can show using \eqref{eqn:upper bound on R_1}-\eqref{eqn:upper bound on R_4} that $R_e^1=\tilde{O}(n(m+n)^2p^2T\sqrt{N})$, $R_e^2=\tilde{O}(nT\sqrt{|Q|Nh(\CS_{\star})})$, $R_e^3=\tilde{O}(\sqrt{nTN})$, and $R_e^4=\tilde{O}(n(n+m)^2T\sqrt{N})$, combining which implies \eqref{eqn:R_Ae} and completes the proof of Theorem~\ref{thm:regret of as alg}.\hfill\qed

\subsection{Discussions about the Results in Theorem~\ref{thm:regret of as alg}}\label{sec:discussions}
{\bf Length of the system identification phase:} Recall that Eq.~\eqref{eqn:tau_1} specifies the minimum length of the system identification phase in Algorithm~\ref{alg:online for AS} (i.e., $\tau_1p$). To gain more insights on how $\tau_0,\tau_1$ depend on other problem parameters, letting the regularization in the least square approach be $\lambda\ge z_b$, and supposing $TN\ge n$ and $TN\ge p$, one can show
\begin{align}
\frac{\tau_0}{\varepsilon_0^2}=O(1)\frac{\zeta^4\eta_0^4(\vartheta^2+1)^2}{(1-\eta)^4(T-1)}n(m+n)\ell^5\tilde{\Gamma}^6\tilde{\beta}^{4\ell-4}\max\{\sigma_R^2,\sigma_Q^2\}(1+\nu^{-1})^2\log\frac{NT}{\delta},\label{eqn:tau_1 rewrite}
\end{align}
where $O(1)$ is a universal constant. Since $\log(NT/\delta)=o(T)$, we see from Eq.~\eqref{eqn:tau_1 rewrite} that a larger value of $T$, i.e., the number of time steps in each episode $k\in[N]$ implies a smaller lower bound on $\tau_1$. Thus, the regret bound in Theorem~\ref{thm:regret of as alg} holds for $N>\tau_1 p$, which can be shown to be equivalent to $N$ being greater than a polynomial in the problem parameters. {\color{black}If $N\ge\tau_0^2/\varepsilon_0^4$, $\tau_1=\lceil\max\{\sqrt{N},\tau_0/\varepsilon_0^2\}\rceil$ reduces to $\tau_1=\lceil\sqrt{N}\rceil$.}

{\color{black} {\bf Knowledge about the unknown system:} One can check that the choices of $\tau_1,\bar{y}_b$ require knowledge of $\sigma_Q,\sigma_R,\zeta_{\CS},\eta_{\CS},\sigma^2,\zeta_0,\eta_0,\vartheta,\ell,\nu,\Gamma_{\CS}$ (for all $\CS\subseteq\G$ with $|\CS|=H$), where $\sigma_Q,\sigma_R,\sigma$ are given by our assumptions on the cost matrices and noise covariance, and $\zeta_0,\eta_0$ (resp., $\ell,\nu$) are given by Assumption~\ref{ass:system parameters} (resp., Assumption~\ref{ass:controllability}). The other parameters may also be computed (or bounded) given some knowledge of the unknown system. First, as shown in \cite{zhang2021regret}, for any $\CS\subseteq\G$ with $|\CS|=H$, $\eta_{\CS}$ and $\zeta_{\CS}$ can be expressed as $\eta_{\CS}=\sqrt{1-1/\max_{t\in[T],k\in[N]}\norm{P_{t,\CS}^k}}$ and $\zeta_{\CS}=\sqrt{\max_{t\in[T],k\in[N]}\norm{P_{t,\CS}^k}}$. For any $t\in[T]$ and any $k\in[N]$, Eq.~\eqref{eqn:control gain} yields $\norm{K_{t-1,\CS}^k}\le\vartheta^2\norm{P_{t,\CS}^k}$, and one can further upper bound $\norm{P_{t,\CS}^k}$ given any stabilizing controller $K_{\CS}$ (corresponding to the set of actuators $\CS$) (see Lemma~\ref{lemma:bound on P} in Appendix~\ref{app:tech lemmas}). Thus, by the definition of $\Gamma_{\CS}$ in Eq.~\eqref{eqn:def of Gamma S_t}, to compute (or bound) $\eta_{\CS},\zeta_{\CS},\Gamma_{\CS}$, we need to know (or upper bound) $\max_{t\in[T],k\in[N]}\norm{P_{t,\CS}^k}$ and know the upper bound $\vartheta$ on $\norm{A}$ and $\norm{B}$.} 

{\color{black}{\bf Output of Algorithm~\ref{alg:online for AS}:} Since Algorithm~\ref{alg:online for AS} uses the {\bf Exp3.S} subroutine to select the sets of actuators in the control phase of Algorithm~\ref{alg:online for AS}, as we argued in Section~\ref{sec:bandit problem}, {\bf Exp3.S} produces a (random) sequence of subsets of selected actuators $\CS_{N_{p+1}},\dots,\CS_N$ that can be different across the episodes, which ensures exploring new sets of actuators that have not been chosen before and exploiting the set of actuators that yield the lowest cost up until the current episode. As we showed in Section~\ref{sec:proof of thm 1}, such a sequence of subsets of selected actuators yields a $\sqrt{N}$-regret bound on $R_{e}^2$.}

{\bf Factors in the regret  bound:} First, the regret bound in Theorem~\ref{thm:regret of as alg} contains the $\sqrt{T^2N}$ factor. {\color{black}Although $\sqrt{T^2N}$ is not sublinear in the total number of time steps in the episodic setting of Problem~\eqref{eqn:LQR obj} (i.e., $TN$), it matches with the optimal regret bound (in terms of the scaling of $T,N$) that can be achieved by any model-based algorithms for general episodic reinforcement learning problems \cite{azar2017minimax}. If $T=o(\sqrt{N})$ (i.e., the number of time steps in each episode is small relative to the total number of episodes), the factor $\sqrt{T^2N}$ will become sublinear in $N$.} Second, the regret bound contains an exponential factor in $\ell$. As we argued before, $\ell<<n$ if $\text{rank}(B_{\CS})$ is large. In particular, $\ell=1$ if $\text{rank}(B_{\CS})=n$ (for any $\CS\subseteq\G$ with $|\CS|=H$). Third, since $R_{\A_e}$ defined in Eq.~\eqref{eqn:regret of A_e} is a dynamic regret as we argued in Remark~\ref{remark:dynamic regret}, the regret bound in \eqref{eqn:R_Ae} contains the factor $\sqrt{h(\CS_{\star})}$, where $h(\CS_{\star})$ measures the number of switchings in the benchmark $\CS_{\star}=(\CS_{\star}^1,\dots,\CS_{\star}^N)$. Such a factor of $h(\CS_{\star})$ is typical in the bounds on the dynamic regret of online algorithms \cite{auer2002nonstochastic,zhang2018dynamic,zinkevich2003online}. If the static regret described in Remark~\ref{remark:dynamic regret} is considered, then $h(\CS_{\star})=1$. Finally, the regret bound contains the factor $\sqrt{|\Q|}$ with $|\Q|={|\G|\choose H}$, which will not be a bottleneck if either of $|\G|$ or $H$ is small or bounded by a constant. In fact, a factor of $|\Q|={|\G|\choose H}$ is unavoidable in the regret of any online algorithm defined in Eq.~\eqref{eqn:regret of A_e} for Problem~\eqref{eqn:LQR obj}, since Problem~\eqref{eqn:LQR obj} is an NP-hard combinatorial optimization problem \cite{streeter2008online,ye2020complexity}. In Section~\ref{sec:large scale instances}, we will show how to extend our algorithm design and regret analysis to handle large-scale instances of Problem~\eqref{eqn:LQR obj}.

\section{Algorithm Design for non-episodic Setting and Regret Analysis}\label{sec:alg for non-episodic setting}
In this section, we consider the non-episodic setting of Problem~\eqref{eqn:LQR obj} described in Section~\ref{sec:online AS non-episodic}. For any $\CS_t\subseteq\G$ and any $t\in\{0,\dots,T-1\}$, we see from Eqs.~\eqref{eqn:cost per time step}-\eqref{eqn:final time step cost} that the cost $c_t(\CS_{0:t},u_{\CS_{0:t}})$ of time step $t\in\{0,\dots,T-1\}$ depends on $\CS_0,\dots,\CS_{t-1}$ via the state $x_t$. %In words, resetting the state of system~\eqref{eqn:LTI} to $x_0^k=0$ at the beginning of each episode $k\in[N]$ in the episodic setting removes the influence of the sets of actuators $\CS^1,\dots,\CS^{k-1}$ selected in the previous episodes. 
Hence, Remark~\ref{remark:condition for Exp3} implies that the {\bf Exp3.S} algorithm and the result in  Lemma~\ref{thm:regret for exp3.s} cannot be directly applied to solve the non-episodic setting of Problem~\eqref{eqn:LQR obj} in the same way as Algorithm~\ref{alg:online for AS}, which creates the major challenge when we move from the episodic setting to the non-episodic setting.

Nonetheless, given a non-episodic instance of Problem~\eqref{eqn:LQR obj} described in Section~\ref{sec:online AS non-episodic}, one may construct an episodic instance of Problem~\eqref{eqn:LQR obj} as follows. First, we group the time steps $0,\dots,T-1$ in the non-episodic instance of Problem~\eqref{eqn:LQR obj} into $N^{\prime}=\lfloor T/T^{\prime}\rfloor$ consecutive episodes with length $T^{\prime}\in\Z_{\ge1}$, where the $k$th episode starts at $t=(k-1)T^{\prime}$ and ends at $t=kT^{\prime}-1$. In each episode $k\in[N^{\prime}]$, we fix the set of selected actuators, i.e., we let $\CS_{(k-1)T^{\prime}}=\cdots=\CS_{kT^{\prime}-1}=\CS^k$, where $\CS^k\subseteq\G$ with $|\CS^k|=H$.\footnote{For simplicity, we assume that $T^{\prime}N^{\prime}=T$; otherwise, we can modify the number of time steps in the last episode.}  We then follow the notations introduced for the episodic setting in the previous sections. Specifically, we may write the state, control and disturbance at any time step $t\in\{0,\dots,T^{\prime}-1\}$ in any episode $k\in[N^{\prime}]$ as $x_t^k$, $u_{t,\CS^k}^k$ and $w_t^k$, respectively, e.g., $x_{(k-1)T^{\prime}+t}=x_t^k$ and $x_0^k=x_{T^{\prime}}^{k-1}$ with $x_0^1=0$. Note that the initial state $x_0^k$ of any episode $k\in[N^{\prime}]$ is not reset to $0$ in the episodic instance of Problem~\eqref{eqn:LQR obj} constructed above, and that $x_0^k=x_{T^{\prime}}^{k-1}$ depends on $\CS^{1},\dots,\CS^{k-1}$. For any episode $k\in[N^{\prime}]$, the cost matrices are set to be $Q^k=Q$, $R^k=R$, $Q_f^k=0$ if $k<N^{\prime}$ and $Q_f^k=Q_f$ if $k=N^{\prime}$. Similarly to Eq.~\eqref{eqn:episode t cost}, we denote the cost of episode $k$ as $J_{k}(\CS^{k},u_{\CS^{k}}^{k})$, where for notational simplicity we hide the dependency of $J_{k}(\CS^{k},u_{\CS^{k}}^{k})$ on the sets of actuators $\CS^1,\dots,\CS^{k-1}$ selected before episode $k$. One can now apply Algorithm~\ref{alg:online for AS} to the episodic instance constructed above; the detailed steps are summarized in Algorithm~\ref{alg:online for AS non-episodic}. {\color{black}Similarly, the {\bf Exp3.S} subroutine in Algorithm~\ref{alg:online for AS non-episodic} is applied to the MAB instance, where the total number of episodes in {\bf Exp3.S} is $N_s=N-N_{p+1}^{\prime}+1$, the set of all possible actions is $\Q=\{\CS\subseteq\G:|\CS|=H\}$, and the cost associated with each possible action $\CS\in\Q$ in episode $k$ is $y_{\CS}^k=\frac{1}{T^{\prime}}J_k(\CS^k,u_{\CS}^k)$.}

\begin{algorithm2e}[h]
\SetNoFillComment
\caption{Non-episodic Setting}\label{alg:online for AS non-episodic}
\KwIn{Parameters $\tau_1^{\prime},\lambda,N^{\prime},T^{\prime},\bar{y}_b^{\prime}$, and $K_{\G_j}$ for all $j\in[p]$ from Assumption~\ref{ass:system parameters}.}
{\color{black} Initialize $N_1^{\prime}=1$.}\\
{\color{black}\For{$j=1$ to $p$}{
    Set $N^{\prime}_{j+1}\gets N^{\prime}_{j}+\tau_1^{\prime}$.
}}
{\color{black} Set $T\gets T^{\prime},N\gets N^{\prime},N_j\gets N_j^{\prime}\ \forall j\in[p+1],\bar{y}_b\gets\bar{y}_b^{\prime}$; follow lines~4-16 in Algorithm~\ref{alg:online for AS}, where the cost $y_{\CS_k}^k$ in line~16 is changed to be $y_{\CS^k}^k=\frac{1}{T^{\prime}}J_k(\CS^k,u_{\CS^k}^k)$.}\\
{\color{black}\KwOut{$\CS_t,u_{t,\CS_t},\forall t\in\{0,\dots,T-1\}$.}}
\end{algorithm2e}

The intuition behind the above construction is that if we fix a set of actuators $\CS^k$ for $T^{\prime}$ time steps in an episode $k\in[N^{\prime}]$ (and design the corresponding control $u_{\CS^k}^k$ based on the certainty equivalence approach), then one can use Lemma~\ref{lemma:bound on Psi_hat} to show that the influence of the initial condition $x_0^k$ on the cost $c_t(\CS_{0:t},u_{\CS_{0:t}})$ (defined in Eqs.~\eqref{eqn:cost per time step}-\eqref{eqn:final time step cost}) at any time step $t\in\{(k-1)T^{\prime},\dots,kT^{\prime}-1\}$ in episode $k$ decays exponentially as $t$ increases. This in turn implies that $c_t(\CS_{0:t},u_{\CS_{0:t}})$ tends to be independent of $\CS^{1},\dots,\CS^{k-1}$ selected before episode $k\in[N^{\prime}]$, and we can then adopt the analysis developed in Section~\ref{sec:regret analysis} for the episodic setting.

In the sequel, we interchangeably write $\CS_{0},\dots,\CS_{T-1}$ and $\CS^1,\dots,\CS^{N^{\prime}}$, $u_{0,\CS_0},\dots,u_{T-1,\CS_{T-1}}$ and $u_{\CS^1}^1,\dots,u_{\CS^{N^{\prime}}}^{N^{\prime}}$, and $x_0,\dots,x_{T-1}$ and $x^{0}_1,\dots,x_{T^{\prime}-1}^0,\dots,x^{N^{\prime}}_{0},\dots,x^{N^{\prime}}_{T^{\prime}-1}$. Similarly to Theorem~\ref{thm:regret of as alg}, we prove the following result (Theorem~\ref{thm:regret of as alg non-episodic}) for the regret $R_{\A_c}$ of Algorithm~\ref{alg:online for AS non-episodic} defined in Eq.~\eqref{eqn:regret of A_c}, where we use the notations introduced in \eqref{eqn:def of Gamma S_t}-\eqref{eqn:max sigma_1 Q and R} and \eqref{eqn:parameters} (with $N=N^{\prime}$ and $T=T^{\prime}$). Note that since we let $\CS_{(k-1)T^{\prime}}=\cdots=\CS_{kT^{\prime}-1}=\CS^k$ in any episode $k\in[N^{\prime}]$, we consider the benchmark $\CS^{\star}=(\CS_0^{\star},\dots,\CS_{T-1}^{\star})$ in Eq.~\eqref{eqn:regret of A_c} with $\CS^{\star}_{(k-1)T^{\prime}}=\cdots=\CS^{\star}_{kT^{\prime}-1}=\CS_{\star}^k$ for all $k\in[N^{\prime}]$, where $\CS$ is any $\CS_{\star}^k\subseteq\G$ with $|\CS_{\star}^k|= H$. The proof of Theorem~\ref{thm:regret of as alg non-episodic} follows by quantifying the dependency of $c_t(\CS_{0:t},u_{\CS_{0:t}})$ on $x_0^k$ as we described above, and carefully adapting the techniques from the proof of Theorem~\ref{thm:regret of as alg}.

%we prove the following result (Theorem~\ref{thm:regret of as alg non-episodic}) for the regret of Algorithm~\ref{alg:online for AS non-episodic} defined in Eq.~\eqref{eqn:regret of A_c}, where we use the notations introduced in \eqref{eqn:def of Gamma S_t}-\eqref{eqn:max sigma_1 Q and R} and \eqref{eqn:parameters} (with $N=N^{\prime}$ and $T=T^{\prime}$). Note that since we let $\CS_{(k-1)T^{\prime}}=\cdots=\CS_{kT^{\prime}-1}=\CS^k$ in any episode $k\in[N^{\prime}]$, we consider the benchmark $\CS^{\star}=(\CS_0^{\star},\dots,\CS_{T-1}^{\star})$ in Eq.~\eqref{eqn:regret of A_c} with $\CS^{\star}_{(k-1)T^{\prime}}=\cdots=\CS^{\star}_{kT^{\prime}-1}=\CS_{\star}^k$ for all $k\in[N^{\prime}]$, where $\CS_{\star}^k\subseteq\G$ and $|\CS_{\star}^k|= H$. The benchmark is alternatively denoted as $\CS_{\star}=(\CS^{1}_{\star},\cdots,\CS^{N^{\prime}}_{\star})$. The proof of Theorem~\ref{thm:regret of as alg non-episodic} follows by quantifying the dependency of $c_t(\CS_{0:t},u_{\CS_{0:t}})$ on $x_0^k$ as we described above, and applying similar techniques to those in the proof of Theorem~\ref{thm:regret of as alg}. Note that similar discussions to those in Section~\ref{sec:discussions} can be applie Theorem~\ref{thm:regret of as alg non-episodic}.
\begin{theorem}
\label{thm:regret of as alg non-episodic}
Suppose that Assumptions~\ref{ass:cost matrices}-\ref{ass:system parameters} hold. Let $T^{\prime}=\big\lceil\big(4(e-1)(h(\CS_{\star})\ln(|\Q|T)+e)|\Q|\big)^{-1/3}T^{1/3}\big\rceil$, and $N^{\prime}=\lceil T/T^{\prime}\rceil$, where $\Q=\{\CS\subseteq\G:|\CS|=H\}$, and $h(\CS_{\star})$ is defined in Eq.~\eqref{eqn:number of switchings}. Consider any $\delta\in\R_{>0}$ with $0<\delta<1$. Denote 
\begin{equation*}
\label{eqn:tau_0 non-episodic}
\tau_0^{\prime}=\frac{160np\Big(\frac{\lambda\vartheta^2}{\sigma^2}+2(n+m)\log\big(\frac{8n}{\delta}(p+\frac{Tz_b^{\prime}}{\lambda})\big)\Big)}{T^{\prime}-1},
\end{equation*}
where
\begin{equation*}
\label{eqn:z_b non-episodic}
z_b^{\prime}=\frac{180\zeta_0^4(1+\eta_0)^2\sigma^2}{(1-\eta_0)^2}\big(2(\vartheta^2+1)\eta_0^2m+n\big)\log\frac{8T}{\delta}.
\end{equation*}
In Algorithm~\ref{alg:online for AS non-episodic}, let $\tau_1^{\prime}\ge\Big\lceil\max\Big\{\sqrt{N^{\prime}},\frac{\tau_0^{\prime}}{\varepsilon_0^2}\Big\}\Big\rceil$ and
\begin{align}\nonumber
\bar{y}_b^{\prime}\ge(2\sigma_Q+\kappa^2\sigma_R)\frac{36\eta^4\sigma^2}{(1-\eta)^2}(20\vartheta^2\eta_0^2m+10n)\log\frac{8T}{\delta}.  
\end{align}
Then, for any $T>\tau_1^{\prime}pT^{\prime}$ with $T^{\prime}>T_m=\frac{2}{1-\eta}(\frac{1}{3}\log T+\log\zeta)>0$, the following holds with probability at least $1-\delta$:
\begin{align}
&R_{\A_c}= \tilde{O}(n(m+n)^2p^2\sqrt{|\Q|h(\CS_{\star})}T^{2/3}),\label{eqn:R_Ac}
\end{align}
where $R_{\A_c}$ is defined in Eq.~\eqref{eqn:regret of A_c}, $\Q=\{\CS\subseteq\G:|\CS|=H\}$, and $\tilde{O}(\cdot)$ hides polynomial factors in $\log(|\Q|N^{\prime}),\log((m+n)T/\delta),\sigma_R,\sigma_Q,\sigma,\kappa,\zeta,\tilde{\Gamma},\ell,\tilde{\beta},(1-\eta)^{-1},\nu^{-1}$, where $\tilde{\beta}=1+\norm{A}$,  $\nu\in\R_{>0},\ell\in[n-1]$ are given in Assumption~\ref{ass:controllability}, and $\sigma_R,\sigma_Q$ are defined in \eqref{eqn:max sigma_1 Q and R}.
\end{theorem}

\subsection{Proof of Theorem~\ref{thm:regret of as alg non-episodic}}\label{sec:proof of thm2}
The proof of Theorem~\ref{thm:regret of as alg non-episodic} is similar to that of Theorem~\ref{thm:regret of as alg}. In particular, we leverage a probabilistic event $\CE^{\prime}$ defined as \eqref{eqn:event E} (with $N=N^{\prime}$ and $T=T^{\prime}$). Recalling the definition of $R_{\A_c}$ given in Eq.~\eqref{eqn:regret of A_c}, one can decompose $R_{\A_c}=R_c^1+R_c^2+R_c^3+R_c^4$ with
\begin{align*}
R_c^1 &= \E_{\A_c}\Big[\sum_{k\in\tilde{\K}^{\prime}}J_k(\CS^{k},u_{\CS^{k}}^{k})\Big],\\
R_c^2 &= \E_{\A_c}\Big[\sum_{k\in\K^{\prime}}J_k(\CS^{k},u_{\CS^{k}}^{k})\Big]-\sum_{k\in\K^{\prime}}J_k(\CS^{k}_{\star},u_{\CS^{k}_{\star}}^{k}),\\
R_c^3 &= \sum_{k\in\K^{\prime}}\big(J_k(\CS^{k}_{\star},u_{\CS^{k}_{\star}}^{k})-\hat{J}_k(\CS^{k}_{\star})\big),\\
R_c^4 &= \sum_{k\in\K^{\prime}}\big(\hat{J}_k(\CS^{k}_{\star})\big)-J(\CS_{\star}),
\end{align*}
where $\tilde{\K}^{\prime}$ and $\K^{\prime}$ are defined in Eqs.~\eqref{eqn:T_tilde} and \eqref{eqn:T} (with $N=N^{\prime}$ and $T=T^{\prime}$), respectively, and $\hat{J}_k(\CS^k_{\star})$ is given by Eq.~\eqref{eqn:exp for J_t hat}. Similarly to Lemma~\ref{lemma:event E high prob}, we can show that $\Prob(\CE^{\prime})\ge1-\delta/2$ for any $0<\delta<1$. Moreover, similarly to Lemma~\ref{lemma:est error of Theta_Gj(i)}, we have the following result; the proof is included in Appendix~\ref{app:regret bound proofs non-episodic}.
\begin{lemma}
\label{lemma:est error of Theta_Gj(i) non-episodic}
Consider any $0<\delta<1$, and suppose that the event $\CE^{\prime}$ holds. For any $j\in[p]$, it holds that $\norm{\hat{\Theta}_{\G_j}-\Theta_{\G_j}}^2\le\min\{\frac{\varepsilon_0^2}{p},\frac{\tau_0^{\prime}}{\sqrt{N^{\prime}}}\}$, where $\Theta_{\G_j}\triangleq\begin{bmatrix}A & B_{\G_j}\end{bmatrix}$.
\end{lemma}

For notational simplicity in the sequel, let us denote 
\begin{equation}
\label{eqn:x_b}
x_b= \frac{6\zeta^2}{1-\eta}\sqrt{20\vartheta^2\eta_0^2\sigma^2m+10\sigma^2n}\sqrt{\log\frac{8T}{\delta}}.
\end{equation}
The following results show that the cost $c_t(\CS_{0:t},u_{\CS_{0:t}})$ at any time step $t\in\{(k-1)T^{\prime},\dots,kT^{\prime}-1\}$ in episode $k\in\K^{\prime}$ is almost independent of $\CS^{1},\dots,\CS^{k-1}$ selected before episode $k$, as we argued above. The proofs are included in Appendix~\ref{app:regret bound proofs non-episodic}.
\begin{lemma}
\label{lemma:norm bound on state}
Consider any $\CS^{1},\dots,\CS^{N^{\prime}}$ that satisfy $\CS^k=\G_j$ for all $k\in\{N^{\prime}_j,\dots,N^{\prime}_{j+1}-1\}$ and all $j\in[p]$, where $N^{\prime}_j$ is given as line~3 of Algorithm~\ref{alg:online for AS}, and $\CS^k\subseteq\G$ with $|\CS^k|=H$ for all $k\in\K^{\prime}$. Suppose the sequence of actuators $\CS^1,\dots,\CS^{N^{\prime}}$ is selected in the episodic instance of Problem~\eqref{eqn:LQR obj} that we constructed above, and the corresponding control $u_{t,\CS^k}^k$ is given by lines~7 and 13 in Algorithm~\ref{alg:online for AS} for any $k\in[N^{\prime}]$ and any $t\in\{0,\dots,T^{\prime}-1\}$. Then, $\norm{x_t^k}\le x_b$ for all $k\in[N^{\prime}]$ and all $t\in\{0,\dots,T^{\prime}\}$.
\end{lemma}

\begin{lemma}
\label{lemma:almost episodic}
Consider any sequence of actuators $\CS^1,\dots,\CS^{N^{\prime}}$ and the control $u_{t,\CS^k}^k$ described in Lemma~\ref{lemma:norm bound on state} for the episodic instance of Problem~\eqref{eqn:LQR obj} that we constructed. For any $k_0\in\K^{\prime}$, consider another sequence of actuators $\tilde{\CS}^1,\dots,\tilde{\CS}^{k_0}$, where $\tilde{\CS}^{k}=\CS^{k}$ for all $k\in\tilde{\K}^{\prime}\cup\{k_0\}$, and $\tilde{\CS}^{k}\subseteq\G$ with $|\tilde{\CS}^{k}|=H$ for all $k^{\prime}\in[k_0]\setminus(\tilde{\K}^{\prime}\cup\{k_0\})$. Suppose the control $u_{t,\tilde{\CS}^k}^k$ corresponding to $\tilde{\CS}^1,\dots,\tilde{\CS}^{k_0}$ is give by lines~7 and 13 in Algorithm~\ref{alg:online for AS} for any $k\in[k_0]$ and any $t\in\{0,\dots,T^{\prime}-1\}$. Then, under the event $\CE^{\prime}$,
\begin{equation*}
|c_t(\CS_{0:t},u_{\CS_{0:t}})-c_t(\tilde{\CS}_{0:t},u_{\tilde{\CS}_{0:t}})|\le8(\sigma_R+\kappa^2\sigma_Q)x_b^2T^{-1/3},
\end{equation*}
for all $t\in\{(k_0-1)T^{\prime}+T_m,\dots,k_0T^{\prime}-1\}$.
\end{lemma}

We are now in place to upper bound $R_1^c$, $R_2^c$, $R_3^c$ and $R_4^c$; the proofs are similar to those for the lemmas in Section~\ref{sec:proof of thm 1} and are included in Appendix~\ref{app:regret bound proofs non-episodic}.
\begin{lemma}
\label{lemma:upper bound on R_1^c}
Under the event $\CE^{\prime}$, it holds that 
\begin{equation*}
R_c^1\le\max\{\sigma_Q,\sigma_R\}\tau_1^{\prime}p(2\eta_0^2+3)T^{\prime}x_b^2.
\end{equation*}
\end{lemma}

\begin{lemma}
\label{lemma:upper bound on R_c^2}
Under the event $\CE^{\prime}$, it holds that 
\begin{equation*}
R_c^2\le8(\sigma_Q+\kappa^2\sigma_R)x_b^2T^{2/3}+(2\sigma_Q+\kappa^2\sigma_R)x_b^2(T_m+1)\big(4(e-1)(h(\CS_{\star})\ln(|\Q|T)+e)|\Q|\big)^{1/3}T^{2/3}.
\end{equation*}
\end{lemma}

\begin{lemma}
\label{lemma:upper bound on R_c^3}
Under the event $\CE^{\prime}$, the following holds with probability at least $1-\delta/2$:
\begin{align}
&R_c^3\le32\sqrt{T}\frac{\sigma(\sigma_Q+\sigma_R\kappa^2)\vartheta\zeta}{1-\eta^2}x_b+32(\sigma_Q+\sigma_R\kappa^2)\frac{\zeta^2}{1-\eta^2}\sigma^2\sqrt{T(\log\frac{16T}{\delta})^3}.\label{eqn:upper bound on R_c^3}
\end{align}
\end{lemma}

\begin{lemma}
\label{lemma:upper bound on R_c^4}
Under the event $\CE^{\prime}$, it holds that
\begin{align}\nonumber
&R_c^4\le (\sigma_Q+\sigma_R\kappa^2)\frac{4\zeta^2x_b^2}{1-\eta^2}+\frac{4\max_{\CS\subseteq\G,|\CS|=H}\{n,m_{\CS}\}T^{\prime}\zeta^2\tau_0^{\prime}\sqrt{N^{\prime}}}{(1-\eta^2)}\sigma(\sigma_R+\Gamma^3)\\\nonumber
&\qquad\qquad\times\big(3\tilde{\Gamma}^6(20\tilde{\Gamma}\sigma_R)^{\ell-1}32\ell^{\frac{5}{2}}\tilde{\beta}^{2(\ell-1)}(1+\nu^{-1})\max\{\sigma_Q,\sigma_R\}\big)^2.
\end{align}
\end{lemma}

Finally, similarly to the proof of Theorem~\ref{thm:regret of as alg}, we complete the proof of Theorem~\ref{thm:regret of as alg non-episodic} by combining Lemmas~\ref{lemma:upper bound on R_1^c}-\ref{lemma:upper bound on R_c^4} and noting the way we set $N^{\prime},T^{\prime}$ in Theorem~\ref{thm:regret of as alg non-episodic}.

\section{Handling Large-Scale Problem Instance}\label{sec:large scale instances}
We now extend our algorithm design and regret analysis to efficiently handle large-scale instances of Problem~\eqref{eqn:LQR obj}.
%recall from Section~\ref{sec:algorithm design for episodic setting} that Algorithm~\ref{alg:online for AS} uses a single {\bf Exp3.S} subroutine to select a set of actuators $\CS^k\in\Q$ with $\Q=\{\CS\subseteq\G:|\CS|=H\}$ for any $k\in\K$, where $\Q$ is the set of all possible actions in the {\bf Exp3.S} subroutine and $\K$ is defined in Eq.~\eqref{eqn:T}. 
We first consider the episodic setting of Problem~\eqref{eqn:LQR obj}. Leveraging the ideas from \cite{streeter2008online}, we propose to use $H$ statistically independent copies of the {\bf Exp3.S} subroutine in parallel, denoted as $M_1,\dots,M_H$, to choose the $H$ actuators in each episode. Detailed steps are summarized in Algorithm~\ref{alg:online for AS large scale}, where the system identification phase is the same as Algorithm~\ref{alg:online for AS}. {\color{black} We now explain how the {\bf Exp3.S} subroutines in Algorithm~\ref{alg:online for AS large scale} are used to select the actuators. Consider any $j\in[H]$ and any $k\in\{N_{p+1},\dots,N\}$. Let $s_j^k\in\G$ be the actuator selected by the {\bf Exp3.S} subroutine $M_j$ and let $\CS_{j}^{\prime k}=\{s_1^k,\dots,s_j^k\}$ with $\CS_0^{\prime k}=\emptyset$. Thus, $\CS^{\prime k}_H$ is the set of actuators selected by the $H$ {\bf Exp3.S} subroutines. The {\bf Exp3.S} subroutine $M_j$ is applied to the MAB instance, where the total number of episodes is $N_s=N-N_{p+1}+1$, the set of all possible actions is $\Q=\G$, and the cost associated with each possible action $s\in\Q$ in episode $k$ is 
\begin{equation}\label{eqn:cost of arms in large-scale}
y_{j,s}^k=J_k(\CS^{\prime k}_{j-1}\cup\{s\},u^k_{\CS^{\prime k}_{j-1}\cup\{s\}})-J_k(\CS^{\prime k}_{j-1},u_{\CS^{\prime k}_{j-1}}^k),
\end{equation}
where $J_k(\cdot,\cdot)$ is defined in Eq.~\eqref{eqn:episode t cost}. Note that in Algorithm~\ref{alg:online for AS large scale}, the actual cost that $M_j$ receives by selecting $s_j^k$ is given by
\begin{align}
\hat{y}^k_{j,s_j^k}\triangleq J_k(\CS^k,u^k_{\CS^k})\mathds{1}\{s=s_j^k,j=i,b_k=1\}\label{eqn:y_hat},
\end{align}
which can be different from the true cost $y_{j,s_j^k}^k$, where $b_k$ is a Bernoulli random variable with parameter $\rho$, and $i$ and $s$ are sampled from $[H]$ and $\G$ uniformly at random (u.a.r), respectively. Moreover, the actual set of actuators selected by Algorithm~\ref{alg:online for AS large scale} in episode $k$ (i.e., $\CS^k$) can also be different from $\CS^{\prime k}_H$ selected by the {\bf Exp3.S} subroutines, depending on $b_k$.}

\begin{algorithm2e}[h]
\SetNoFillComment
\caption{Large-Scale Problem Instance}\label{alg:online for AS large scale}
\KwIn{Parameters $\tau_1,\lambda,N,T,\bar{y}_b,\rho$, and $K_{\G_j}$ for all $j\in[p]$ from Assumption~\ref{ass:system parameters}.}
Follow lines~1-9 in Algorithm~\ref{alg:online for AS} to obtain $\hat{A},\hat{B}$.\\
{\color{black} Initialize $H$ independent {\bf Exp3.S} subroutines $M_1,\dots,M_H$ with $N_s=N-N_{p+1}+1$, $\Q=\G$, and $\alpha_1,\alpha_2$ according to Lemma~\ref{thm:regret for exp3.s}.}\\
{\color{black}\For{$j=1$ to $H$}{
Enter the $1$st iteration of the for loop in lines~2-6 in $M_j$; select $s_j^{N_{p+1}}\in\Q$ according to the probabilities $q_{j,1}^{N_{p+1}},\dots,q_{j,|Q|}^{N_{p+1}}$ computed by line~3 in $M_j$; construct $\CS^{\prime N_{p+1}}_H=\cup_{j\in[H]}s_j^{N_{p+1}}$.
}}
    %\tcc{Exploitation}
\For{$k=N_{p+1}$ to $N$}{
Sample $b_k\overset{\text{i.i.d.}}{\sim}\text{Bernoulli}(\rho)$.\\

    Sample $i\in[H]$ u.a.r. and $s\in\G$ u.a.r.\\
    If $b_k=1$, select $\CS^k=\CS^{\prime k}_{i-1}\cup\{s\}$ for episode $k$; if $b_k=0$, select $\CS^k=\CS^{\prime k}_H$ for episode $k$. \\
        \For{$t=0$ to $T-1$}{
            Obtain $\hat{K}_{t,\CS^k}^{k}$ using $\hat{A},\hat{B}_{\CS^k}$ via Eq.~\eqref{eqn:control gain S_t est}.\\
            Play $u_{t,\CS^k}^{k}=\hat{K}_{t,\CS^k}^{k}x_t^{k}$.
        }
        {\color{black}\For{$j=1$ to $H$}{
        Receive the cost $\hat{y}_{j,s_j^k}^k$; follow lines~4-6 in $M_j$ with $y_a=-\bar{y}_b$ and $y_b=\bar{y}_b$; finish the $(k-N_{p+1}+1)$th iteration of the for loop in lines~2-6 in $M_j$.\\
        Enter the $(k-N_{p+1}+2)$th iteration of the for loop in lines~2-6 in $M_j$; select $s_j^{k+1}$ according to the probabilities  $q_{j,1}^{k+1},\dots,q_{j,|Q|}^{k+1}$; construct $\CS^{\prime k+1}_H=\cup_{j\in[H]}s_j^{k+1}$.
        }}

}
{\color{black}\KwOut{$\CS^k,u_{\CS^k}^k=(u_{0,\CS^k}^k,\dots,u_{T-1,\CS^k}^k),\forall k\in[N]$.}}
\end{algorithm2e}

{\color{black} As we argued in Sections~\ref{sec:algorithm design for episodic setting} and~\ref{sec:regret analysis}, Problem~\eqref{eqn:LQR obj} (i.e., Problem~\eqref{eqn:LQR obj 2nd}) is NP-hard, and using a single {\bf Exp3.S} subroutine in Algorithm~\ref{alg:online for AS} leads to the exponential factor $|\Q|={|\G|\choose H}$ in $|\G|$ in both the running time and the regret bound of Algorithm~\ref{alg:online for AS}. To overcome the computational bottleneck, Algorithm~\ref{alg:online for AS large scale} leverages $H$ {\bf Exp3.S} subroutines each of which selects a single actuator in each episode as we described above.} One can check that the running time of each episode in Algorithm~\ref{alg:online for AS large scale} is $O(H((n+m)^3T+|\G|T))=O(H(n+m)^3T)$. 

{\color{black} To overcome the $|\Q|={|\G|\choose H}$ factor in the regret analysis, we will leverage the notion of $c$-regret introduced for online algorithms for combinatorial optimization problems (see, e.g., \cite{streeter2008online,salazar2021differentially}). The $c$-regret is parameterized by $c\in(0,1]$ whose value will be specified shortly. For any $k\in[N]$, denote
\begin{equation}
\label{eqn:g_k normalized}
g_k(\CS)\triangleq J_k(\emptyset)-J_k(\CS),
\end{equation}
for all $\CS\subseteq\G$, where $J_k(\CS)$ is given by Eq.~\eqref{eqn:opt LQR cost S_t}. %As shown in, e.g., \cite{chamon2021approximately}, the set function $g_k:2^{\G}\to\R_{\ge0}$ is monotone non-decreasing with $g_k(\emptyset)=0$.\footnote{A set function $g:2^{\G}\to\R_{\ge}$ is said to be monotone non-decreasing if $g(\A)\le g(\B)$ for all $\A\subseteq\B\subseteq\G$.} 
Now, we augment the elements in the ground set $\G$ (of all the candidate actuators) and define 
\begin{equation}\label{eqn:augmented G}
\bar{\G}=\{(s^{N_{p+1}},\dots,s^K):s^k\in\G,k\in\K\}
\end{equation}
with $\K$ given by Eq.~\eqref{eqn:T}. For any $k\in\K$, let $\bar{\CS}^k=\{\bar{s}^k\in\bar{s}:\bar{s}\in\bar{\CS}\}$ with $\bar{s}^k$ denoting the $k$th element of the tuple $\bar{s}\in\bar{\CS}$. Next, we define $\bar{g}(\bar{\CS})=\sum_{k\in\K}g_k(\bar{\CS}^k)$ for all $\bar{\CS}\subseteq\bar{G}$. One can check that Problem~\eqref{eqn:LQR obj 2nd} (over the episodes in $\K$) can be equivalently written as 
\begin{equation}\label{eqn:obj 3rd}
\max_{\bar{\CS}\subseteq\G,|\bar{\CS}|=H}\bar{g}(\bar{\CS}). 
\end{equation}
Since Problem~\eqref{eqn:obj 3rd} is NP-hard, offline approximation algorithms have been proposed to solve Problem~\eqref{eqn:obj 3rd} with known system matrices $A$ and $B$. For example, the (offline) greedy algorithm can be applied to Problem~\eqref{eqn:obj 3rd} and return a solution $\bar{\CS}_g$ such that $\bar{g}(\bar{\CS}_g)\ge(1-e^{-c_g})g(\bar{\CS}_{\star})$,\footnote{The greedy algorithm initializes $\bar{\CS}_g=\emptyset$ and iteratively adds $\bar{s}_{\star}\in\arg\max_{\bar{s}\in\bar{\G}}(\bar{g}(\bar{\CS}_g\cup\{\bar{s}\})-\bar{g}(\bar{\CS}_g))$ to $\bar{\CS}_g$ until $|\bar{\CS}_g|=H$.} where $\bar{\CS}_{\star}$ is an optimal solution to Problem~\eqref{eqn:obj 3rd} and $c_g\in(0,1]$ is the submodularity ratio of $\bar{g}(\cdot)$ defined to be the largest $c_g\in\R$ such that  
\begin{equation}\label{eqn:submodularity ratio}
\sum_{\bar{s}\in\bar{\B}\setminus\bar{\A}}\big(\bar{g}(\bar{\A}\cup\{\bar{s}\})-\bar{g}(\bar{\A})\big)\ge c_g\big(\bar{g}(\bar{\A}\cup\bar{\B})-\bar{g}(\bar{\A})\big),
\end{equation}
for all $\bar{\A},\bar{\B}\subseteq\bar{\G}$ (see, e.g., \cite{bian2017guarantees,chamon2021approximately}, for more details).\footnote{Note that computing the exact value of $c_g$ from \eqref{eqn:submodularity ratio} can be intractable. Nonetheless, all of our arguments in Section~\ref{sec:large scale instances} still hold if $c_g$ is replaced with a computable lower bound (e.g., \cite{chamon2021approximately}).} Based on the above arguments, one can view the actuators selected by any $M_j$ from episodes $k=N_{p+1}$ to $N$ as a single action, denoted as $\bar{s}_j=(s_j^{N_{p+1}},\dots,s_j^K)$, which corresponds to the element in the $j$th iteration of the greedy algorithm.}

{\color{black}Based on the above arguments,} we introduce the following $(1-e^{-c_g})$-regret to measure the performance of Algorithm~\ref{alg:online for AS large scale}:
\begin{equation}
\label{eqn:regret of A_e large}
R_{\A_l}=(1-e^{-c_g})\big(\sum_{k=1}^NJ_k(\emptyset)-J_k(\CS^{k}_{\star})\big)-\E_{\A_l}\Big[\big(\sum_{k=1}^NJ_k(\emptyset)-J_k(\CS^k,u_{\CS^k}^k)\big)\Big],
\end{equation}
where $\E[\cdot]$ denotes the expectation with respect to the randomness of the algorithm, and $\CS^{k}_{\star}$ is an optimal solution to \eqref{eqn:LQR obj 2nd} (with cost matrices $Q^k,R^k,Q_f^k$ and an extra constraint $\CS_0=\cdots=\CS_{T-1}$). Note that the benchmark $\sum_{k=1}^NJ_k(\CS^k_{\star})$ in Eq.~\eqref{eqn:regret of A_e} is equivalent to the normalized benchmark $\sum_{k=1}^N(J_k(\emptyset)-J_k(\CS^{k}_{\star}))$ in Eq.~\eqref{eqn:regret of A_e large}, since one can replace the objective function $J_k(\CS)$ in \eqref{eqn:LQR obj 2nd} with $J_k(\emptyset)-J_k(\CS)$, and consider the maximization over all $\CS$, which does not change the optimal solution to \eqref{eqn:LQR obj 2nd}. In words, the benchmark $\sum_{k=1}^N(J_k(\emptyset)-J_k(\CS^{k}_{\star}))$ in $R_{\A_l}$ is the improvement (i.e., decrease) of the cost of Problem~\eqref{eqn:LQR obj} when the sets of actuators $\CS_{\star}^1,\dots,\CS_{\star}^N$ are selected, over the cost when no actuator is selected for any episode $k\in[N]$. Such a normalization of the benchmark is necessary when analyzing the $c$-regret of online algorithms for combinatorial optimization problems (see, e.g., \cite{streeter2008online}, for more details). Accordingly, $R_{\A_l}$ compares the optimal improvement in the cost of Problem~\eqref{eqn:LQR obj} against the improvement corresponding to  Algorithm~\ref{alg:online for AS large scale}. For our analysis in this section, we make the following assumption.
\begin{assumption}
\label{ass:stable A}
(a) The matrix $A\in\R^{n\times n}$ in system~\eqref{eqn:LTI} is stable; (b) the pair $(A,B_s)$ is $(\ell,\nu)$ controllable for all $s\in\G$.
\end{assumption}
Under Assumption~\ref{ass:stable A}(a), Assumption~\ref{ass:system parameters} is naturally satisfied by choosing $\G_j=\emptyset$ and $K_{\G_j}=0$ for all $j\in[p]$. Recall that Assumption~\ref{ass:stable A}(b) is a sufficient condition for Assumption~\ref{ass:controllability} to hold as we argued in Section~\ref{sec:ce LQR}. Using similar arguments to those in Section~\ref{sec:proof of thm 1}, one show that under Assumption~\ref{ass:stable A} and $\CE$ defined in \eqref{eqn:event E}, $J_k(\emptyset)$, $J_k(\CS^k,u_{\CS^k}^k)$ and $y_{j,s}^k$ scale linearly with $T$ for all $k\in\{N_{p+1},\dots,N\}$. 
In the sequel, we use the same notations as those defined in \eqref{eqn:def of Gamma S_t}-\eqref{eqn:max sigma_1 Q and R} and \eqref{eqn:parameters} to denote the parameters of Problem~\eqref{eqn:LQR obj}, except that we replace $|\CS|=H$ in the definitions with $|\CS|\le H$. We then have the following result; the proof is included in Appendix~\ref{app:regret bound proofs non-episodic}. {\color{black} The proof extends the analysis in \cite{streeter2008online} for submodular objective functions (i.e., $c_g=1$) to approximately submodular functions (i.e., $c_g\in(0,1]$), and adopts the analyses and results developed in Sections~\ref{sec:algorithm design for episodic setting}-\ref{sec:regret analysis} for Problem~\eqref{eqn:LQR obj}.}
\begin{proposition}
\label{prop:regret of A_e large}
Consider any $\delta\in\R_{>0}$ with $0<\delta<1$, and the same setting as Theorem~\ref{thm:regret of as alg}. Additionally, suppose Assumption~\ref{ass:stable A} holds, and in Algorithm~\ref{alg:online for AS large scale} let 
\begin{equation}\label{eqn:rho}
\rho=\Big(\frac{\log(|\G|N)+e}{N}\Big)^{1/3}.
\end{equation}
Then, for any $N>\max\{\tau_1p,(\log(|\G|N)+e)\}$, the following holds with probability at least $1-\delta$:
\begin{equation}
R_{\A_l}=\tilde{O}(n(m+n)^2p^2T|\G|^{3/2}H^2h(\CS_{\star})^{1/2}N^{2/3}),
\end{equation}
where $h(\CS_{\star})$ is defined in Eq.~\eqref{eqn:number of switchings}, and $\tilde{O}(\cdot)$ hides polynomial factors in $\log(|\G|N),\log((m+n)TN/\delta)$ and other parameters of Problem~\eqref{eqn:LQR obj}.%\footnote{The same notations are used as those defined in \eqref{eqn:def of Gamma S_t}-\eqref{eqn:max sigma_1 Q and R} and \eqref{eqn:parameters}, but we replace $|\CS|=H$ in the definitions with $|\CS|\le H$. Moreover, similarly to footnote~5, the parameter $\Gamma$ given by \eqref{eqn:parameters} is independent of $T$, under Assumption~\ref{ass:stable A}.}
\end{proposition}

{\color{black} Next, we consider the non-episodic setting of Problem~\eqref{eqn:LQR obj}. Following the arguments in Section~\ref{sec:alg for non-episodic setting}, given a non-episodic instance of Problem~\eqref{eqn:LQR obj}, we can first construct an episodic instance with parameters $N^{\prime},T^{\prime}$, and then apply Algorithm~\ref{alg:online for AS large scale}. Here, the corresponding $H$ {\bf Exp3.S} subroutines in Algorithm~\ref{alg:online for AS large scale} are applied to the same MAB instances described above Eq.~\eqref{eqn:cost of arms in large-scale}, except that we scale the costs $y_{j,s}^k$ and $\hat{y}_{j,s_j^k}^k$ defined in Eqs.~\eqref{eqn:cost of arms in large-scale} and \eqref{eqn:y_hat}, respectively, by a multiplicative factor $1/T^{\prime}$. Similarly, following our arguments leading up to Eq.~\eqref{eqn:regret of A_e large} and using the notations in Section~\ref{sec:online AS non-episodic}, the $(1-e^{-c_g})$-regret of Algorithm~\ref{alg:online for AS large scale} in the non-episodic setting is given by 
\begin{equation}
\label{eqn:regret of A_e large nonepisodic}
R_{\A_l^{\prime}}=(1-e^{-c_g})\big(J(\emptyset)-J(\CS^{\star})\big)-\E_{\A^{\prime}_l}\Big[J(\emptyset)-\sum_{t=0}^{T-1}c_t(\CS_{0:t},u_{\CS_{0:t}})\Big].
\end{equation}
where $J(\CS^{\star})$ is defined as~\eqref{eqn:opt LQR cost S_t}, $\CS^{\star}=(\CS^{\star}_0,\dots,\CS^{\star}_{T-1})$ is an optimal solution to \eqref{eqn:LQR obj 2nd} (with an extra constraint $\CS^{\star}_{(k-1)T^{\prime}}=\cdots=\CS^{\star}_{kT^{\prime}-1}$ for all $k\in[N^{\prime}]$), $\emptyset$ is a short hand for the $T$-tuple $(\emptyset,\dots,\emptyset)$, and $c_t(\cdot,\cdot)$ is defined in Eq.~\eqref{eqn:cost per time step}. The result below is proved in Appendix~\ref{app:proofs for large-scale instances}.
\begin{proposition}\label{prop:regret of A_e large non-episodic}
Suppose Assumptions~\ref{ass:cost matrices}-\ref{ass:stable A} hold. Set $T^{\prime}=\lceil T^{1/4}\rceil$ and set $N^{\prime},\tau^{\prime}_1,\bar{y}_b^{\prime}$ in the same way as Theorem~\ref{thm:regret of as alg non-episodic}. Additionally, in Algorithm~\ref{alg:online for AS large scale} let $\rho =\frac{ \big(\log(|\G|N^{\prime})+e\big)^{1/3}}{T^{1/4}}$. Consider any $\delta\in\R_{>0}$ with $0<\delta<1$. Then, for any $T>\tau_1^{\prime}pT^{\prime}$ with $T^{\prime}>T_m=\frac{2}{1-\eta}(\frac{1}{4}\log T+\log\zeta)>0$, the following holds with probability at least $1-\delta$:
\begin{align}
&R_{\A_l^{\prime}}= \tilde{O}(n(m+n)^2p^2|\G|^{3/2}H^2h(\CS_{\star})^{1/2}T^{3/4}),
\end{align}
where $h(\CS_{\star})$ is defined in Eq.~\eqref{eqn:number of switchings}, and $\tilde{O}(\cdot)$ hides polynomial factors in $\log(|\G|N^{\prime}),\log((m+n)T/\delta)$ and other parameters of Problem~\eqref{eqn:LQR obj}.
\end{proposition}
}
Recalling our arguments in Remark~\ref{remark:dynamic regret}, one can check that the regret bound on $R_{\A_l}$ (resp., $R_{\A_l^{\prime}}$) in Propositions~\ref{prop:regret of A_e large} (resp., Proposition~\ref{prop:regret of A_e large non-episodic}) also holds for general benchmark $\CS_{\star}=(\CS_{\star}^1,\dots,\CS_{\star}^N)$ (resp., $\CS_{\star}=(\CS_{\star}^1,\dots,\CS_{\star}^{N^{\prime}})$), where $\CS_{\star}^k$ is any $\CS_{\star}^k\subseteq\G$ with $|\CS_{\star}^k|=H$.

%\begin{remark}
%\label{remark:c-regret} 
%As discussed in \cite{streeter2008online}, the notion of $c$-regret is inspired by the fact that the offline combinatorial optimization problems (e.g., Problem~\eqref{eqn:LQR obj}) are NP-hard in general, and one may use approximation algorithms to obtain an approximate solution \cite{ye2020complexity}. Thus, the parameter $c\in(0,1)$ in Eq.~\eqref{eqn:regret of A_e large} is set to be the approximation ratio of the approximation algorithm applied to solve Problem~\eqref{eqn:LQR obj} offline. Specifically, supposing the system matrices $A$ and $B$ are known, it was shown in \cite{chamon2021approximately} that an offline greedy algorithm returns a solution $\CS^k_g$ with $|\CS^k_g|=H$ to Problem~\eqref{eqn:LQR obj} that satisfies the following approximation ratio: $g_k(\CS^k_g)\ge(1-e^{-c^{\prime}})g_k(\CS_{\star}^k)$, where $g_k(\cdot)$ is defined in Eq.~\eqref{eqn:g_k normalized} with $g_k(\emptyset)=0$, and $c\in(0,1)$ depends on the parameters of Problem~\eqref{eqn:LQR obj}. %Hence, we set $c=1-e^{-c^{\prime}}$ for $\bar{R}_{\A_e}$ defined in Eq.~\eqref{eqn:regret of A_e large}. To proceed, we make the following assumption for the regret analysis of Algorithm~\ref{alg:online for AS large scale}.
%\end{remark}

\newpage
%\vspace{-0.3cm}
\section{Simulation Results}\label{sec:simulation results}
\subsection{Medium-size Episodic Instances}\label{sec:simulations medium size}
We validate the results in Theorem~\ref{thm:regret of as alg} for Algorithm~\ref{alg:online for AS}, using the episodic instances of Problem~\eqref{eqn:LQR obj} constructed as follows. We randomly generate the matrices $A\in\R^{5\times 5}$ and $B\in\R^{5\times 10}$ such that Assumption~\ref{ass:controllability} is satisfied and $A$ is unstable. Let each column in $B\in\R^{5\times 10}$ correspond to one candidate actuator, and let the cardinality constraint on the set of selected actuators be $H=2$. The cost matrices are set to be $Q^k=R^k=I$ and $Q_f^k=2I$ for all $k\in[N]$. The covariance of the disturbance $w_t^{k}$ is set to be $W=I$ for all $t\in\{0,\dots,T-1\}$ and all $k\in[N]$. The number of time steps in any episode $k\in[N]$ is set to be $T=5$. {\color{black} Given $A$ and $B$ generated above, we construct the known stabilizing $K_{\G_i}$ with $|\G_i|=2$ for all $i\in[5]$ in Assumption~\ref{ass:system parameters}. We then apply Algorithm~\ref{alg:online for AS} to the instances of Problem~\eqref{eqn:LQR obj} constructed above, where the parameters $\tau_1,\bar{y}_b$ are set according to Theorem~\ref{thm:regret of as alg} and $\lambda=1$ in the least squares \eqref{eqn:ls approach} for the system identification phase in Algorithm~\ref{alg:online for AS}. We obtain the regret $R_{\A_e}$ of Algorithm~\ref{alg:online for AS} against an optimal static benchmark $\CS_{\star}=(\CS_{\star}^1,\dots,\CS_{\star}^N)$, i.e., $\CS_{\star}^1=\cdots=\CS_{\star}^N=\CS_{\tt opt}$ and $\CS_{\tt opt}\in\arg\min_{\CS\subseteq\G,|\CS|=H}\sum_{k=1}^NJ_k(\CS)$, with $h(\CS_{\star})=1$.} In Fig.~\ref{fig:regret per round}, we plot $R_{\A_e}/N$ and $R_{\A_e}/\sqrt{N}$ for different values of the total number of episodes $N$.\footnote{All the numerical results in Section~\ref{sec:simulation results} are averaged over $20$ experiments and shaded regions display quartiles.} From Fig.~\ref{fig:regret per round}(a), we see that $R_{\A_e}/N$ decreases as $N$ increases. From Fig.~\ref{fig:regret per round}(b), we see that $R_{\A_e}/\sqrt{N}$ (slightly) increases as $N$ increases. Hence, the results in Fig.~\ref{fig:regret per round}(a) and (b) match with the regret bound given by Eq.~\eqref{eqn:R_Ae}. {\color{black}Specifically, the regret bound in Eq.~\eqref{eqn:R_Ae} scales as $\sqrt{N}\log N$, which implies that $R_{\A_e}/N=O(\log N/\sqrt{N})$ and $R_{\A_e}/\sqrt{N}=O(\log N)$.} Moreover, $R_{\A_e}/N$ is around $20$ when $N=3000$, since the regret bound in Eq.~\eqref{eqn:R_Ae} also contains other parameters of Problem~\eqref{eqn:LQR obj}.
%\vspace{-0.3cm}
\begin{figure}[htbp]
    \centering
    \subfloat[a][$R_{\A_e}/N$ vs. $N$]{
    \includegraphics[width=0.495\linewidth]{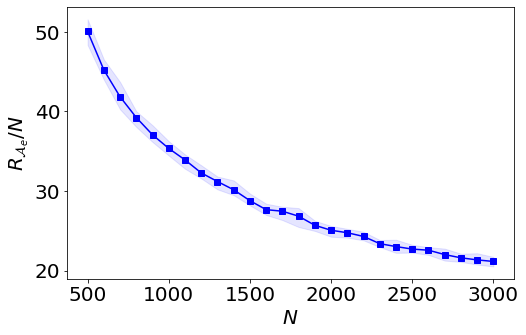}}
    \subfloat[b][$R_{\A_e}/\sqrt{N}$ vs. $N$]{
    \includegraphics[width=0.475\linewidth]{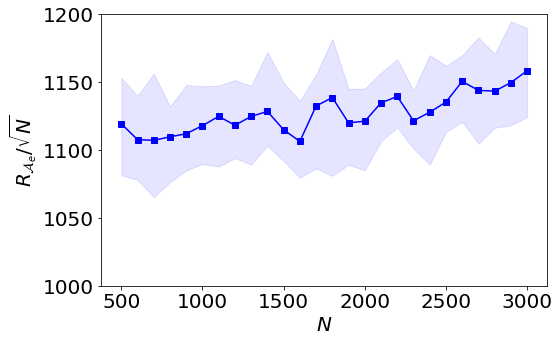}}
    \caption{$R_{\A_e}/N$ and $R_{\A_e}/\sqrt{N}$ against $N$. }
    \label{fig:regret per round}
\end{figure}
%\vspace{-0.1cm}
\begin{figure}[htbp]
    \centering
    \subfloat[a][$R_{\A_e}/N$ vs. $H$]{
    \includegraphics[width=0.495\linewidth]{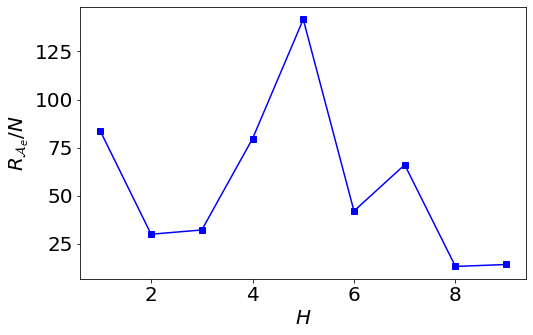}}
    \subfloat[b][Running times vs. $n$]{
    \includegraphics[width=0.475\linewidth]{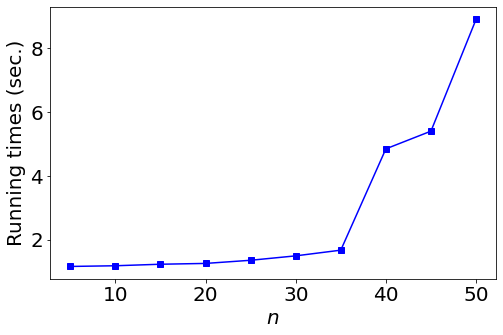}}
    \caption{The influence of the problem parameter $H$ (resp., $n$) on the performance (resp., running times) of Algorithm~\ref{alg:online for AS}.}
    \label{fig:influence of problem parameters}
\end{figure}

Now, we investigate how the other parameters of Problem~\eqref{eqn:LQR obj} influence the performance and running times of Algorithm~\ref{alg:online for AS}, using the instances of Problem~\eqref{eqn:LQR obj} constructed above (with different values of $H$ and $n$). In Fig.~\ref{fig:influence of problem parameters}, we plot $R_{\A_e}/N$ for different values of the cardinality constraint $H$, which shows that as $H$ increases, $R_{\A_e}/N$ first increases and then decreases. The result in Fig.~\ref{fig:influence of problem parameters}(a) matches with the regret bound in Eq.~\eqref{eqn:R_Ae}, since the regret bound contains the factor $\sqrt{|Q|}$ with $|\Q|={10\choose H}$ in the instances of Problem~\eqref{eqn:LQR obj} that we constructed. Note that $R_{\A_e}/N$ in Fig.~\ref{fig:influence of problem parameters}(a) decreases as $H$ increases from $1$ to $2$, which is potentially due to the fact that the regret bound in Eq.~\eqref{eqn:R_Ae} also contains the factor $p=\lceil 10/H\rceil$. In Fig.~\ref{fig:influence of problem parameters}, we plot the running times of Algorithm~\ref{alg:online for AS} for different values of $n$ (i.e., the dimension of the system matrix $A$). Similarly, we generate $A\in\R^{n\times n}$ for all $n=5,10,\dots,50$ and $B\in\R^{n\times 15}$ randomly. Fig.~\ref{fig:influence of problem parameters}(b) shows that the running time of Algorithm~\ref{alg:online for AS} increases as $n$ increases, which aligns with the time complexity $O((n+m)^3T+|\Q|T)$ of each episode in Algorithm~\ref{alg:online for AS}.

\begin{figure}[htbp]
    \centering
    \subfloat[a][{\color{black}$R_{\A_l^{\prime}}/T$ vs. $T$}]{
    \includegraphics[width=0.495\linewidth]{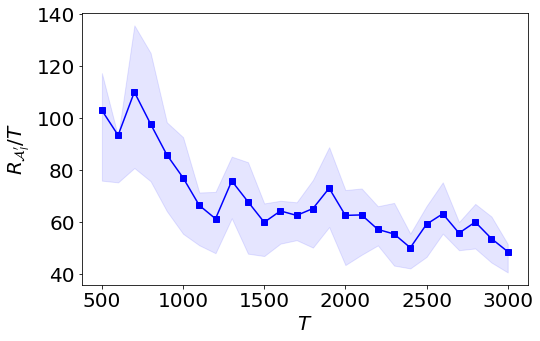}}
    \subfloat[b][{\color{black}$R_{\A_l^{\prime}}/T^{3/4}$ vs. $T$}]{
    \includegraphics[width=0.495\linewidth]{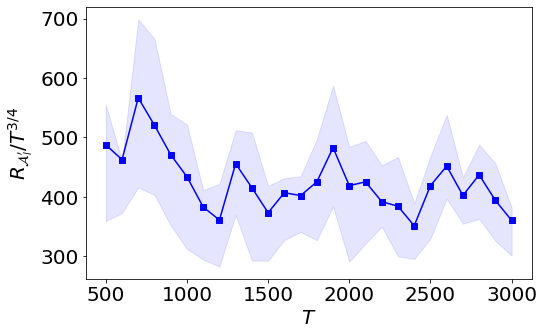}}
    \caption{{\color{black}$R_{\A_l^{\prime}}/T$ and $R_{\A_l^{\prime}}/T^{3/4}$ against $T$.}}
    \label{fig:large scale non-episodic regret}
\end{figure}
\subsection{Large-scale Non-episodic Instances}\label{sec:simulations large scale}
{\color{black}We next validate the results in Proposition~\ref{prop:regret of A_e large non-episodic} for Algorithm~\ref{alg:online for AS large scale} in the non-episodic setting. First, we randomly generate the matrices $A,B\in\R^{50\times 50}$ such that $A$ is stable. Let each column in $B\in\R^{50\times 50}$ correspond to one candidate actuator, and let the cardinality constraint on the set of selected actuators be $H=20$. The cost matrices are set to be $R=10^{-3}I,Q=Q_f=2\cdot10^{-3}I$. The covariance of the disturbance $w_t$ is set to be $W=I$ for all $t\in\{0,\dots,T-1\}$. Since $A$ is stable, we choose the stabilizing $K_{\G_i}=0$ for all $i\in[p]$ in Assumption~\ref{ass:system parameters}. As argued in Sections~\ref{sec:alg for non-episodic setting}-\ref{sec:large scale instances}, we can first construct an episodic instance of Problem~\eqref{eqn:LQR obj} given the non-episodic instance generated above, and then apply Algorithm~\ref{alg:online for AS large scale}, where the parameters $T^{\prime},N^{\prime},\tau_1^{\prime},\bar{y}_b^{\prime},\rho$ are set according to Proposition~\ref{prop:regret of A_e large non-episodic} and $\lambda=1$ in the system identification phase in Algorithm~\ref{alg:online for AS large scale}. Since ${50\choose 20}\approx5\times10^{13}$ and Problem~\eqref{eqn:LQR obj} is NP-hard, both Algorithm~\ref{alg:online for AS non-episodic} and obtaining an optimal solution $\CS^{\star}=(\CS^{\star}_1,\dots,\CS^{\star}_{T-1})$ to Problem~\eqref{eqn:LQR obj 2nd} become intractable. Thus, we obtain the regret $R_{\A_l^{\prime}}$ of Algorithm~\ref{alg:online for AS large scale} (in the non-episodic setting) against a random static benchmark $\CS^{\star}=(\CS_1^{\star},\dots,\CS_{T-1}^{\star})$, where $\CS^{\star}_0=\cdots=\CS^{\star}_{T-1}=\CS_{\tt rand}$ and $\CS_{\tt rand}$ is chosen from $\G$ randomly with $|\CS_{\tt rand}|=H$. Moreover, we replace $1-e^{-c_g}$ with $1$ in Eq.~\eqref{eqn:regret of A_e large nonepisodic} so that $R_{\A_l^{\prime}}$ is lifted to the $1$-regret of Algorithm~\ref{alg:online for AS large scale}. In Fig.~\ref{fig:large scale non-episodic regret}, we plot $R_{\A_l^{\prime}}/T$ and $R_{\A_l^{\prime}}/T^{3/4}$ for different values of the total time steps $T$. Fig.~\ref{fig:large scale non-episodic regret}(a) shows that $R_{\A_l^{\prime}}/T$ decreases as $T$ increases, which aligns with the $T^{3/4}$-regret bound in Eq.~\eqref{eqn:regret of A_e large nonepisodic}. Fig.~\ref{fig:large scale non-episodic regret}(b) shows that $R_{\A_l^{\prime}}/T^{3/4}$ also tends to decrease as $T$ increases, which potentially implies that the regret bound is not tight in terms of $T$. Note that Fig.~\ref{fig:large scale non-episodic regret} also shows that Algorithm~\ref{alg:online for AS large scale} yields good regret performance in terms of the stronger notion of $1$-regret.
%\vspace{-0.5cm}
%\vspace{-1cm}
\begin{figure}[htbp]
    \centering
    \subfloat[a][{\color{black}Running times vs. $m$}]{
    \includegraphics[width=0.495\linewidth]{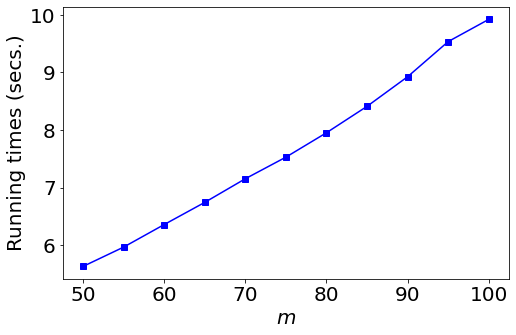}}
    \subfloat[b][{\color{black} Running times vs. $H$}]{
    \includegraphics[width=0.475\linewidth]{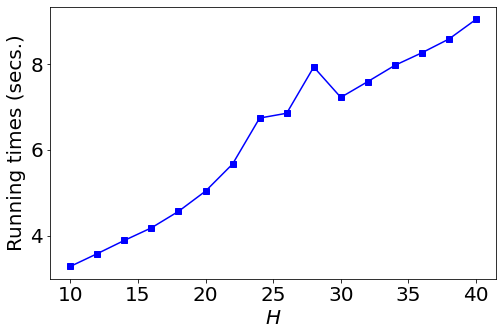}}
    \caption{{\color{black}The running times of Algorithm~\ref{alg:online for AS large scale} against $m$ and $H$.}}
    \label{fig:runtime large scale}
\end{figure}

As for the running times of Algorithm~\ref{alg:online for AS large scale}, we plot the running times of Algorithm~\ref{alg:online for AS large scale} when applied to the non-episodic instances constructed above with different values of $m$ and $H$. Fig.~\ref{fig:runtime large scale} aligns with the time complexity $O(H(n+m)^3T)$ of Algorithm~\ref{alg:online for AS large scale} and shows that Algorithm~\ref{alg:online for AS large scale} is suitable for large-scale (non-episodic) instances of Problem~\eqref{eqn:LQR obj}.}

\section{Conclusion}\label{sec:Conclusion}
We studied the simultaneous actuator selection and controller design problem for LQR, when the system matrices are unknown. We proposed online actuator selection algorithms to solve the problem by interacting with the system in an online manner, where the algorithms specify both the sets of actuators to be utilized under a cardinality constraint and the controls corresponding to the sets of selected actuators. We considered episodic and non-episodic settings of the problem, and showed that our online algorithms yield sublinear regrets with respect to the horizon length considered in the problem. We extended our algorithm design and analysis to efficiently handle instances of the problem when both the total number of candidate actuators and the cardinality constraint scale large. We numerically validated our theoretical results.

\bibliography{main}

\appendix
%\vspace{-0.6cm}
\section{Proofs pertaining to the certainty equivalence approach}\label{app:ce approach}
\subsection*{Proof of Lemma~\ref{lemma:mismatch of K_hat and P_hat}}
In this proof, we drop the dependency of various terms on $\CS$ and $k$ for notational simplicity, while the proof holds for any $\CS\subseteq\G$ and any $k\in[N]$. For example, we write $P_{t,\CS}^{k}$ as $P_t$, and write $B_{\CS}$ as $B$. To prove \eqref{eqn:est error of K_hat}, we first note that 
\begin{equation*}
\norm{B^{\top}P_t B - \hat{B}^{\top}\hat{P}_t\hat{B}}
\le\norm{B^{\top}P_t B - B^{\top}P_t\hat{B}} + \norm{B^{\top}P_t\hat{B} - B^{\top}\hat{P}_t\hat{B}} 
+ \norm{B^{\top}\hat{P}_t\hat{B} - \hat{B}^{\top}\hat{P}_t\hat{B}},
\end{equation*}
which implies that 
\begin{align}\nonumber
&\norm{B^{\top}P_t B - \hat{B}^{\top}\hat{P}_t\hat{B}}\\\nonumber
\le&\norm{\hat{B}}\norm{\hat{P}_t}\varepsilon + \norm{\hat{B}}\norm{B}D\varepsilon + \norm{B}\norm{P_t}\varepsilon\\\nonumber
\le&(\Gamma+\varepsilon)(\Gamma+D\varepsilon)\varepsilon + (\Gamma+\varepsilon)\Gamma D\varepsilon + \Gamma^2\varepsilon\\
\le&\tilde{\Gamma}^2D\varepsilon + \tilde{\Gamma}^2D\varepsilon + \Gamma^2\varepsilon\le3\tilde{\Gamma}^2D\varepsilon,\label{eqn:diff of one term in K_k}
\end{align}
where $\Gamma$ (i.e., $\Gamma_{\CS}$) is defined in Eq.~\eqref{eqn:def of Gamma S_t}. Note that $\sigma_n(R^{k})\ge1$ from Assumption~\ref{ass:cost matrices}. Also recalling the definitions of $K_{t-1,\CS}^{k}$ and $\hat{K}_{t-1,\CS}^{k}$ in Eqs.~\eqref{eqn:optimal control law} and \eqref{eqn:control gain S_t est}, respectively, the rest of the proof for \eqref{eqn:est error of K_hat} now follows from similar arguments to those for \cite[Lemma~2]{mania2019certainty}.

To prove Eq.~\eqref{eqn:est error of P_hat}, one can first use Eq.~\eqref{eqn:control gain S_t est} to rewrite Eq.~\eqref{eqn:recursion for P_k(S_t) est} as
\begin{equation*}
\label{eqn:another recursion for Pk_hat}
\hat{P}_{t-1} = Q + \hat{K}_{t-1}^{\top}R\hat{K}_{t-1} +(\hat{A}+\hat{B}\hat{K}_{t-1})^{\top}\hat{P}_t(\hat{A}+\hat{B}\hat{K}_{t-1}).
\end{equation*}
Similarly, one can obtain from Eqs.~\eqref{eqn:control gain}-\eqref{eqn:recursion for P_k} the following:
\begin{equation*}
\label{eqn:another recursion for Pk}
P_{t-1} = Q + K_{t-1}^{\top}R K_{t-1} +(A+BK_{t-1})^{\top}P_t(A+BK_{t-1}).
\end{equation*}
Now, using similar arguments to those above for \eqref{eqn:diff of one term in K_k}, one can show via \eqref{eqn:est error of K_hat} that 
\begin{align}\nonumber
&\norm{A + BK_{t-1}-\hat{A}-\hat{B}\hat{K}_{t-1}}\\\nonumber
\le&\norm{A-\hat{A}}+\norm{BK_{t-1}-B\hat{K}_{t-1}}+\norm{B\hat{K}_{t-1}-\hat{B}\hat{K}_{t-1}}\\
\le&\varepsilon + \Gamma\varepsilon + (\Gamma+\varepsilon)3\tilde{\Gamma}^3 D\varepsilon\le4\tilde{\Gamma}^4D\varepsilon.\label{eqn:Lk minus Lk_hat}
\end{align}
Denoting $\hat{L}_{t-1}\triangleq\hat{A}+\hat{B}\hat{K}_{t-1}$ and $L_{t-1}=A+BK_{t-1}$, we have that 
\begin{equation*}
\label{eqn:Pk minus Pk_hat}
\norm{P_{t-1}-\hat{P}_{t-1}}\le\norm{K_{t-1}^{\top}RK_{t-1} - \hat{K}_{t-1}^{\top}R\hat{K}_{t-1}} \\+ \norm{L_{t-1}^{\top}P_kL_{t-1} - \hat{L}_{t-1}^{\top}\hat{P}_t\hat{L}_{t-1}}.
\end{equation*}
Similarly, one can show that 
\begin{align}\nonumber
&\norm{K_{t-1}^{\top}RK_{t-1} - \hat{K}_{t-1}^{\top}R\hat{K}_{t-1}}\\\nonumber
\le&(\Gamma+3\tilde{\Gamma}^3 D\varepsilon)\sigma_1(R)3\tilde{\Gamma}^3 D\varepsilon+3\tilde{\Gamma}^3 D\varepsilon\sigma_1(R)\Gamma\varepsilon\\\nonumber
\le&3\tilde{\Gamma}^3D\sigma_1(R)\varepsilon(2\Gamma+3\tilde{\Gamma}^3D\varepsilon)\\
\le&6\tilde{\Gamma}^4\sigma_1(R)D\varepsilon + 9\tilde{\Gamma}^6\sigma_1(R)D^2\varepsilon^2\le 15\tilde{\Gamma}^6\sigma_1(R)D^2\varepsilon.\label{eqn:first term in Pk minus Pk_hat}
\end{align}
Let us also denote $\Delta L_{t-1}\triangleq L_{t-1}-\hat{L}_{t-1}$. Noting that $\norm{A+BK_{t-1}}\le\tilde{\Gamma}^2$ and recalling~\eqref{eqn:Lk minus Lk_hat}, one can show that
\begin{align}\nonumber
&\norm{L_{t-1}^{\top}P_tL_{t-1} - \hat{L}_{t-1}^{\top}\hat{P}_t\hat{L}_{t-1}}\\\nonumber
\le&(\norm{\Delta L_{t-1}}+\tilde{\Gamma}^2)(\Gamma+D\varepsilon)\norm{\Delta L_{t-1}}+ (\norm{\Delta L_{t-1}}+\tilde{\Gamma}^2)\tilde{\Gamma}^2D\varepsilon+\norm{\Delta L_{t-1}}\Gamma\tilde{\Gamma}^2\\
\le&16\tilde{\Gamma}^9D^3\varepsilon^2 + 4\tilde{\Gamma}^7D^2\varepsilon + 4\tilde{\Gamma}^6D^2\varepsilon^2 + \tilde{\Gamma}^4D\varepsilon + 4\tilde{\Gamma}^6\Gamma D\varepsilon\le29\tilde{\Gamma}^9D^3\varepsilon.\label{eqn:second term in Pk minus Pk_hat}         
\end{align} 
The inequality in \eqref{eqn:est error of P_hat} now follows from combining \eqref{eqn:first term in Pk minus Pk_hat} and \eqref{eqn:second term in Pk minus Pk_hat}, and noting the fact that $\sigma_1(R^{k}_{\CS})\le\sigma_R$.\hfill\qed

\subsection*{Proof of Lemma~\ref{lemma:bound on est error of P_hat 1}}
Our proof is based on a similar idea to that for the proof of \cite[Proposition~3]{mania2019certainty}. To simplify the notations in the proof, we assume that $T=\varphi\ell$ for some $\varphi\in\Z_{\ge1}$; otherwise we only need to focus on the time steps from $T-\tilde{\varphi}\ell$ to $T$ of Problem~\eqref{eqn:LQR obj}, where $\tilde{\varphi}$ is the maximum positive integer such that $T-\tilde{\varphi}\ell\ge0$. Under the assumption that $T=\varphi\ell$, we need to show that \eqref{eqn:bound on est error of P_hat} holds for $t\in\{0,\ell,\dots,\varphi\ell\}$. Note that \eqref{eqn:bound on est error of P_hat} holds for $t=T$, since $P_{T,\CS}^{k}=\hat{P}_{T,\CS}^{k}=Q_f^{k}$. In the rest of this proof, we drop the dependency of various terms on $\CS$ and $k$ for notational simplicity, while the proof works for any $\CS\subseteq\G$ (with $|\CS|=H$) and any $k\in[N]$. First, for any $\gamma\in\Z_{\ge1}$ (with $\gamma\ell\le T$), let us consider the noiseless LQR problem for system~\eqref{eqn:LTI}, i.e., $x_{t+1}=Ax_t+Bu_t$, from time step $\gamma\ell$ to $T$. Let the initial state $x_{\gamma\ell}$ be any vector in $\R^n$ with $\norm{x_{\gamma\ell}}\le1$. Similarly to Eq.~\eqref{eqn:episode t cost}, we define the cost 
\begin{equation*}
\tilde{J}(A,B,u_{\gamma\ell:T-1})\triangleq\Big(\sum_{j=\gamma}^{\varphi-1}\sum_{t=0}^{\ell-1}x_{j\ell+t}^{\top}Q x_{j\ell+t}+u_{j\ell+t}^{\top}Ru_{j\ell+t}\Big)+x_{T}^{\top}Q_f x_{T},
\end{equation*}
where $u_{\gamma\ell:T-1}=(u_{\gamma\ell},\dots,u_{T-1})$. Again, we know from \cite{bertsekas2017dynamic} that the minimum value of $\tilde{J}(A,B,u_{\gamma\ell:T-1})$ (over all control policies $u_{\gamma\ell:T-1}$) is achieved by $\tilde{u}_t=K_{t,\CS}x_t$ for all $t\in\{\gamma\ell,\gamma\ell+1,\dots,T-1\}$, where $K_{t,\CS}$ is given by Eq.~\eqref{eqn:control gain}. Moreover, we know that $\tilde{J}(A,B,\tilde{u}_{\gamma\ell:T-1})=x_{\gamma\ell}^{\top}P_{\gamma\ell}x_{\gamma\ell}$, where $P_{\gamma\ell}$ can be obtained from Eq.~\eqref{eqn:recursion for P_k} with $P_T=Q_f$.

Next, consider another LTI system given by $\hat{x}_{t+1}=\hat{A}\hat{x}_{t}+\hat{B}\hat{u}_{t}$ over the same time horizon and starting from the same initial state $\hat{x}_{\gamma\ell}=x_{\gamma\ell}$ as we described above. Define the corresponding cost as
\begin{equation*}
J(\hat{A},\hat{B},\hat{u}_{\gamma\ell:T-1})\triangleq\Big(\sum_{j=\gamma}^{\varphi-1}\sum_{t=0}^{\ell-1}\hat{x}_{j\ell+t}^{\top}Q \hat{x}_{j\ell+t}+\hat{u}_{j\ell+t}^{\top}R\hat{u}_{j\ell+t}\Big)+\hat{x}_{T}^{\top}Q_f \hat{x}_{T},
\end{equation*}
where $\hat{u}_{\gamma\ell:T-1}=(\hat{u}_{\gamma\ell},\dots,\hat{u}_{T-1})$. Similarly, the minimum value of $J(\hat{A},\hat{B},\hat{u}_{\gamma\ell:T-1})$ (over all control policies $\hat{u}_{\gamma\ell:T-1}$) is achieved by $u^{\prime}_{t}=\hat{K}_{t,\CS}^{k}\hat{x}_{t}$ for all $t\in\{\gamma\ell,\gamma\ell+1,\dots,T-1\}$, where $\hat{K}_{t,\CS}$ is given in Eq.~\eqref{eqn:control gain S_t est}. The minimum cost is given by $J(\hat{A},\hat{B},u^{\prime}_{\gamma\ell:T-1})=x_{\gamma\ell}^{\top}\hat{P}_{\gamma\ell}x_{\ell\gamma}$, where $\hat{P}_{\gamma\ell}$ can be obtained from Eq.~\eqref{eqn:recursion for P_k(S_t) est} with $\hat{P}_T=Q_f$. Moreover, note that  
\begin{align*}
%\label{eqn:upper bound on J hat}
J(\hat{A},\hat{B},u^{\prime}_{\gamma\ell:T-1})\le J(\hat{A},\hat{B},\hat{u}_{\gamma\ell:T-1}),
\end{align*}
where $\hat{u}_{\gamma\ell:T-1}$ is an arbitrary control policy and the inequality follows from the optimality of $u^{\prime}_{\gamma\ell:T-1}$. Recalling that $\varepsilon$ is assumed to be small enough such that the right-hand side of \eqref{eqn:bound on est error of P_hat} is smaller than or equal to $1$, one can obtain from Lemma~\ref{lemma:lower bound on sigma_1 C_hat} that $\sigma_n(\hat{\C}_{\ell,\CS})\ge\frac{\nu}{2}>0$, which implies that the pair $(\hat{A},\hat{B})$ is controllable. Now, one can follow similar arguments to those for the proof of \cite[Proposition~3]{mania2019certainty} and show that $\hat{u}_{\gamma\ell:\varphi\ell-1}$ can be chosen such that $\hat{x}_{\varphi^{\prime}\ell}=x_{\varphi^{\prime}\ell}$ for all $\varphi^{\prime}\in\{\gamma,\gamma+1,\dots,\varphi\}$. It then follows from the above arguments that 
\begin{align}
&x_{\gamma\ell}^{\top}\hat{P}_{\gamma\ell}x_{\gamma\ell} - x_{\gamma\ell}^{\top}P_{\gamma\ell}x_{\gamma\ell}\le\Big(\sum_{j=\gamma}^{\varphi-1}\sum_{t=0}^{\ell-1}\hat{x}_{j\ell+t}^{\top}Q\hat{x}_{j\ell+t}+\hat{u}_{j\ell+t}^{\top}R\hat{u}_{j\ell+t}-x_{j\ell+t}^{\top}Qx_{j\ell+t}-u_{j\ell+t}^{\top}Ru_{j\ell+t}\Big).\label{eqn:J diff}
\end{align}
One can further follow similar arguments to those for the proof of \cite[Propostion~3]{mania2019certainty} and show that $\hat{u}_{\gamma\ell:T-1}$ in Eq.~\eqref{eqn:J diff} can be chosen such that the following holds:
\begin{equation}
\label{eqn:J diff 1}
x_{\gamma\ell}^{\top}\hat{P}_{\gamma\ell}x_{\gamma\ell} - x_{\gamma\ell}^{\top}P_{\gamma\ell}x_{\gamma\ell}\le\frac{1}{2}\mu_{\gamma\ell}\varepsilon,
\end{equation}
under the assumption that $\frac{1}{2}\mu_{\gamma\ell}\varepsilon\le1$, where $\mu_{\gamma\ell}$ (i.e., $\mu_{\gamma\ell,\CS}^{k}$) is defined in Eq.~\eqref{eqn:def of mu}. Now, reversing the roles of $(A,B)$ and $(\hat{A},\hat{B})$ in the arguments above, one can also obtain that
\begin{equation}
\label{eqn:J diff 2}
x_{\gamma\ell}^{\top}P_{\gamma\ell}x_{\gamma\ell} - x_{\gamma\ell}^{\top}\hat{P}_{\gamma\ell}x_{\gamma\ell}\le\frac{1}{2}\mu_{\gamma\ell}\frac{\norm{\hat{P}_{\gamma\ell}}}{\norm{P_{\gamma\ell}}}\varepsilon,
\end{equation}
under the assumption that $\frac{1}{2}\mu_{\gamma\ell}\frac{\norm{\hat{P}_{\gamma\ell}}}{\norm{P_{\gamma\ell}}}\varepsilon\le1$.\footnote{Note that the proof technique in \cite{mania2019certainty} is for the infinite-horizon (noiseless) LQR problem, which can be adapted to the finite-horizon setting studied here. The details of such an adaption are omitted for conciseness.} Note from Eq.~\eqref{eqn:recursion for P_k} and Assumption~\ref{ass:cost matrices} that $P_{\gamma\ell}\succeq Q\succeq I$, and note that \eqref{eqn:J diff 1} and \eqref{eqn:J diff 2} hold for any $x_{\gamma\ell}\in\R^n$ with $\norm{x_{\gamma\ell}}\le1$ as we discussed above. It then follows from \eqref{eqn:J diff 1} that $\lambda_1(\hat{P}_{\gamma\ell})\le\lambda_1(P_{\gamma\ell})+1$, i.e., $\norm{\hat{P}_{\gamma\ell}}\le\norm{P_{\gamma\ell}}+1\le2\norm{P_{\gamma\ell}}$. Hence, we have from \eqref{eqn:J diff 1} and \eqref{eqn:J diff 2} that $\lambda_1(\hat{P}_{\gamma\ell}-P_{\gamma\ell})\le\mu_{\gamma\ell}\varepsilon$ and $\lambda_1(P_{\gamma\ell}-\hat{P}_{\gamma\ell})\le\mu_{\gamma\ell}\varepsilon$, which further implies \eqref{eqn:bound on est error of P_hat}.\hfill\qed

\subsection*{Proof of Proposition~\ref{prop:bound on est error of P_hat}}
Similarly to the proof of Lemma~\ref{lemma:bound on est error of P_hat 1}, we assume for simplicity that $T=\varphi\ell$ for some $\varphi\in\Z_{\ge1}$; the proof will follow similarly if this assumption on $T$ does not hold. Note that $\sigma_R\ge1$ (from Assumption~\ref{ass:cost matrices}), and that $\varepsilon$ is assumed to satisfy that $\mu_{\CS}\varepsilon\le1$. Recalling the definition of $\mu_{t,\CS}^{k}$ in Eq.~\eqref{eqn:def of mu}, one can then show that $\mu_{t,\CS}^{k}\varepsilon\le1$ for all $t\in\{0,1,\dots,T\}$, which implies via Lemma~\ref{lemma:bound on est error of P_hat 1} that $\norm{P_{t,\CS}^{k}-\hat{P}_{t,\CS}^{k}}\le\mu_{t,\CS}^{k}\varepsilon$ and thus \eqref{eqn:bound on est error of P_hat general} holds, for all $t\in\{0,\ell,\dots,\varphi\ell\}$. Now, consider any $\gamma\in[\varphi]$. Since $\varepsilon\le1$ and $\mu_{\gamma\ell,\CS}^{k}\ge1$, we have from Lemma~\ref{lemma:mismatch of K_hat and P_hat} that $\norm{P_{\gamma\ell-1,\CS}^{k}-\hat{P}_{\gamma\ell-1,\CS}^{k}}\le(44\tilde{\Gamma}_{\CS}^9\sigma_R)\mu_{\gamma\ell,\CS}^{k}\varepsilon$. Repeatedly applying \eqref{eqn:est error of P_hat} in Lemma~\ref{lemma:mismatch of K_hat and P_hat}, we obtain that $\norm{P_{\gamma\ell-j,\CS}^{k}-\hat{P}_{\gamma\ell-j,\CS}^{k}}\le(44\tilde{\Gamma}_{\CS}^9\sigma_R)^j\mu_{\CS}\varepsilon$ for all $j\in[\ell-1]$. Thus, we have shown that \eqref{eqn:bound on est error of P_hat general} also holds for all $t \in\{\gamma\ell-j:\gamma\in[\varphi],j\in[\ell-1]\}$. Combining the above arguments together completes the proof of \eqref{eqn:bound on est error of P_hat general} for all $t\in\{0,\dots,T\}$. The proof of \eqref{eqn:est error of K_hat general} now follows from similar arguments to those for \eqref{eqn:est error of K_hat} in the proof of Lemma~\ref{lemma:mismatch of K_hat and P_hat}.\hfill\qed

%\vspace{-0.3cm}
\subsection*{Proof of Lemma~\ref{lemma:J_hat minus J}}
For notational simplicity, we again drop the dependency of various terms on $k$ and $\CS$ in this proof. First, we let $J(x_t)$ be the cost of using the optimal control gain $K_t$ given in Eq.~\eqref{eqn:control gain}, starting from the state $x_t$, where $x_t$ (i.e., $x_t^{k}$) is the state at time step $t$ when the certainty equivalence control $u_{t^{\prime}}=\hat{K}_{t^{\prime}}x_{t^{\prime}}$ is used for all $t^{\prime}\in\{0,1,\dots,t-1\}$. Therefore, we have from our discussions in Sections~\ref{sec:LQR} that
\begin{equation}
\label{eqn:J(x_k)}
J(x_t)=x_t^{\top}P_tx_t+\sum_{i=t}^{T-1}P_{i+1}W,
\end{equation}
where $P_t$ (i.e., $P_{t,\CS}$) is given by Eq.~\eqref{eqn:recursion for P_k} with $P_T=Q_f$. Denoting $c_t=x_t^{\top}Qx_t+u_t^{\top}Ru_t$, where $u_t=\hat{K}_tx_t$, we can rewrite $\hat{J}_k(\CS)$ defined in Eq.~\eqref{eqn:episode t ce cost} as
\begin{equation*}
\label{eqn:J_hat rewrite}
\hat{J}_k(\CS)=\E\Big[\big(\sum_{t=0}^{T-1}c_t\big)+x_T^{\top}Q_fx_T\Big].
\end{equation*}
It now follows that
\begin{align}\nonumber
&\hat{J}_k(\CS)-J_k(\CS)\\\nonumber
=&\E\Big[\big(\sum_{t=0}^{T-1}c_t+J(x_t)-J(x_t)\big)+x_T^{\top}Q_fx_T\Big]-J(x_0)\\
=&\E\Big[\big(\sum_{t=0}^{T-1}c_t+J(x_{t+1})-J(x_t)\big)\Big].\label{eqn:J_hat minus J telescope}
\end{align}
To obtain Eq.~\eqref{eqn:J_hat minus J telescope}, we use the telescope sum and note from Eq.~\eqref{eqn:J(x_k)} that $J(x_T)=x_T^{\top}Q_fx_T$. Next, considering a single term in the summation on the right hand side of Eq.~\eqref{eqn:J_hat minus J telescope}, we have
\begin{equation*}
\E\Big[c_t+J(x_{t+1})-J(x_t)\Big]=\E\Big[x_t^{\top}(Q+\hat{K}_t^{\top}R\hat{K}_t)x_t+x_{t+1}^{\top}P_{t+1}x_{t+1}-x_t^{\top}P_tx_t-P_{t+1}W\Big].
\end{equation*}
Noting that $x_{t+1}=(A+B\hat{K}_t)x_t+w_t$ and recalling that $w_t$ is a zero-mean white Gaussian noise process, we then obtain
\begin{equation}
\E\Big[c_t+J(x_{t+1})-J(x_t)\Big]=\E\Big[x_t^{\top}(Q+\hat{K}_t^{\top}R\hat{K}_t)x_t+x_t^{\top}\big((A+B\hat{K}_t)^{\top}P_{t+1}(A+B\hat{K}_t)-P_t\big)x_t\Big].\label{eqn:J_hat minus J single}
\end{equation}
Since $P_t$ satisfies the recursion in Eq.~\eqref{eqn:recursion for P_k}, one can use Eq.~\eqref{eqn:control gain} and obtain
\begin{equation*}
\label{eqn:P_k+1 to P_k}
P_t = Q + K_t^{\top}R K_t +(A+BK_t)^{\top}P_{t+1}(A+BK_t).
\end{equation*}
Using similar arguments to those in the proof of \cite[Lemma~10]{fazel2018global}, one can now show that
\begin{equation}
\label{eqn:J_hat minus J single 1}
\E\Big[c_t+J(x_{t+1})-J(x_t)\Big]=\E\Big[x_t^{\top}\Delta K_t^{\top}(R+B^{\top}P_tB^{\top})\Delta K_tx_t\Big],
\end{equation}
where $\Delta K_t=\hat{K}_t-K_t$. It then follows from Eqs.~\eqref{eqn:J_hat minus J telescope} and \eqref{eqn:J_hat minus J single 1} that Eq.~\eqref{eqn:J_hat minus J} holds, completing the proof.\hfill\qed

\subsection*{Proof of Lemma~\ref{lemma:bound on Psi_hat}}
Recall the definition of $\hat{\Psi}_{t_2,t_1}^{k}(\CS)$ (resp., $\Psi_{t_2,t_1}^{k}(\CS)$) in Eq.~\eqref{eqn:Psi} (resp., Eq.~\eqref{eqn:Psi_hat}), and note that 
\begin{equation*}
A+B_{\CS}\hat{K}_{t,\CS}^{k}=A+B_{\CS}K_{t,\CS}^{k}+B_{\CS}(\hat{K}_{t,\CS}^{k}-K_{t,\CS}^{k}),
\end{equation*}
for all $t\in\{0,1,\dots,T-1\}$. The proof now follows from Lemma~\ref{lemma:power of perturbed matrix} in Appendix~\ref{app:tech lemmas}.\hfill\qed

\subsection*{Proof of Proposition~\ref{prop:bound on J_hat minus J}}
First, we provide an upper bound on $\Sigma_t^{k}\triangleq\E[x^{k}_tx^{k\top}_t]$ for all $t\in\{0,1,\dots,T-1\}$. Recalling that we have assumed that $x^{k}_0=0$, it follows that $\Sigma_0^{k}=0$. Considering any $t\in[T-1]$, one can show via Eq.~\eqref{eqn:LTI with B_i} that
\begin{equation*}
\label{eqn:recursion for Sigma_k}
\Sigma_t^{k}=(A+B_{\CS}\hat{K}_{t-1,\CS}^{k})\Sigma_{t-1}^{k}(A+B_{\CS}\hat{K}_{t-1,\CS}^{k})^{\top} + W,
\end{equation*}
which implies that
\begin{equation*}
\label{eqn:expression for Sigma_k}
\Sigma_t^{k}= \sum_{i=0}^{t-1}\hat{\Psi}_{t-1,i}^{k}(\CS)W\hat{\Psi}_{t-1,i}^{k}(\CS)^{\top},
\end{equation*}
where $\hat{\Psi}_{t-1,i}^{k}(\CS)$ is defined in Eq.~\eqref{eqn:Psi_hat}. Now, under the assumption on $\varepsilon$, we can apply the upper bound on $\norm{\hat{\Psi}_{t,i}^{k}(\CS)}$ in Lemma~\ref{lemma:bound on Psi_hat} and obtain
\begin{align}\nonumber
\norm{\Sigma_t^{k}}&\le\sigma_1(W) \zeta_{\CS}^2\sum_{i=1}^{t}(\frac{1+\eta_{\CS}}{2})^{2(t-i)}\\\nonumber
&\le\frac{\sigma_1(W)\zeta_{\CS}^2}{1-(\frac{1+\eta_{\CS}}{2})^2}\le\frac{4\sigma_1(W)\zeta_{\CS}^2}{1-\eta_{\CS}^2}.\label{eqn:upper bound on Sigma_k}
\end{align}
One can also show via Eq.~\eqref{eqn:J_hat minus J} in Lemma~\ref{lemma:J_hat minus J} that 
\begin{equation*}
\label{eqn:upper bound on J_hat minus J pre}
\hat{J}_k(\CS)-J_k(\CS)\le\sum_{t=0}^{T-1}\norm{\Sigma_t^{k}}\norm{R^{k}+B_{\CS}^{\top}P_t^{k}B_{\CS}}\norm{\Delta\hat{K}_t^{k}(\CS)}_F^2.
\end{equation*}
Moreover, under the assumption on $\varepsilon$, we can also apply the upper bound on $\norm{\Delta\hat{K}_t^{k}(\CS)}$ in Proposition~\ref{prop:bound on est error of P_hat}, where $\Delta\hat{K}_t^{k}(\CS)\in\R^{m_{\CS}\times n}$. Also noting that $\norm{\Delta\hat{K}_t^{k}(\CS)}_F^2\le\min\{n,m_{\CS}\}\norm{\Delta\hat{K}_t^{k}(\CS)}^2$, and recalling the definition of $\Gamma_{\CS}$, we can then combine the above arguments together and obtain~\eqref{eqn:upper bound on J_hat minus J}.\hfill\qed

\section{Proofs pertaining to Theorem~\ref{thm:regret of as alg}}\label{app:regret bound proofs}
\subsection*{Proof of Lemma~\ref{lemma:event E high prob}}
First, we know from Assumption~\ref{ass:noise process} that $w_t^{k}\overset{\text{i.i.d.}}{\sim}\CN(0,\sigma^2I)$ for all $t\in\{0,\dots,T-1\}$ and for all $k\in[N]$. One can then apply \cite[Lemma~34]{cassel2020logarithmic} and obtain that $\Prob(\CE_w)\ge1-\delta/8$. Similarly, recalling that $\tilde{w}_t^{k}\overset{\text{i.i.d.}}{\sim}\CN(0,2\sigma^2\eta_0^2I)$ for all $t\in\{0,\dots,T-1\}$ and for all $k\in[N]$, one can apply \cite[Lemma~34]{cassel2020logarithmic} and obtain that $\Prob(\CE_{\tilde{w}})\ge1-\delta/8$.
    
Next, for any $j\in[p]$, we have from Lemma~\ref{lemma:est error of Theta_i hat} that with probability at least $1-\delta/(8p)$,
\begin{equation*}
\Tr(\Delta_{\G_j}^{\top}V_{\G_j}\Delta_{\G_j})\\\le4\sigma^2n\log\Big(\frac{8np\det(V_{\G_j})}{\delta\det(\lambda I)}\Big) + 2\lambda\norm{\Theta_{\G_j}}_F^2.
\end{equation*}
Applying a union bound over all $j\in[p]$, we obtain that $\Prob(\CE_{\Theta})\ge1-\delta/8$.
    
Finally, recalling Eq.~\eqref{eqn:dynamics of x rewrite}, for any $j\in[p]$ and any $N_{j}\le k\le N_{j+1}-1$, we denote a sigma field $\F_{t,j}^{k}=\sigma(x_0^{t},u_{0,\G_j}^{k},\dots,x_t^{k},u_{t,\G_j}^{k})$ for all $t\in\{0,\dots,T-1\}$, where $u_{t,\G_j}^{k}=K_{\G_j}x_t^{k}+\tilde{w}_t^{k}$ with $\tilde{w}_t^{k}\overset{\text{i.i.d.}}{\sim}\CN(0,2\sigma^2\eta_0^2I)$. Note that for any $t\in[T-1]$, $z_{t,\G_j}^{k}=\begin{bmatrix}x_t^{k\top} & u_{t,\G_j}^{k\top}\end{bmatrix}^{\top}$ is conditional Gaussian given $\F_{t-1,j}^{k}$. Moreover, one can use similar arguments to those for \cite[Lemma~34]{cohen2019learning} and show that 
\begin{equation*}
\E[z_{t,\G_j}^{k}z_{t,\G_j}^{k\top}|\F_{t-1,j}^{k}]\succeq\frac{\sigma^2}{2}I,
\end{equation*}
for all $t\in[T-1]$. Now, noting from the choice of $\tau_1$ in Eq.~\eqref{eqn:tau_1} that $\tau_1T\ge200(m+n)\log\frac{96p}{\delta}$, one can apply \cite[Lemma~36]{cassel2020logarithmic} to show that for any $j\in[p]$, the following holds with probability at least $1-\delta/(8p)$:
\begin{equation*}
\sum_{t=N_{j}}^{N_{j+1}-1}\sum_{t=0}^{T-1}z_{t,\G_j}^{k}z_{t,\G_j}^{k\top}\succeq\frac{(T-1)\tau_1\sigma^2}{80}I.
\end{equation*}
Applying a union bound over all $j\in[p]$ yields that $\Prob(\CE_z)\ge1-\delta/8$. 
    
Combining the above arguments together and applying a union bound over $\CE_w$, $\CE_{\tilde{w}}$, $\CE_{\Theta}$, and $\CE_z$, we complete the proof of the lemma.\hfill\qed

\subsection*{Proof of Lemma~\ref{lemma:est error of Theta_Gj(i)}}
Consider any $j\in[p]$. First, under $\CE_{\Theta}$, we have
\begin{align}\nonumber
&\Tr(\Delta_{\G_j}^{\top}V_{\G_j}\Delta_{\G_j})\\\nonumber
\le&4\sigma^2n\log\Big(\frac{8np\det(V_{\G_j})}{\delta\det(\lambda I)}\Big) + 2\lambda\norm{\Theta_{\G_j}}_F^2\\\nonumber
\le&4\sigma^2n\log\Big(\frac{8np\det(V_{\G_j})}{\delta\det(\lambda I)}\Big) + 2\lambda\min\{n,n+m_{\G_j}\}\norm{\Theta_{\G_j}}^2\\
\le&4\sigma^2n\log\Big(\frac{8np\det(V_{\G_j})}{\delta\det(\lambda I)}\Big) + 2\lambda n\vartheta^2,\label{eqn:upper bound on trace term}
\end{align}
where $\Delta_{\G_j}=\hat{\Theta}_{\G_j}-\Theta_{\G_j}$, $m_{\G_j}=\sum_{i\in\G_j}m_i$ and $\vartheta$ is defined in \eqref{eqn:parameters}. Next, under $\CE_z$, we have
\begin{equation}
V_{\G_j}=\lambda I+\sum_{k=N_{j}}^{N_{j+1}-1}\sum_{t=0}^{T-1}z_{t,\G_j}^{k}z_{t,\G_j}^{k\top}\succeq\frac{(T-1)\tau_1\sigma^2}{80}I.\label{eqn:lower bound on V_Gj(i)}
\end{equation}
Combining \eqref{eqn:upper bound on trace term} and \eqref{eqn:lower bound on V_Gj(i)} together and rearranging terms yields
\begin{equation*}
\norm{\Delta_{\G_j}}^2\le\norm{\Delta_{\G_j}}^2_F\le\frac{80}{\tau_1\sigma^2(T-1)}\Big(4\sigma^2n\log\Big(\frac{8np\det(V_{\G_j})}{\delta\det(\lambda I)}\Big) + 2\lambda n\vartheta^2\Big).
\end{equation*}
Next, we aim to provide an upper bound on $\norm{V_{\G_j}}$. We see that
\begin{align}
\norm{V_{\G_j}}\le\lambda+\sum_{k=N_{j}}^{N_{j+1}-1}\sum_{t=0}^{T-1}\norm{z_{t,\G_j}}^2,\label{eqn:upper bound on V_Gj(i)}
\end{align}
where $z_{t,\G_j}^{k}=\begin{bmatrix}x_t^{k\top} & u_{t,\G_j}^{k\top}\end{bmatrix}^{\top}$, with $u_{t,\G_j}^{k}=K_{\G_j}x_t^{k}+\tilde{w}_t^{k}$ and $\tilde{w}_t^{k}\overset{\text{i.i.d.}}{\sim}\CN(0,2\sigma^2\eta_0^2I)$. Noting Eq.~\eqref{eqn:dynamics of x rewrite} and recalling from Assumption~\ref{ass:system parameters} that $\norm{(A+B_{\G_j}K_{\G_j})^t}\le\zeta_0\eta_0^t$ for all $t\in\R_{\ge0}$, where $0<\eta_0<1$, one can use \cite[Lemma~38]{cassel2020logarithmic} and show that
\begin{equation*}
\norm{x_t^{k}}\le\frac{\zeta_0}{1-\eta_0}\max_{t^{\prime}\in\{0,\dots,T-1\}}\norm{B_{\CS^{k}}\tilde{w}_{t^{\prime}}^{k}+w_{t^{\prime}}^{k}},
\end{equation*}
for all $k\in\tilde{\K}$ and all $t\in\{0,\dots,T\}$, where $\tilde{\K}$ is defined in~\eqref{eqn:T_tilde}. Thus, under $\CE$, we have
\begin{align}\nonumber
\norm{x_t^{k}}&\le\frac{\zeta_0}{1-\eta_0}\Big(\vartheta\eta_0\sigma\sqrt{10m\log\frac{8TN}{\delta}}+\sigma\sqrt{5n\log\frac{8TN}{\delta}}\Big)\\\nonumber
&\le\frac{\zeta_0}{1-\eta_0}(\vartheta\eta_0\sigma\sqrt{10m}+\sigma\sqrt{5n})\sqrt{\log\frac{8TN}{\delta}}\\
&\le\frac{\zeta_0}{1-\eta_0}\sqrt{20\vartheta^2\eta_0^2\sigma^2m+10\sigma^2n}\sqrt{\log\frac{8TN}{\delta}}.\label{eqn:upper bound on x}
\end{align}
Since $\norm{z_{t,\G_j}^{k}}\le\norm{x_t^{k}}+\norm{u_{t,\G_j}^{k}}\le(1+\eta_0)\norm{x_t^{k}}+\norm{\tilde{w}_t^{k}}$, one can combine the above arguments and show that under $\CE$,
\begin{align}
\norm{z_{t,\G_j}^{k}}&\le\frac{\zeta_0(1+\eta_0)}{1-\eta_0}\sqrt{20\vartheta^2\eta_0^2\sigma^2m+10\sigma^2n}\sqrt{\log\frac{8TN}{\delta}}+\eta_0\sigma\sqrt{10m\log\frac{8TN}{\delta}}\le\sqrt{z_b}.\label{eqn:upper bound on z}
\end{align}
Plugging \eqref{eqn:upper bound on z} into \eqref{eqn:upper bound on V_Gj(i)}, we obtain
\begin{align}\nonumber
\norm{V_{\G_j}}\le\lambda+\tau_1Tz_b,  
\end{align}
which implies that
\begin{align}\nonumber
\log\frac{8np\det(V_{\G_j})}{\delta\det(\lambda I)}&\le\log\Big(\frac{8pn(\lambda+\tau_1 Tz_b)}{\delta\lambda}\Big)^{m_{\G_j}+n}\\\nonumber
&=(m_{\G_j}+n)\log\Big(\frac{8np}{\delta}+\frac{8np\tau_1 Tz_b}{\lambda}\Big)\\\nonumber
&\le(m+n)\log\Big(\frac{8np}{\delta}+\frac{8nN Tz_b}{\lambda}\Big),
\end{align}
where the second inequality follows from the fact that $p\tau_1\le N$. Combining the above arguments together, one can show via the choice of $\tau_1$ in Eq.~\eqref{eqn:tau_1} and algebraic manipulations that $\norm{\Delta\G_j}^2\le\min\{\varepsilon_0^2/p,\frac{\tau_0}{\sqrt{N}}\}$, which completes the proof of the lemma.\hfill\qed

\subsection*{Proof of Lemma~\ref{lemma:upper bound on R_e^1}}
First, from the definition of Algorithm~\ref{alg:online for AS}, we know that $R_e^1$ satisfies that 
\begin{align}\nonumber
R_e^1= \sum_{k\in\tilde{\K}}J_k(\CS^k,u_{\CS^k}^{k})-\sum_{k\in\tilde{\K}}J_k(\CS_{\star}^k)
\le\sum_{k\in\tilde{\K}}J_k(\CS^k,u_{\CS^k}^{k}).
\end{align}
Considering any $k\in\tilde{\K}$ and noting lines~5-7 in Algorithm~\ref{alg:online for AS}, we have from Eqs.~\eqref{eqn:episode t cost} and~\eqref{eqn:dynamics of x rewrite} that 
\begin{equation*}
J_k(\CS^k,u^{k}_{\CS^k})=\Big(\sum_{t=0}^{T-1}x_t^{k\top}Q^{k} x^{k}_t+u_{t,\CS^k}^{k\top}R_{\CS^k}^{k}u^{k}_{t,\CS^k}\Big)+x_T^{k\top}Q_f^{k} x_T^{k}.
\end{equation*}
Thus, we have 
\begin{equation*}
R_e^1\le\max\{\sigma_Q,\sigma_R\}\bigg(\Big(\sum_{k\in\tilde{\K}}\sum_{t=0}^{T-1}x_t^{k\top}x_t^{k}\\+u_{t,\CS^k}^{k\top}u_{t,\CS^k}^{k}\Big)+x_T^{k\top}x_T^{k}\bigg).
\end{equation*}
Recall that $u_{t,\G_j}^{k}=K_{\G_j}x_t^{k}+\tilde{w}_t^{k}$ for all $t\in\{0,\dots,T-1\}$, where $\tilde{w}_t^{k}\overset{\text{i.i.d.}}{\sim}\CN(0,2\sigma^2\eta_0^2I)$. It follows that under the event $\CE$ defined in Eq.~\eqref{eqn:event E},
\begin{align}\nonumber
\norm{u_{t,\G_j}^{k}}\le\eta_0\norm{x_t^{k}}+\sqrt{10m\log\frac{8TN}{\delta}},
\end{align}
which implies that 
\begin{align}\nonumber
\norm{u_{t,\G_j}^{k}}^2\le2\eta_0^2\norm{x_t^{k}}^2+20m\log\frac{8TN}{\delta}.
\end{align}
Moreover, we see from \eqref{eqn:T_tilde} and the definition of Algorithm~\ref{alg:online for AS} that $|\tilde{\K}|=\tau_1p$. Now, recalling the upper bound on $\norm{x_t^{k}}$ for all $k\in\tilde{\K}$ and all $t\in\{0,\dots,T\}$ given by~\eqref{eqn:upper bound on x} in the proof of Lemma~\ref{lemma:est error of Theta_Gj(i)}, one can show that under $\CE$, 
\begin{equation}
R_e^1\le\max\{\sigma_Q,\sigma_R\}\frac{\tau_1p(2\eta_0^2+3)\zeta_0^2T}{(1-\eta_0)^2}\big(20\vartheta^2\eta_0^2\sigma^2m+10\sigma^2n\big)\log\frac{8TN}{\delta}.
\end{equation}
\hfill\qed

\subsection*{Proof of Lemma~\ref{lemma:upper bound on R_e^2}}
Consider any episode $k\in\K$ in Algorithm~\ref{alg:online for AS}. Noting that $x_0^k=0$, one can show that the state of system~\eqref{eqn:LTI with B_i} corresponding to $\CS^k$ selected in line~12 of Algorithm~\ref{alg:online for AS} satisfies
\begin{align}\nonumber
x_{t+1}^{k} = \sum_{i=0}^{t}\hat{\Psi}_{t,i}^{k}(\CS^k)w_{i}^{k},
\end{align}
where $\hat{\Psi}_{t,i}^{k}(\CS^k)$ is defined in Eq.~\eqref{eqn:Psi_hat}. Moreover, supposing that the event $\CE$ holds, we know from Lemma~\ref{lemma:est error of Theta_Gj(i)} that $\norm{\hat{\Theta}_{\G_j}-\Theta_{\G_j}}\le\varepsilon_0/\sqrt{p}$ for all $j\in[p]$. It follows that $\hat{A}$ and $\hat{B}$ obtained in line~9 of Algorithm~\ref{alg:online for AS} satisfy that $\norm{\hat{A}-A}\le\varepsilon_0$ and $\norm{\hat{B}-B}\le\varepsilon_0$, which also implies that $\norm{\hat{B}_{\CS^k}-B_{\CS^k}}\le\varepsilon_0$, where $\hat{B}_{\CS^k}$ contains the columns of $\hat{B}$ that correspond to $\CS^k$. Now, one can obtain from the choice of $\varepsilon_0$ in \eqref{eqn:parameters} and Proposition~\ref{prop:bound on est error of P_hat} that 
\begin{equation*}
\norm{\hat{K}_{t,\CS^k}^{k}-K_{t,\CS^k}^{k}}\le\frac{1-\eta_{\CS^k}}{2\norm{B_{\CS^k}}\zeta_{\CS^k}},\ \forall t\in\{0,\dots,T-1\},
\end{equation*}
which also implies that 
\begin{equation*}
\norm{\hat{K}_{t,\CS^k}^{k}}\le\kappa,\ \forall t\in\{0,\dots,T-1\},
\end{equation*}
where $\hat{K}_{t,\CS^k}^{k}$ and $K_{t,\CS^k}^{k}$ are given by Eqs.~\eqref{eqn:control gain S_t est} and \eqref{eqn:control gain}, respectively. We have from Lemma~\ref{lemma:bound on Psi_hat} that
\begin{align*}
\norm{\hat{\Psi}_{t_2,t_1}^{k}(\CS^k)}&\le\zeta_{\CS^k}(\frac{1+\eta_{\CS^k}}{2})^{t_2-t_1},
\end{align*}
for all $t_1,t_2\in\{0,\dots,T-1\}$ with $t_2\ge t_1$, where we know from Lemma~\ref{lemma:bound on norm of Phi} that $0<(1+\eta_{\CS^k})/2<1$. One can now use similar arguments to those for \cite[Lemma~38]{cassel2020logarithmic} and show that
\begin{equation*}
\norm{x_t^{k}}\le\frac{2\zeta_{\CS^k}}{1-\eta_{\CS^k}}\max_{k^{\prime}\in\K,t^{\prime}\in\{0,\dots,T-1\}}\norm{w_{t^{\prime}}^{k^{\prime}}}.
\end{equation*}
Thus, under the event $\CE$ defined in Eq.~\eqref{eqn:event E}, we have that 
\begin{align}\nonumber
\norm{x_t^{k}}\le\frac{2\zeta\sigma}{1-\eta}\sqrt{5n\log\frac{8TN}{\delta}},
\end{align}
for all $k\in\K$ and all $t\in\{0,\dots,T\}$. Furthermore, we recall from Eq.~\eqref{eqn:episode t cost} that 
\begin{align}\nonumber
&J_k(\CS^k,u_{\CS}^{k})\\\nonumber
=&\Big(\sum_{t=0}^{T-1}x_t^{k\top}Q^{k} x^{k}_t+u_{t,\CS^k}^{k\top}R_{\CS^k}^{k}u^{k}_{t,\CS^k}\Big)+x_T^{k\top}Q_f^{k} x_T^{k}\\\nonumber
=&\Big(\sum_{t=0}^{T-1}x_t^{k\top}(Q^{k}+\hat{K}_{t,\CS}^{k\top}R_{\CS^k}^{k}\hat{K}_{t,\CS^k}^{k})x_t^{k}\Big)+x_T^{k\top}Q_f^{k}x_T^{k},
\end{align}
where we use the fact that $u_{t,\CS^k}^{k}=\hat{K}_{t,\CS^k}^{k}x_t^{k}$. It follows from our above arguments that under the event $\CE$,
\begin{align}
J_k(\CS^k,u_{\CS^k}^{k})\le T(2\sigma_Q+\kappa^2\sigma_R)\frac{4\zeta^2\sigma^2}{(1-\eta)^2}5n\log\frac{8TN}{\delta}\le\bar{y}_b,\label{eqn:upper bound on cost under event E}
\end{align}

To proceed, recall that we use the {\bf Exp3.S} algorithm in Algorithm~\ref{alg:online for AS} to select $\CS^k$ for all $k\in \K$. As we argued in Section~\ref{sec:algorithm design for episodic setting}, each action in the {\bf Exp3.S} algorithm corresponds to a set $\CS\subseteq\G$ with $|\G|=H$, i.e., the set of all possible actions $\Q$ in the {\bf Exp3.S} algorithm is given by $\Q=\{\CS\subseteq\G:|\CS|=H\}$. Moreover, the cost of the action corresponding to $\CS^k$ in episode $k\in[N]$ is given by $J_k(\CS^k,u_{\CS^k}^{k})$. Thus, we can replace $y_b$ (resp., $y_a$) in \eqref{eqn:regret for Exp3} with $\bar{y}_b$ (resp., $0$), and obtain that under the event $\CE$,
\begin{align}\nonumber
R_e^2&=\E_{\A_e}\Big[\sum_{k\in\K}J_k(\CS^k,u_{\CS^k}^{k})\Big]-\sum_{k\in\K}J_k(\CS^k_{\star},u_{\CS^k_{\star}}^{k})\\\nonumber
&\le \bar{y}_b2\sqrt{e-1}\sqrt{|\Q|N(h(\CS_{\star})\log(|\Q|N)+e)}.
\end{align}
\hfill\qed

\subsection*{Proof of Lemma~\ref{lemma:upper bound on R_e^3}}
First, consider any $k\in\K$. Similarly to our arguments in the proof of Lemma~\ref{lemma:upper bound on R_e^2}, we see that $J_k(\CS_{\star}^k,u_{\CS_{\star}^k}^{k})$ is given by
\begin{equation*}
J_k(\CS^k_{\star},u_{\CS^k_{\star}}^{k})=\Big(\sum_{t=0}^{T-1}x_t^{k\top}(Q^{k}+\hat{K}_{t,\CS^k_{\star}}^{k\top}R_{\CS^k_{\star}}^{k}\hat{K}_{t,\CS^k_{\star}}^{k})x_t^{k}\Big)\\+x_T^{k\top}Q_f^{k}x_T^{k},
\end{equation*}
where $x_0^{k}=0$. Applying Eqs.~\eqref{eqn:LTI with B_i} and \eqref{eqn:recursion for P_k(S_t) tilde}, one can show that
\begin{align}\nonumber
&J_k(\CS^k_{\star},u_{\CS^k_{\star}}^{k})\\\nonumber
=&\Big(\sum_{t=0}^{T-1}x_t^{k\top}\tilde{P}_{t,\CS^k_{\star}}^{k}x_t^{k}-(x_{t+1}^{k}-w_t^{k})^{\top}\tilde{P}_{t+1,\CS^k_{\star}}^{k}(x_{t+1}^{k}-w_t^{k})\Big)+x_T^{k\top}Q_f^{k}x_T^{k}\\\nonumber
=&\Big(\sum_{t=0}^{T-1}x_t^{k\top}\tilde{P}_{t,\CS^k_{\star}}^{k}x_t^{k}-x_{t+1}^{k\top}\tilde{P}_{t+1,\CS^k_{\star}}^{k}x_{t+1}^{k}+2w_t^{k\top}\tilde{P}_{t+1,\CS^k_{\star}}^{k}(A+B_{\CS^k_{\star}})x_t^{k}+w_t^{k\top}\tilde{P}_{t+1,\CS^k_{\star}}^{k}w_t^{k}\Big)+x_T^{k\top}Q_f^{k}x_T^{k}\\\nonumber
=&\Big(\sum_{t=0}^{T-1}2w_t^{k\top}\tilde{P}_{t+1,\CS^k_{\star}}^{k}(A+B_{\CS^k_{\star}})x_t^{k}+w_t^{k\top}\tilde{P}_{t+1,\CS^k_{\star}}^{k}w_t^{k}\Big),
\end{align}
where we note that $\tilde{P}_{T,\CS^k_{\star}}^{k}=Q_f^{k}$. Recalling the definition of $R_e^3$, we see that 
\begin{align}\nonumber
R_e^3&=\Big(\sum_{k\in\K}\big(\sum_{t=0}^{T-1}2w_t^{k\top}\tilde{P}_{t+1,\CS^k_{\star}}^{k}(A+B_{\CS^k_{\star}})x_t^{k}w_t^{k\top}\tilde{P}_{t+1,\CS^k_{\star}}^{k}w_t^{k}\big)-\hat{J}_k(\CS^k_{\star})\Big)\\
&=\sum_{k\in\K}\big(\sum_{t=0}^{T-1}2w_t^{k\top}\tilde{P}_{t+1,\CS^k_{\star}}^{k}(A+B_{\CS^k_{\star}})x_t^{k}+w_t^{k\top}\tilde{P}_{t+1,\CS^k_{\star}}^{k}w_t^{k}-\sigma^2\Tr(\tilde{P}_{t+1,\CS^k_{\star}}^{k})\big).\label{eqn:decompose of R_3}
\end{align}
Now, for any $k\in\K$, one can apply Eq.~\eqref{eqn:recursion for P_k(S_t) tilde} recursively to show that 
\begin{align*}
&\tilde{P}_{t,\CS^k_{\star}}^{k}=\Big(\sum_{i=t}^{T-1}\hat{\Psi}_{i,t}^{k\top}(\CS^k_{\star})(Q^{k}+\hat{K}_{i,\CS^k_{\star}}^{k\top}R_{\CS^k_{\star}}^{k}\hat{K}_{t,\CS^k_{\star}}^{k})\hat{\Psi}_{i,t}^{k}(\CS^k_{\star})\Big)+\hat{\Psi}_{N,t}^{k\top}(\CS^k_{\star})\tilde{P}_{T,\CS^k_{\star}}^{k}(\CS^k_{\star})\hat{\Psi}_{T,t}^{k}(\CS^k_{\star}),
\end{align*}
for all $t\in\{0,\dots,T-1\}$, where $\hat{\Psi}_{i,t}^{k}(\CS^k_{\star})$ is defined in Eq.~\eqref{eqn:Psi_hat}. Next, suppose that the event $\CE$ defined in Eq.~\eqref{eqn:event E} holds. Similarly to our arguments in the proof of Lemma~\ref{lemma:upper bound on R_e^2}, we know that with the choice of $\varepsilon_0$ in \eqref{eqn:parameters}, 
\begin{equation*}
\norm{\hat{K}_{t,\CS^k_{\star}}^{k}}\le\kappa,\ \forall t\{0,\dots,T-1\},
\end{equation*}
where $\kappa$ is defined in \eqref{eqn:parameters}, and $\hat{K}_{t,\CS^{k}_{\star}}^{k}$ is given by Eq.~\eqref{eqn:control gain S_t est}. We also have from Lemma~\ref{lemma:bound on Psi_hat} that
\begin{align*}
\norm{\hat{\Psi}_{t_2,t_1}^{k}(\CS^k_{\star})}&\le\zeta_{\CS^k_{\star}}(\frac{1+\eta_{\CS^k_{\star}}}{2})^{t_2-t_1},
\end{align*}
for all $t_1,t_2\in\{0,\dots,T-1\}$ with $t_2\ge t_1$, where $\zeta_{\CS^k_{\star}},\eta_{\CS^k_{\star}}$ are given by Lemma~\ref{lemma:bound on norm of Phi} with $0<(1+\eta_{\CS^k_{\star}})/2<1$. For any $t\in\{0,\dots,T-1\}$, one can then show that 
\begin{align}\nonumber
\norm{\tilde{P}_{t,\CS^k_{\star}}^{k}}&\le(\sigma_Q+\sigma_R\kappa^2)\zeta_{\CS^k_{\star}}^2\sum_{i=0}^{T-t}(\frac{1+\eta_{\CS^k_{\star}}}{2})^{2i}\\\nonumber
&\le(\sigma_Q+\sigma_R\kappa^2)\frac{4\zeta_{\CS^k_{\star}}^2}{1-\eta_{\CS^k_{\star}}^2}\le(\sigma_Q+\sigma_R\kappa^2)\frac{4\zeta^2}{1-\eta^2}.
\end{align} 
Furthermore, we recall from our arguments in the proof of Lemma~\ref{lemma:upper bound on R_e^2} that under the event $\CE$ defined in Eq.~\eqref{eqn:event E}, 
\begin{align}\nonumber
\norm{x_t^{k}}\le\frac{2\zeta\sigma}{1-\eta}\sqrt{5n\log\frac{8TN}{\delta}},
\end{align}
for all $T\in\{0,\dots,T-1\}$. 

To proceed, let us denote 
\begin{equation*}
V_{t,\CS^k_{\star}}^{k}=\tilde{P}_{t+1,\CS^k_{\star}}^{k}(A+B_{\CS^k_{\star}})x_t^{k}.
\end{equation*}
From our arguments above, we see that under $\CE$,
\begin{align}\nonumber
\norm{V_{t,\CS^k_{\star}}^{k}}&\le\norm{\tilde{P}_{t+1,\CS^k_{\star}}^{k}}\norm{A+B_{\CS^k_{\star}}}\norm{x_t^{k}}\\\nonumber
&\le\frac{16\sigma(\sigma_Q+\sigma_R\kappa^2)\vartheta\zeta^3}{(1-\eta^2)(1-\eta)}\sqrt{5n\log\frac{8TN}{\delta}},
\end{align}
for all $t\in\{0,\dots,T-1\}$ and all $k\in\K$, where $\vartheta$ is defined in \eqref{eqn:parameters}, which implies that 
\begin{equation*}
\sum_{k\in\K}\sum_{t=0}^{T-1}\norm{V_{t,\CS^k_{\star}}^{k}}^2\le TN\frac{256\sigma^2(\sigma_Q+\sigma_R\kappa^2)^2\vartheta^2\zeta^6}{(1-\eta^2)^2(1-\eta)^2} 5n\log\frac{8TN}{\delta}.
\end{equation*}
Noting from Assumption~\ref{ass:noise process} that $w_t^{k}\overset{\text{i.i.d.}}{\sim}\CN(0,\sigma^2I)$ for all $t\in\{0,\dots,T-1\}$ and for all $k\in[N]$, one can now apply \cite[Lemma~30]{cohen2019learning} and obtain that under the event $\CE$, the following holds with probability at least $1-\delta/4$:
\begin{equation}
2\sum_{k\in\K}\sum_{t=0}^{T-1}w_t^{k\top}V_{t,\CS^k_{\star}}^{k}\le64\sigma\sqrt{TN}\frac{\sigma(\sigma_Q+\sigma_R\kappa^2)\vartheta\zeta^3}{(1-\eta^2)(1-\eta)}\\\times\sqrt{5n\log\frac{8TN}{\delta}}.\label{eqn:R_3,1}
\end{equation}

Moreover, based on our arguments above, one can apply \cite[Lemma~31]{cohen2019learning} and obtain that under the event $\CE$, the following holds with probability at least $1-\delta/4$:
\begin{equation}
\sum_{k\in\K}\big(\sum_{t=1}^{T-1}w_t^{k\top}\tilde{P}_{t+1,\CS^k_{\star}}^{k}w_t^{k}\big)-\sigma^2\Tr(\tilde{P}_{t+1,\CS^k_{\star}}^{k})\big)\le8(\sigma_Q+\sigma_R\kappa^2)\frac{4\zeta^2}{1-\eta^2}\sigma^2\sqrt{TN(\log\frac{16TN}{\delta})^3}.\label{eqn:R_3,2}
\end{equation}
Recalling the decomposition of $R_e^3$ in \eqref{eqn:decompose of R_3}, we can apply~\eqref{eqn:R_3,1}-\eqref{eqn:R_3,2} together with a union bound, and obtain an upper bound on $R_e^3$ that holds with probability at least $1-\delta/2$ under $\CE$.\hfill\qed

\subsection*{Proof of Lemma~\ref{lemma:upper bound on R_e^4}}
As we argued in the proof of Lemma~\ref{lemma:upper bound on R_e^2}, under the event $\CE$ defined in Eq.~\eqref{eqn:event E}, $\hat{A}$ and $\hat{B}$ obtained in line~9 of Algorithm~\ref{alg:online for AS} satisfy that $\norm{\hat{A}-A}\le\sqrt{\frac{\tau_0}{\sqrt{N}}}$ and $\norm{\hat{B}-B}\le\sqrt{\frac{\tau_0}{\sqrt{N}}}$, which also implies that $\norm{\hat{B}_{\CS}-B_{\CS}}\le\sqrt{\frac{\tau_0}{\sqrt{N}}}$ for all $\CS\subseteq\G$ with $|\CS|=H$. Under the event $\CE$, one can then show via the choice of $\varepsilon_0$ in \eqref{eqn:parameters} and Proposition~\ref{prop:bound on J_hat minus J} that for any $k\in\K$,
\begin{align}\nonumber
&\hat{J}_k(\CS^k_{\star})-J_k(\CS^k_{\star})\\\nonumber
&\le \frac{4\max_{\CS\subseteq\G,|\CS|=H}\{n,m_{\CS}\}T\zeta^2\tau_0}{(1-\eta^2)\sqrt{N}}\sigma(\sigma_R+\Gamma^3)\big(3\tilde{\Gamma}^6(20\tilde{\Gamma}\sigma_R)^{\ell-1}32\ell^{\frac{5}{2}}\tilde{\beta}^{2(\ell-1)}(1+\nu^{-1})\max\{\sigma_Q,\sigma_R\}\big)^2.
\end{align}
Recalling from \eqref{eqn:T} that $|\K|=N-\tau_1p$, we complete the proof of the lemma.\hfill\qed

\section{Proofs pertaining to Theorem~\ref{thm:regret of as alg non-episodic}}\label{app:regret bound proofs non-episodic}
\subsection*{Proof of Lemma~\ref{lemma:est error of Theta_Gj(i) non-episodic}}
First, recall from Assumption~\ref{ass:system parameters} that for any $j\in[p]$, $\norm{(A+B_{\G_j}K_{\G_j})^t}\le\zeta_0\eta_0^t$ for all $t\in\R_{\ge0}$, where $0<\eta_0<1$. From the definition of Algorithm~\ref{alg:online for AS non-episodic}, we see that $\CS_{(k-1){T^{\prime}}}=\cdots=\CS_{kT^{\prime}-1}=\G_j$ for all $k\in\{N^{\prime}_j,\dots,N^{\prime}_{j+1}-1\}$. Since $T^{\prime}>\frac{2}{1-\eta}\log2\zeta$ (from the hypothesis in Theorem~\ref{thm:regret of as alg non-episodic}), and $x_0^k=x_{T^{\prime}}^{k-1}$ with $x_0^1=0$, one can now apply \cite[Lemma~39]{cassel2020logarithmic} and show that 
\begin{equation*}
\norm{x_t^k}\le\frac{3\zeta_0^2}{1-\eta_0}\max_{k^{\prime}\in\tilde{\K},t^{\prime}\in\{0,\dots,T^{\prime}-1\}}\norm{B_{\CS^{k^{\prime}}}\tilde{w}_{t^{\prime}}^{k^{\prime}}+w_{t^{\prime}}^{k^{\prime}}},
\end{equation*}
for all $k\in\tilde{\K}^{\prime}$ and all $t\in\{0,\dots,T^{\prime}\}$. Using similar arguments to those in the proof of Lemma~\ref{lemma:est error of Theta_Gj(i)}, we have that under the event $\CE^{\prime}$, 
\begin{equation*}
\label{eqn:upper bound on state norm 1}
\norm{x_t^k}\le \frac{3\zeta^2_0}{1-\eta_0}\sqrt{20\vartheta^2\eta_0^2\sigma^2m+10\sigma^2n}\sqrt{\log\frac{8T}{\delta}},
\end{equation*}
for all $k\in\tilde{\K}^{\prime}$ and all $t\in\{0,\dots,T^{\prime}\}$. Noting the choice of $\tau_1^{\prime}$, the proof now follows that of Lemma~\ref{lemma:est error of Theta_Gj(i)}.\hfill\qed

\subsection*{Proof of Lemma~\ref{lemma:norm bound on state}}
Recall from Assumption~\ref{ass:system parameters} that for any $j\in[p]$, $\norm{(A+B_{\G_j}K_{\G_j})^t}\le\zeta_0\eta_0^t\le\zeta\eta^t$ for all $t\in\R_{\ge0}$, where $\zeta,\eta$ are defined in \eqref{eqn:parameters} with $0<\eta<1$. Moreover, by the choice of $\varepsilon_0$ given in \eqref{eqn:parameters}, we have from our arguments in the proof of Lemma~\ref{lemma:upper bound on R_e^2} that 
\begin{align*}
\norm{\hat{\Psi}_{t_2,t_1}^{k}(\CS^k)}&\le\zeta_{\CS^k}(\frac{1+\eta_{\CS^k}}{2})^{t_2-t_1}\le\zeta(\frac{1+\eta}{2})^{t_2-t_1},
\end{align*}
for all $k\in\K^{\prime}$ and all $t_1,t_2\in\{0,\dots,T^{\prime}\}$ with $t_2\ge t_1$. Similarly to the proof of Lemma~\ref{lemma:est error of Theta_Gj(i) non-episodic}, since $T^{\prime}>\frac{2}{1-\eta}\log2\zeta$ and $x_0^k=x_{T^{\prime}}^{k-1}$ with $x_0^1=0$, one can apply \cite[Lemma~39]{cassel2020logarithmic} and show that 
\begin{equation*}
\norm{x_t^k}\le\frac{6\zeta^2}{1-\eta}\max\Big\{\max_{k^{\prime}\in\tilde{\K}^{\prime},t^{\prime}\in\{0,\dots,T^{\prime}-1\}}\norm{B_{\CS^{k^{\prime}}}\tilde{w}_{t^{\prime}}^{k^{\prime}}+w_{t^{\prime}}^{k^{\prime}}},\max_{k^{\prime}\in\K^{\prime},t^{\prime}\in\{0,\dots,T^{\prime}-1\}}\norm{ w_{t^{\prime}}^{k^{\prime}}}\Big\},
\end{equation*}
for all $k\in[N^{\prime}]$ and all $t\in\{0,\dots,T^{\prime}\}$.
It then follows that under the event $\CE^{\prime}$, $\norm{x_t^k}\le x_b$ for all $k\in[N^{\prime}]$ and all $t\in\{0,\dots,T^{\prime}\}$.\hfill\qed

\subsection*{Proof of Lemma~\ref{lemma:almost episodic}}
Let $x_t^k$ (resp., $\tilde{x}_t^k$) denote the state of system~\eqref{eqn:LTI} when the sequence of actuators $\CS^1,\dots,\CS^{k_0}$ (resp., $\tilde{\CS}^1,\dots,\tilde{\CS}^{k_0}$) is selected and the control $u_{t,\CS^k}^k$ (resp., $u_{t,\tilde{\CS}^k}^k$) is applied for any $k\in[k_0]$ and any $t\in\{0,\dots,T^{\prime}-1\}$. We see from Lemma~\ref{lemma:norm bound on state} that $\norm{x_{t}^{k}}\le x_b$ and $\norm{\tilde{x}_t^{k}}\le x_b$ for any $k\in[k_0]$ and any $t\in\{0,\dots,T^{\prime}-1\}$. Now, let us consider any $t\in\{(k_0-1)T^{\prime}+T_m,\dots,k_0T^{\prime}-1\}$ with $t<T-1$. Using Eq.~\eqref{eqn:LTI with B_i} and noting the definitions of $\CS^1,\dots,\CS^{N^{\prime}}$ and $\tilde{\CS}^1,\dots,\tilde{\CS}^{N^{\prime}}$, one can then show that 
\begin{align}\nonumber
x_{t} &= \hat{\Psi}_{t,0}^k(\CS^{k_0})x_0^{k_0}+ \sum_{i=0}^{t-1}\hat{\Psi}_{t-1,i}^{k_0}(\CS^{k_0})w_{i}^{k_0},\\\nonumber
\tilde{x}_{t} &= \hat{\Psi}_{t,0}^{k_0}(\CS^{k_0})\tilde{x}_0^{k_0}+ \sum_{i=0}^{t-1}\hat{\Psi}_{t-1,i}^{k_0}(\CS^k)w_{i}^{k_0},
\end{align}
where $\hat{\Psi}_{t-1,i}^{k_0}(\CS^{k_0})$ is given by Eq.~\eqref{eqn:Psi_hat}. It follows that 
\begin{align}\nonumber
\norm{x_t-\tilde{x}_t}&\le\norm{\hat{\Psi}_{t,0}^{k_0}(\CS^{k_0})}\norm{x_0^{k_0}-\tilde{x}_0^{k_0}}\\
&\le 2\zeta(\frac{1+\eta}{2})^{T_m}x_b\le2x_bT^{-1/3},\label{eqn:upper bound on state diff}
\end{align}
where the second inequality follows from Lemma~\ref{lemma:bound on Psi_hat} with $\zeta,\eta$ defined in \eqref{eqn:parameters} and $0<\eta<1$. To obtain the third inequality in \eqref{eqn:upper bound on state diff}, we first note that $1-\frac{1+\eta}{2}<-\log\frac{1+\eta}{2}$, and then use
\begin{align*}
&T_m=\frac{2}{1-\eta}(\frac{1}{3}\log T+\log\zeta)\ge\frac{\log T^{-1/3}-\log\zeta}{\log\frac{1+\eta}{2}},\\
\Rightarrow& \zeta(\frac{1+\eta}{2})^{T_m}\le T^{-1/3}.
\end{align*}
To proceed, we have from Eq.~\eqref{eqn:cost per time step} that 
\begin{align*}
c_t(\CS_{0:t},u_{\CS_{0:t}})=x_t^{\top}(Q+\hat{K}_{t_0,\CS^{k_0}}^{k_0\top}R_{\CS^{k_0}}\hat{K}_{t_0,\CS^{k_0}}^{k_0})x_t,\\
c_t(\tilde{\CS}_{0:t},u_{\tilde{\CS}_{0:t}})=\tilde{x}_t^{\top}(Q+\hat{K}_{t_0,\CS^{k_0}}^{k_0\top}R_{\CS^{k_0}}\hat{K}_{t_0,\CS^{k_0}}^{k_0})\tilde{x}_t,
\end{align*}
where $t_0=t-(k_0-1)T^{\prime}$ and $\norm{\hat{K}_{t_0,\CS^{k_0}}^{k_0}}\le\kappa$ as we argued in the proof of Lemma~\ref{lemma:upper bound on R_e^2}. Thus, we have
\begin{align*}
&|c_t(\CS_{0:t},u_{\CS_{0:t}})-c_t(\tilde{\CS}_{0:t},u_{\tilde{\CS}_{0:t}})|\\
\le& 2\max\{\norm{x_t},\norm{\tilde{x}_t}\}\norm{Q+\hat{K}_{t_0,\CS^{k_0}}^{k_0\top}R\hat{K}_{t_0,\CS^{k_0}}^{k_0}}\norm{x_t-\tilde{x}_t}\\
\le&4(\sigma_Q+\kappa^2\sigma_R)x_b^2T^{-1/3},
\end{align*}
where we use the fact that $c_t(\cdot)$ is Lipschitz continuous with respect to $x_t,\tilde{x}_t$.

Next, let us consider $t=T-1$. Recalling Eq.~\eqref{eqn:final time step cost} and following similar arguments to those above, one can show that 
\begin{align*}
&|c_t(\CS_{0:t},u_{\CS_{0:t}})-c_t(\tilde{\CS}_{0:t},u_{\tilde{\CS}_{0:t}})|\\
\le& 2\max\{\norm{x_t},\norm{\tilde{x}_t}\}\norm{Q+\hat{K}_{t_0,\CS^{k_0}}^{k_0\top}R_{\CS^{k_0}}\hat{K}_{t_0,\CS^{k_0}}^{k_0}}\norm{x_t-\tilde{x}_t}+2\max\{\norm{x_{t+1}},\norm{\tilde e{x}_{t+1}}\}\norm{Q_f}\norm{x_{t+1}-\tilde{x}_{t+1}}\\
\le&4(2\sigma_Q+\kappa^2\sigma_R)x_b^2T^{-1/3},
\end{align*}
completing the proof of the lemma.\hfill\qed

\subsection*{Proof of Lemma~\ref{lemma:upper bound on R_1^c}}
The proof follows the proof of Lemma~\ref{lemma:upper bound on R_e^1} by noting the upper bound $\norm{x_t^k}\le x_b$ for all $k\in[N^{\prime}]$ and all $t\in\{0,\dots,T^{\prime}\}$ provided in Lemma~\ref{lemma:norm bound on state}.\hfill\qed

\subsection*{Proof of Lemma~\ref{lemma:upper bound on R_c^2}}
Suppose the event $\CE^{\prime}$ holds. First, following our arguments in the proof of Lemma~\ref{lemma:almost episodic} and recalling Eqs.~\eqref{eqn:cost per time step}-\eqref{eqn:final time step cost}, one can show that 
\begin{equation}
0\le c(\CS_{0:t},u_{\CS_{0:t}})\le(2\sigma_Q+\kappa^2\sigma_R)x_b^2\le \bar{y}_b^{\prime},
\end{equation}
for any $t\ge(k-1)T^{\prime}$ with $k\in\K^{\prime}$, and any $\CS_{0:t}$ with the control $u_{\CS_{0:t}}$ described in Lemma~\ref{lemma:norm bound on state}. Now, one can apply the proof techniques for \cite[Theorem~2\&4]{arora2012online} together with Lemma~\ref{lemma:almost episodic}, and show that
\begin{align*}
R_c^2&\le (2\sigma_Q+\kappa^2\sigma_R)x_b^2\big(T^{\prime}R(\frac{T}{T^{\prime}})+\frac{TT_m}{T^{\prime}}\big)+8(\sigma_Q+\kappa^2\sigma_R)x_b^2T^{-1/3}T\\
&\le8(\sigma_Q+\kappa^2\sigma_R)x_b^2T^{2/3}+(2\sigma_Q+\kappa^2\sigma_R)x_b^2(T_m+1)\big(4(e-1)(h(\CS_{\star})\ln(|\Q|T)+e)|\Q|\big)^{1/3}T^{2/3},
\end{align*}
where $R(\frac{T}{T^{\prime}})\triangleq2\sqrt{e-1}\sqrt{|\Q|\frac{T}{T^{\prime}}(h(\CS_{\star})\ln(|\Q|\frac{T}{T^{\prime}})+e)}$ and $(2\sigma_Q+\eta_0^2\sigma_R)x_b^2R(\frac{T}{T^{\prime}})$ is the regret of the {\bf Exp3.S} subroutine in Algorithm~\ref{alg:online for AS large scale}.\hfill\qed

\subsection*{Proof of Lemma~\ref{lemma:upper bound on R_c^3}}
First, we note from Eq.~\eqref{eqn:exp for J_t hat} that
\begin{align*}
\hat{J}_k(\CS_{\star}^k)&=\E\big[x_0^{k\top}\tilde{P}_{0,\CS_{\star}^k}^kx_0^k\big]+\sigma^2\sum_{t=0}^{T^{\prime}-1}\Tr(\tilde{P}_{t+1,\CS_{\star}^k}^k)\\
&\ge\sum_{t=0}^{T^{\prime}-1}\Tr(\tilde{P}_{t+1,\CS_{\star}^k}^kW),
\end{align*}
where $\tilde{P}^k_{t,\CS_{\star}^k}$ is given by Eq.~\eqref{eqn:recursion for P_k(S_t) tilde}. Next, following the arguments in the proof of Lemma~\ref{lemma:upper bound on R_e^3} and invoking Lemma~\ref{lemma:norm bound on state}, one can show that \eqref{eqn:upper bound on R_c^3} holds with probability at least $1-\delta/2$ under $\CE^{\prime}$.\hfill\qed

\subsection*{Proof of Lemma~\ref{lemma:upper bound on R_c^4}}
First, observe that $J(\CS_{\star})$ defined in \eqref{eqn:opt LQR cost S_t} can be written into the episodic form $J_k(\CS_{\star}^k)$ based on the episodic instance of Problem~\eqref{eqn:LQR obj} that we constructed. In particular, we have from \eqref{eqn:opt LQR cost S_t} that
\begin{equation*}
J(\CS^{\star})=\sum_{k\in[N^{\prime}]}J_k(\CS_{\star}^k)\ge\sum_{k\in\K^{\prime}}J_k(\CS_{\star}^k).
\end{equation*}
We then have
\begin{align*}
R_c^4&\le \sum_{k\in\K^{\prime}}\big(\hat{J}_k(\CS_{\star}^k)-J_k(\CS_{\star}^k)\big)\le\sum_{k\in\K^{\prime}}\Big(\E\big[x_0^{k\top}\tilde{P}_{0,\CS_{\star}^k}^kx_0^k\big]+\sigma^2\sum_{t=0}^{T^{\prime}-1}\Tr(\tilde{P}_{t+1,\CS_{\star}^k}^k)-\sigma^2\sum_{t=0}^{T^{\prime}-1}\Tr(P_{t+1,\CS_{\star}^k}^k)\Big),
\end{align*}
where $\tilde{P}_{t,\CS_{\star}^k}^k$ (resp., $P_{t,\CS_{\star}^k}^k$) is given by Eq.~\eqref{eqn:recursion for P_k(S_t) tilde} (resp., Eq.~\eqref{eqn:recursion for P_k}). Note that under the event $\CE^{\prime}$, we have from Lemma~\ref{lemma:norm bound on state} that $\norm{x_0^k}\le x_b$ for all $k\in\K^{\prime}$. We have also shown in the proof of Lemma~\ref{lemma:upper bound on R_e^3} that $\norm{\tilde{P}_{0,\CS_{\star}^k}^k}\le(\sigma_Q+\sigma_R\kappa^2)\frac{4\zeta^2}{1-\eta^2}$. It follows that under the event $\CE^{\prime}$,
\begin{equation*}
\E\big[x_0^{k\top}\tilde{P}_{0,\CS_{\star}^k}^kx_0^k\big]\le(\sigma_Q+\sigma_R\kappa^2)\frac{4\zeta^2x_b^2}{1-\eta^2}.
\end{equation*}
Moreover, following similar arguments to those in the proof of Lemma~\ref{lemma:upper bound on R_e^4}, one can show that
\begin{align}\nonumber
&\sigma^2\sum_{k\in\K^{\prime}}\sum_{t=0}^{T^{\prime}}\big(\Tr(\tilde{P}_{t+1,\CS_{\star}^k}^k)-\Tr(P_{t+1,\CS_{\star}^k}^k)\big)\\\nonumber
&\le \frac{4\max_{\CS\subseteq\G,|\CS|=H}\{n,m_{\CS}\}T^{\prime}\zeta^2\tau_0^{\prime}\sqrt{N^{\prime}}}{(1-\eta^2)}\sigma(\sigma_R+\Gamma^3)\big(3\tilde{\Gamma}^6(20\tilde{\Gamma}\sigma_R)^{\ell-1}32\ell^{\frac{5}{2}}\tilde{\beta}^{2(\ell-1)}(1+\nu^{-1})\max\{\sigma_Q,\sigma_R\}\big)^2.
\end{align}
Combining the above arguments completes the proof.\hfill\qed

\section{Proofs for large-scale instances}\label{app:proofs for large-scale instances}
\subsection*{Proof of Proposition~\ref{prop:regret of A_e large}}
Following similar arguments to the proof of Theorem~\ref{thm:regret of as alg} in Section~\ref{sec:proof of thm 1}, the regret $R_{\A_l}$ of Algorithm~\ref{alg:online for AS large scale} can be decomposed into $R_{\A_l}=R_{\A_l}^1+R_{\A_l}^2+R_{\A_l}^3+R_{\A_l}^4$ with
\begin{align}\nonumber
&R_{\A_l}^1 = (1-e^{-c_g})\big(\sum_{k\in\tilde{\K}}J_k(\emptyset)-J_k(\CS_{\star}^k)\big)-\E_{\A_l}\Big[\sum_{k\in\tilde{K}}J(\emptyset)-J_k(\CS^k,u_{\CS^k}^k)\Big],\\
&R_{\A_l}^2 =(1-e^{-c_g})\big(\sum_{k\in\K}J_k(\emptyset)-J_k(\CS^{k}_{\star},u_{\CS^{k}_{\star}}^{k})\big)-\E_{\A_l}\Big[\sum_{k\in\K}J_k(\emptyset)-J_k(\CS^k,u_{\CS^k}^{k})\Big],\label{eqn:R_Al^2}\\\nonumber
&R_{\A_l}^3 = (1-e^{-c_g})R_e^3,\\\nonumber
&R_{\A_l}^4 = (1-e^{-c_g})R_e^4,
\end{align}
where $R_{\A_l}^1$ corresponds to the system identification phase in Algorithm~\ref{alg:online for AS large scale}, $R_{\A_l}^3,R_{\A_l}^4$ correspond to the certainty equivalent control subroutine and $R_{\A_l}^2$ corresponds to the regret of the {\bf Exp3.S} subroutines $M_1,\dots,M_H$. Suppose the event $\CE$ defined in \eqref{eqn:event E} holds. Following similar arguments to those for Lemmas~\ref{lemma:upper bound on R_e^1}, \ref{lemma:upper bound on R_e^3} and \ref{lemma:upper bound on R_e^4}, one can show that $R_{\A_l}^1+R_{\A_l}^3+R_{\A_l}^4=\tilde{O}(n(m+n)^2p^2T\sqrt{N})$. We then focus on upper bounding $R_{\A_l}^2$ in the remaining of this proof.

To proceed, for any $\bar{\CS}\subseteq\bar{\G}$ with $\bar{\G}$ defined in Eq.~\eqref{eqn:augmented G} and any $k\in\K$, we define
\begin{equation}\label{eqn:def of f_k}
f_k(\bar{\CS}) =J_k(\emptyset)- J_k(\bar{\CS}^k,u_{\bar{\CS}^k}^k),
\end{equation}
where $\bar{\CS}^k=\{\bar{s}^k\in\bar{s}:\bar{s}\in\bar{\CS}\}$ with $\bar{s}^k$ denoting the $k$th element of the tuple $\bar{s}\in\bar{\CS}$, and $u_{t,\bar{\CS}^k}^k=\hat{K}_{t,\bar{\CS}^k}^kx_t^k$ for all $t\in\{0,\dots,T-1\}$ with $\hat{K}_{t,\bar{\CS}^k}^k$ obtained via Eq.~\eqref{eqn:control gain S_t est} using $\hat{A},\hat{B}$ from Algorithm~\ref{alg:online for AS large scale}. For any $\bar{\CS}\subseteq\bar{\G}$, we define $\bar{f}(\bar{\CS})=\sum_{k\in\K} f_k(\bar{\CS})$.  %Note that we can then view the actions selected by any $M_j$ from episodes $k=N_{p+1}$ to $N$ as a single action, denoted as $\bar{s}_j=(s_j^{N_{p+1}},\dots,s_j^K)$. 
For any $j\in[H]$, we further denote $\bar{\CS}^{\prime}_{j}=\{\bar{s}_1,\dots,\bar{s}_j\}$ with $\bar{\CS}_0=\emptyset$, where $\bar{s}_j=(s_j^{N_{p+1}},\dots,s_j^K)$ contains the actuators selected by $M_j$ from episodes $k=N_{p+1}$ to $N$. Based on the above notations, we see from Eq.~\eqref{eqn:cost of arms in large-scale} that the regret of any subroutine $M_j$ in Algorithm~\ref{alg:online for AS large scale} can be written as
\begin{equation*}
r_j = \max_{\bar{s}\in\bar{\G}}\big\{\bar{f}(\bar{\CS}^{\prime}_{j-1}\cup\bar{s})-\bar{f}(\bar{\CS}^{\prime}_{j-1})\big\}-\big(\bar{f}(\bar{\CS}^{\prime}_{j-1}\cup\bar{s}_j)-\bar{f}(\bar{\CS}^{\prime}_{j-1})\big).
\end{equation*}
Following similar arguments to those in the proof of Lemmas~\ref{lemma:upper bound on R_e^3}-\ref{lemma:upper bound on R_e^4}, one can also show that 
\begin{equation*}
\big|\bar{f}(\bar{\CS})-\bar{g}(\bar{\CS})\big|=\tilde{O}(n(n+m)^2T\sqrt{N}),
\end{equation*}
for all $\bar{\CS}\subseteq\bar{\G}$, where $\bar{g}(\bar{\CS})=\sum_{k\in\K}g_k(\bar{\CS}_k)$ with $g_k(\cdot)$ defined in Eq.~\eqref{eqn:g_k normalized}. Now, one can view $\bar{\CS}^{\prime}_H$ selected by $M_1,\dots,M_H$ as a greedy solution to the optimization problem $\max_{\bar{\CS}\subseteq\bar{\G},|\bar{\CS}|\le H}\bar{g}(\bar{\CS})$, where there is an error in evaluating the objective function $\bar{g}(\cdot)$ and the $j$th greedy choice is made with an additive error $r_j$ for all $j\in[H]$. %\footnote{Given a set function $f:2^{\G}\to\R_{\ge0}$ and $H\in\Z_{\ge1}$, the greedy algorithm for $\max_{\CS\subseteq\G,|\CS|\le H}f(\CS)$ initializes $\CS_g=\emptyset$ and iteratively adds $s_{\star}\in\arg\max_{s\in\G}(f(\CS_g\cup\{s\})-f(\CS_g))$ to $\CS_g$ until $|\CS_g|=H$.} 
Using similar arguments to those for \cite[Proposition~7]{chamon2021approximately} and \cite[Lemma~7]{ye2020distributed}, one can show that 
\begin{equation*}
(1-e^{-c_g})\max_{\bar{\CS}\subseteq\G,|\bar{\CS}|=H}\big\{\bar{f}(\bar{\CS})\big\}-\bar{f}(\bar{\CS}^{\prime}_H)\le\tilde{O}(n(n+m)^2T\sqrt{N})+\sum_{j=1}^Hr_j,
\end{equation*}
To proceed, suppose the event $\CE$ defined in Eq.~\eqref{eqn:event E} holds. Following similar arguments to those for \eqref{eqn:upper bound on cost under event E} in the proof of Lemma~\ref{lemma:upper bound on R_e^2}, one can also show that $y_{j,s}^k\in[-\bar{y}_b,\bar{y}_b]$ for all $j\in[H]$, all $s\in\G$ and all $k\in\K$, where $\bar{y}_b=\tilde{O}(nT)$ is defined as Eq.~\eqref{eqn:y_b bar}. Denote $N_e=|\{k:k\in\{N_{p+1},\dots,N\},b_k=1\}|$, where we note that $\E_{\A_l}[N_e]=\rho N$ with $N_e\le N$. Taking the expectation $\E_{\A_l}[\cdot]$ on both sides of the above inequality, and recalling Eq.~\eqref{eqn:R_Al^2}, one can then show via the definition of Algorithm~\ref{alg:online for AS large scale} that
\begin{align}\nonumber
R_{\A_l}^2&=(1-e^{-c^{\prime}})\sum_{k\in\K}\big(J_k(\emptyset)-J_k(\CS^{k}_{\star},u_{\CS^{k}_{\star}}^{k})\big)-\E_{\A_l}\Big[\sum_{k\in\K}J_k(\emptyset)-f_k(\bar{\CS}^{\prime}_H)-J_k(\CS^k,u_{\CS^k}^k)+f_k(\bar{\CS}^{\prime}_H)\Big]\\\nonumber
&\le\tilde{O}(n(n+m)^2T\sqrt{N})+\E_{\A_l}\Big[\sum_{j=1}^Hr_j\Big]+\E_{\A_l}[N_e]\tilde{O}(nT).
\end{align} 

Thus, it remains to bound $\E_{\A_l}[r_j]$ for all $j\in[H]$. We see from Eq.~\eqref{eqn:y_hat} and the definition of Algorithm~\ref{alg:online for AS large scale} that for any $j\in[H]$, any $s\in\G$ and any $k\in\K$,
\begin{align}
\E_{\A_l}[\hat{y}_{j,s}^k]=\frac{\rho}{|\G|H}y_{j,s}^k+\frac{\rho}{|\G|H}J_k(\CS^{\prime k}_{j-1},u_{\CS^{\prime k}_{j-1}}^k).
\end{align}
Also, as we argued above, we have $y_{j,s}^k\in[-\bar{y}_b,\bar{y}_b]$ for all $j\in[H]$, all $s\in\G$ and all $k\in\K$. Following similar arguments to those for \cite[Lemma~5\&Theorem~13]{streeter2008online}, one can now show via Lemma~\ref{thm:regret for exp3.s} that for any $j\in[H]$,
\begin{align}\nonumber
\E_{\A_l}[r_j]&\le\frac{|\G|H}{\rho}\tilde{O}(nT)\E_{\A_l}\Big[\sqrt{|\G|N_e(h(\CS_{\star})\ln(|\G|N_e)+e)}\Big]\\\nonumber
&\le\frac{|\G|H}{\rho}\tilde{O}(nT)\sqrt{\E_{\A_l}\big[|\G|N_e(h(\CS_{\star})\ln(|\G|N_e)+e)\big]}\\\nonumber
&\le\frac{|\G|H}{\rho}\tilde{O}(nT)\sqrt{|\G|\E_{\A_l}[N_e](h(\CS_{\star})\ln(|\G|N)+e)}\\\nonumber
&\le|\G|H\tilde{O}(nT)\sqrt{\frac{|\G|}{\rho}N(h(\CS_{\star})\ln(|\G|N)+e)},
\end{align}
where we use the Jensen's inequality to obtain the second inequality. Combining the above arguments and noting the choice of $\rho$ in Eq.~\eqref{eqn:rho}, we obtain that 
\begin{align*}
R_{\A_l}^2=\tilde{O}(n(m+n)^2T|\G|^{3/2}H^2h(\CS_{\star})^{1/2}N^{2/3}),
\end{align*}
which together with the upper bound on $R_{\A_l}^1+R_{\A_l}^3+R_{\A_l}^4$ argued above complete the proof of the proposition.\hfill\qed

\subsection*{Proof of Proposition~\ref{prop:regret of A_e large non-episodic}}
Following similar arguments to the proofs of Theorem~\ref{thm:regret of as alg non-episodic} and Proposition~\ref{prop:regret of A_e large}, one can decompose the regret $R_{\A_l^{\prime}}$ of Algorithm~\ref{alg:online for AS large scale} (in the non-episodic setting) can be into $R_{\A_l^{\prime}}=R_{\A_l^{\prime}}^1+R_{\A_l^{\prime}}^2$, where $R_{\A_l^{\prime}}^1$ corresponds to the system identification phase and the certainty equivalent control subroutine in Algorithm~\ref{alg:online for AS large scale}, and $R_{\A_l^{\prime}}$ corresponds to {\bf Exp3.S} subroutines in Algorithm~\ref{alg:online for AS large scale}.  In particular, we have
\begin{align}
R_{\A_l^{\prime}}^2 = (1-e^{-c_g})\big(\sum_{k\in\K^{\prime}}J_k(\emptyset)-J_k(\CS^{k}_{\star},u_{\CS^{k}_{\star}}^{k})\big)-\E_{\A_l}\Big[\sum_{k\in\K^{\prime}}J_k(\emptyset)-J_k(\CS^k,u_{\CS^k}^{k})\Big],
\end{align}
where we see from \eqref{eqn:LQR obj 2nd} that $J(\emptyset)$ can be written into the episodic form based on the episodic instance of Problem~\eqref{eqn:LQR obj} constructed in Section~\ref{sec:alg for non-episodic setting}, i.e., $J(\emptyset)=\sum_{k\in[N^{\prime}]}J_k(\emptyset)$. Now, following similar arguments to those for Lemmas~\ref{lemma:upper bound on R_1^c}, \ref{lemma:upper bound on R_c^3} and \ref{lemma:upper bound on R_c^4}, one can show that $R_{\A_l^{\prime}}^1=\tilde{O}(n(m+n)^2\rho T^{\prime}\sqrt{T/T^{\prime}})$ (under the event $\CE^{\prime}$ defined as \eqref{eqn:event E} with $N^{\prime},T^{\prime}$). We then upper bound $R_{\A_l}^2$ using similar arguments to those for Lemma~\ref{lemma:upper bound on R_c^2}. Hence, it remains to upper bound $R_{\A_l^{\prime}}^2$. Following similar arguments to those in the proof of Proposition~\ref{prop:regret of A_e large}, one can show that 
\begin{align}\nonumber
R_{\A_l^{\prime}}^2&=(1-e^{-c^{\prime}})\sum_{k\in\K^{\prime}}\big(J_k(\emptyset)-J_k(\CS^{k}_{\star},u_{\CS^{k}_{\star}}^{k})\big)-\E_{\A_l}\Big[\sum_{k\in\K^{\prime}}J_k(\emptyset)-f_k(\bar{\CS}^{\prime}_H)-J_k(\CS^k,u_{\CS^k}^k)+f_k(\bar{\CS}^{\prime}_H)\Big]\\\nonumber
&\le\tilde{O}(n(n+m)^2T^{\prime}\sqrt{\frac{T}{T^{\prime}}})+\E_{\A_l}\Big[\sum_{j=1}^Hr_j\Big]+\frac{T}{T^{\prime}}\rho\tilde{O}(nT^{\prime}).
\end{align} 
Moreover, following similar arguments to those in the proofs of Proposition~\ref{prop:regret of A_e large}, Lemmas~\ref{lemma:almost episodic} and \ref{lemma:upper bound on R_c^2}, one can bound
\begin{align}
\E_{\A_l^{\prime}}[r_j]\le|\G|H\tilde{O}(nT^{\prime})\sqrt{\frac{|\G|}{\rho}\frac{T}{T^{\prime}}(h(\CS_{\star})\ln(|\G|\frac{T}{T^{\prime}}))}+\tilde{O}(\frac{T}{T^{\prime}}T_m+T^{-1/4}T).
\end{align}
Combining the above arguments together and noting the choices of $T^{\prime}$, $T_m$ and $\rho$, we obtain that 
\begin{align*}
R_{\A_l}^2=\tilde{O}(n(m+n)^2|\G|^{3/2}H^2h(\CS_{\star})^{1/2}T^{3/4}),
\end{align*}
which together with the upper bound on $R_{\A_l^{\prime}}^1$ complete the proof of the proposition.\hfill\qed

\section{Technical lemmas}\label{app:tech lemmas}
\begin{lemma}
\label{lemma:power of perturbed matrix}
Consider a sequence of matrices $X_0,X_1,\dots$, where $X_k\in\R^{n\times n}$ for all $k\ge0$, and a sequence of matrices $\Delta_0,\Delta_1,\dots$, where $\Delta_k\in\R^{n\times n}$ and $\norm{\Delta_i}\le\varepsilon$ for all $k\ge0$. Suppose that there exist $\zeta\in\R_{>0}$ and $\eta\in\R_{>0}$ such that
\begin{equation*}
\label{eqn:bound on M's}
\norm{X_{k_2-1}X_{k_2-2}\cdots X_{k_1}}\le\zeta\eta^{k_2-k_1},
\end{equation*}
for all $k_1,k_2\in\Z_{\ge0}$ with $k_2>k_1$. Then, the following holds:
\begin{equation*}
\label{eqn:power of perturbed matrix}
(X_{k_2-1}+\Delta_{k_2-1})(X_{k_2-2}+\Delta_{k_2-2})\cdots\\\times (X_{k_1}+\Delta_{k_1})\le\zeta(\zeta\varepsilon+\eta)^{k_2-k_1},
\end{equation*}
for all $k_1,k_2\in\Z_{\ge0}$ with $k_2>k_1$.
\end{lemma}
\begin{proof}
First, one can expand the left hand side of \eqref{eqn:power of perturbed matrix} into $2^{k_2-k_1}$ terms. For all $r\in\{0,1,\dots,k_2-k_1\}$ and for all $s\in\{1,2,\dots,{k_2-k_1\choose r}\}$, let $G_{r,s}$ denote a term in the expansion whose degree of $\Delta_i$ is $r$ and whose degree of $X_i$ is $k-r$, where ${k_2-k_1 \choose r}$ is the number of terms in the expansion with the degree of $\Delta_i$ to be $r$. For instance, the term $X_{k_2-1}\Delta_{k_2-2}X_{k_2-3}\Delta_{k_2-4}X_{k_2-5}\cdots X_{k_1}$ may be denoted as $G_{2,s}$ for $s\in\{1,2,\dots,{k_2-k_1 \choose 2}\}$. Under the above notation, the left hand side of \eqref{eqn:power of perturbed matrix} can be written as $\sum_{r=0}^k\sum_{s=1}^{k\choose r}G_{r,s}$.  From \eqref{eqn:bound on M's}, we also note that $\norm{G_{r,s}}\le\zeta^{r+1}\eta^{k_1-k_2-r}\varepsilon^r$. This is because the $\Delta_i$'s in the term $G_{r,s}$ split the $X$'s in $G_{r,s}$ into at most $r+1$ disjoint groups, and $\norm{\Delta_i}\le\varepsilon$ for all $i\in\Z_{\ge0}$. For instance, we see that $X_{k_2-1}\Delta_{k_2-2}X_{k_2-3}\Delta_{k_2-4}X_{k_2-5}\cdots X_{k_1}\le\zeta^3\eta^{k_2-k_1-2}\varepsilon^2$. The rest of the proof then follows from the proof of \cite[Lemma~5]{mania2019certainty}.
\end{proof}

\begin{lemma}
\label{lemma:bound on P}
Consider the following DARE for all $t\in\{0,\dots,T-1\}$:
\begin{equation}
P_{t} = A^{\top}P_{t+1}A-A^{\top}P_{t+1}B(B^{\top}P_{t+1}B+R)^{-1}B^{\top}P_{t+1}A+Q,
\end{equation}
initialized with $P_T=Q_f$, where $Q,Q_f\in\BS_{+}^n$ and $R\in\BS_{++}^m$. Suppose there exists a stabilizing $K\in\R^{m\times n}$ for the pair $(A,B)$ with $\tilde{\zeta}\in\R_{\ge1}$ and $\tilde{\eta}\in\R,0<\tilde{\eta}<1$ such that $\norm{(A+BK)^k}\le \tilde{\zeta}\tilde{\eta}^k$ for all $k\ge0$. Then $\norm{P_t}\le\tilde{\zeta}^2\frac{\sigma_{\max}}{1-\tilde{\eta}^2}$ for all $t\in\{0,\dots,T\}$, where $\sigma_{\max}\triangleq\max\{\norm{Q+K^{\top}RK},\norm{Q_f}\}$.
\end{lemma}
\begin{proof}
From \cite[Chpater~3]{bertsekas2017dynamic}, we know that for any $x\in\R^n$ and any $t\in\{0,\dots,T-1\}$, the following holds:
\begin{equation}\label{eqn:optimal of P}
x^{\top}P_{t}x=\min_{u\in\R^m}\big[x^{\top}Qx+u^{\top}Ru+(Ax+Bu)^{\top}P_{t+1}(Ax+Bu)\big].
\end{equation}
Consider another DARE for all $t\in\{0,\dots,T-1\}$:
\begin{equation*}
\tilde{P}_t=Q+K^{\top}RK+(A+BK)^{\top}\tilde{P}_{t+1}(A+BK),
\end{equation*}
initialized with $\tilde{P}_T=Q_f$. Now, by the optimality of $u$ in Eq.~\eqref{eqn:optimal of P}, considering $u=Kx$ and $t=T-1$ gives 
\begin{equation*}
x^{\top}P_{T-1}x\le\big[x^{\top}Qx+u^{\top}Ru+(Ax+BKx)^{\top}P_{T}(Ax+BKx)\big]=x^{\top}\tilde{P}_{T-1}x,
\end{equation*}
for all $x\in\R^n$, which implies that $P_{T-1}\preceq \tilde{P}_{T-1}$. Since Eq.~\eqref{eqn:optimal of P} holds for all $t\in\{0,\dots,T-1\}$, one can repeat the above arguments and obtain $P_{t}\preceq\tilde{P}_{t}$ for all $t\in\{0,\dots,T-1\}$, which also implies that $\norm{P_t}\le\norm{\tilde{P}_t}$. Using the above recursions for $\tilde{P}_t$, $\tilde{\zeta},\tilde{\eta}$ given by the assumption of the lemma, and some algebra complete the proof of the lemma.
\end{proof}


%%%%%%%%%%%%%%%%%%%%%%%%%%%%%%%%%%%%%%%%%%%%%%%%%%%%%%%%%%%%%%%%%%%%%%%%%%%%%%%%

%%%%%%%%%%%%%%%%%%%%%%%%%%%%%%%%%%%%%%%%%%%%%%%%%%%%%%%%%%%%%%%%%%%%%%%%%%%%%%%%

%%%%%%%%%%%%%%%%%%%%%%%%%%%%%%%%%%%%%%%%%%%%%%%%%%%%%%%%%%%%%%%%%%%%%%%%%%%%%%%%

\end{document}